\documentclass[12pt,a4paper,english,reqno,a4paper]{amsart}
\pdfoutput=1
\usepackage[T1]{fontenc}
\usepackage[utf8]{inputenc}
\usepackage{color}
\usepackage{verbatim}
\usepackage{amstext}
\usepackage{amsthm}
\usepackage{stmaryrd}
\usepackage{setspace}
\makeatletter

\numberwithin{equation}{section}
\numberwithin{figure}{section}
\theoremstyle{plain}
\newtheorem{thm}{\protect\theoremname}[section]
  \theoremstyle{remark}
  \newtheorem{rem}[thm]{\protect\remarkname}
  \theoremstyle{plain}
  \newtheorem{lem}[thm]{\protect\lemmaname}
  \theoremstyle{definition}
  \newtheorem{defn}[thm]{\protect\definitionname}

\usepackage{amsfonts}\usepackage{amsthm}
\usepackage{epsfig}\usepackage{mathrsfs}
\usepackage{amstext}
\usepackage{esint}
\usepackage[T1]{fontenc}
\usepackage[utf8]{inputenc}

\@ifundefined{definecolor}
 {\usepackage{color}}{}

\makeatletter\newcommand{\Rmnum}[1]{\expandafter\@slowromancap\romannumeral#1@}\makeatother

\numberwithin{equation}{section}

\allowdisplaybreaks

\newcommand{\norm}[1]{\left\Vert#1\right\Vert}\newcommand{\abs}[1]{\left\vert#1\right\vert}\newcommand{\set}[1]{\left\{#1\right\}}

\newcommand{\defs}{:=}\newcommand{\sTo}{\rightarrow}

\newcommand{\me}{\mathrm{e}}\newcommand{\dif}{\mathrm{d}}

\DeclareSymbolFont{lettersA}{U}{pxmia}{m}{it}\DeclareMathSymbol{\piup}{\mathord}{lettersA}{"19}

\newcommand{\Real}{\mathbb R}

\newcommand{\mr}[1]{\mathrm{#1}}

\newcommand{\mca}{\mathcal{A}}

\newcommand{\mcc}{\mathcal{C}}
\newcommand{\C}{\mathcal{C}}

\newcommand{\mcf}{\mathcal{F}}
\newcommand{\mcg}{\mathcal{G}}

\newcommand{\mci}{\mathcal{I}}
\newcommand{\mcj}{\mathcal{J}}

\newcommand{\mco}{\mathcal{O}}
\newcommand{\mcp}{\mathcal{P}}

\newcommand{\fd}{\mathscr{D}}

\newcommand{\fk}{\mathscr{K}}

\newcommand{\ft}{\mathscr{T}}

\newcommand{\mbnu}{\boldsymbol\nu}

\makeatother

\makeatother

\usepackage{babel}
  \providecommand{\definitionname}{Definition}
  \providecommand{\lemmaname}{Lemma}
  \providecommand{\remarkname}{Remark}
\providecommand{\theoremname}{Theorem}

\begin{document}

\title[Admissible Locations of Shock Fronts in a Finite Nozzle]{On Admissible Locations of Transonic Shock Fronts for Steady Euler
Flows in an Almost Flat Finite Nozzle with Prescribed Receiver Pressure}

\author{Beixiang Fang}

\author{Zhouping Xin}

\address{B.X. Fang: School of Mathematical Sciences, and MOE-LSC, Shanghai Jiao Tong University, Shanghai 200240, China }

\email{\texttt{bxfang@sjtu.edu.cn}}

\address{Z.P. Xin: The Institute of Mathematical Sciences, The Chinese University
of Hong Kong, Shatin, N.T., Hong Kong}

\email{\texttt{zpxin@ims.cuhk.edu.hk}}


\date{\today}
\begin{abstract}
This paper concerns the existence of transonic shock solutions to the 2-D steady compressible Euler system in an almost flat finite nozzle ( in the sense that it is a generic small perturbation of a flat one ), under physical boundary conditions proposed by Courant-Friedrichs in \cite{CourantFriedrichs1948}, in which the receiver pressure is prescribed at the exit of the nozzle. In the resulting free boundary problem, the location of the shock-front is one of the most desirable information one would like to determine. 
However, the location of the normal shock-front in a flat nozzle can be anywhere in the nozzle so that it provides little information on the possible location of the shock-front when the nozzle's boundary is perturbed. So one of the key difficulties in looking for transonic shock solutions is to determine the shock-front. 
To this end, a free boundary problem for the linearized Euler system will be proposed, whose solution will be taken as an initial approximation for the transonic shock solution. Once an initial approximation is obtained, a further nonlinear iteration could be constructed and proved to lead to a transonic shock solution.
In this paper, a sufficient condition in terms of the geometry of the nozzle and the given exit pressure is derived which yields the existence of the solutions to the proposed free boundary problem.
By this condition, it will be shown that, for the proposed linear free boundary problem, if the nozzle is either strictly expanding or contracting, there exists only one solution as long as the given pressure at the exit lies between certain interval; while if the nozzle is generic with both expanding and contracting portions, there may exist more than one solutions for the same given receiver pressure, which implies the existence of more than one initial approximating shock-fronts.
This non-uniqueness of the initial approximations will lead to the non-uniqueness of the transonic shock solutions for generic nozzles with the same given receiver pressure, which shows the instability of the unperturbed normal shock solution under general perturbations of the nozzle boundary and how this instability may behave.

\end{abstract}

\maketitle
\tableofcontents{}

\section{Introduction}

This paper concerns with the 2-D steady compressible Euler
flows in a finite nozzle with variable sections. The main goal is to find piece-wise smooth weak
solutions with a single shock-front, which are called shock solutions,
to the steady Euler system for an almost flat finite nozzle under
physical boundary conditions proposed by Courant-Friedrichs in \cite{CourantFriedrichs1948}
(see Figure \ref{fig:Nozzle_waved}):
\begin{enumerate}
\item The fluid enters the nozzle with a supersonic state $\overline{U}_{-}$;
\item At the exit of the nozzle, the fluid reaches a given pressure $P_{e}$,
called the receiver pressure, and its value is relatively high so
that the flow is subsonic; 
\item The fluid cannot penetrate the nozzle walls.
\end{enumerate}
Mathematically, this problem can be formulated as a free boundary
problem for the 2-D steady Euler system. The location of the shock-front,
which is the very free boundary, is one of the most
desirable information to be determined. 
\begin{figure}[th]
\centering
\def\svgwidth{200pt}
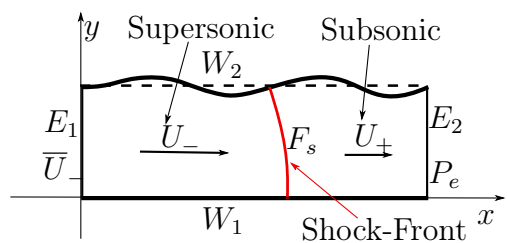

\caption{The steady Euler flows with a single shock-front in an almost flat
nozzle.\label{fig:Nozzle_waved}}
\end{figure}
In this paper, our goal is to find
the admissible locations of the shock-front together with the flow
fields both ahead of and behind it satisfying the above-mentioned boundary conditions for an almost flat nozzle in the
sense that it is a generic perturbation of a flat one. One of the key difficulties in solving this problem is to determine the location of the transonic shock by the geometry of the nozzle and the exit boundary condition.
It should be noted that the boundary condition of given receiver pressure at the exit proposed by Courant-Friedrichs in \cite{CourantFriedrichs1948} is a physically important one, and it is relatively easy to solve the transonic shock problem if this condition is modified \cite{ChenChenSong2006JDE,DuanWeng2011JDE}.
There has been some substantial progress in studying such transonic shock problem. Indeed, in
\begin{figure}[th]
	\centering
	\def\svgwidth{350pt}
	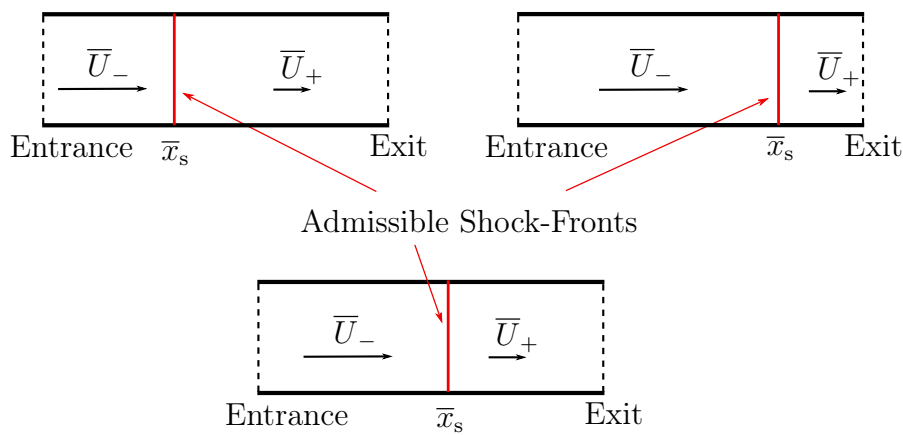
	
	\caption{Infinite admissible normal shock-fronts in a flat nozzle.\label{fig:normal_shocks}}
\end{figure}
case the nozzle is an expanding angular sector or a diverging cone, under the assumption
that the flow parameters only depend on the radius, Courant-Friedrichs
established in \cite{CourantFriedrichs1948} the unique existence
of the transonic shock solutions if the value of the receiver pressure
lies in a certain interval. 
Moreover, Courant-Friedrichs' transonic shock solutions have been shown to be structurally stable for generic small perturbations of both the data and the geometry of the nozzle by Li-Xin-Yin recently in \cite{LiXinYin2009CMP,LiXinYin2009MRL,LiXinYin2013ARMA}
And this structural stability theory of Li-Xin-Yin in \cite{LiXinYin2009CMP,LiXinYin2009MRL,LiXinYin2013ARMA} further confirms the well-posedness of the transonic shock problem formulated in this way.
However, for flat nozzles, it is well-known
that, for a given uniform supersonic state $\overline{U}_{-}$ and a suitable constant receiver pressure, there exists
a unique uniform subsonic state $\overline{U}_{+}$, which connects
with $\overline{U}_{-}$ via a normal transonic shock, while the
location of the shock front can be arbitrary in the nozzle ( see Figure
\ref{fig:normal_shocks}). This fact indicates that the location of
the shock-front cannot be uniquely determined by the pressure at the
exit and makes it difficult to treat the transonic shock problem in an almost flat nozzle as a small perturbation of an unperturbed flat one. 
Thus, a natural question is that, for a general nozzle,
is it possible to determine the location of the shock-front with
the given receiver pressure? 
In this paper, for an almost flat nozzle, we find a natural sufficient condition in terms of the geometry of the nozzle walls and the receiver pressure so that the existence of transonic
shock solutions can be established.
It will be shown that there may exist
more than one shock solutions for the same receiver pressure, depending
on the geometry of the nozzle walls. Indeed, if the nozzle is strictly
expanding or contracting, only one shock solution has been established
( see Figure \ref{fig:Nozzle_Div}) via the arguments in this paper.
However, for a general nozzle with both expanding and contracting
portions, more than one shock solutions could be established for the
same given receiver pressure at the exit (see Figure \ref{fig:Nozzle_Wav}).
\begin{figure}[th]
\centering
\def\svgwidth{200pt}
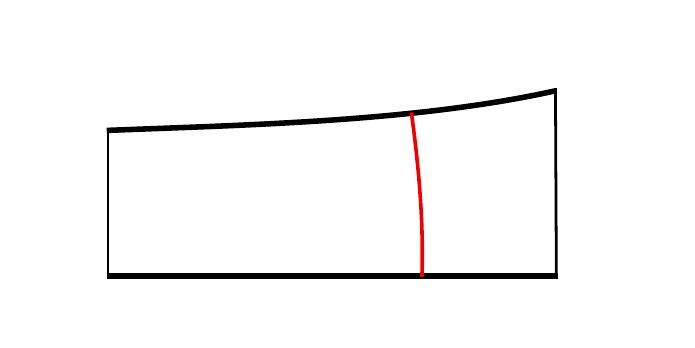
\def\svgwidth{200pt}
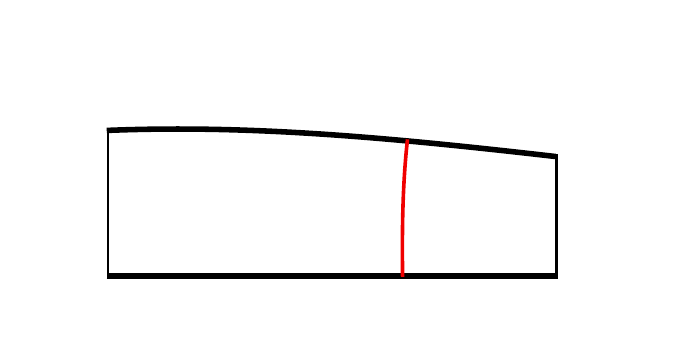

\caption{Existence of the shock solutions in a strictly expanding/contracting
nozzle.\label{fig:Nozzle_Div}}
\end{figure}
It should be remarked that the phenomenon observed here for 2-D steady
Euler system is similar to the one observed by Liu in \cite{Liu1982ARMA,Liu1982CMP}
and by Embid-Goodman-Majda in \cite{EmbidGoodmanMajda1984} for the
gas flows in a nozzle of variable areas governed by a quasi-one-dimensional
model.

\begin{figure}[th]
\centering
\def\svgwidth{400pt}
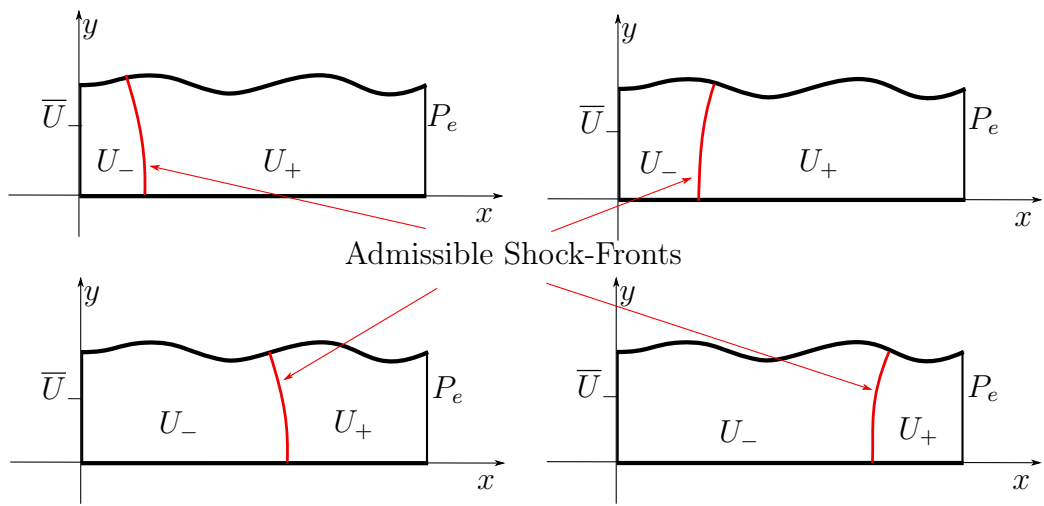

\caption{Existence of multiple transonic shock solutions in a nozzle with both
expanding and contracting portions for the same given receiver pressure.\label{fig:Nozzle_Wav}}
\end{figure}

The transonic shock problem formulated here is a free boundary value problem for mixed type equations. One of the key difficulties in establishing the existence of a solution to this problem is to determine the location of the free boundary( shock-front ). Indeed, in the theory of structural stability of the transonic shocks found in \cite{CourantFriedrichs1948}, the well-posedness of the transonic shock problem established by Li-Xin-Yin in  \cite{LiXinYin2009CMP,LiXinYin2009MRL,LiXinYin2013ARMA} depends crucially on the fact that the location of the background transonic shock is uniquely determined which serves the leading order approximation of the perturbed shock front. Clearly, this key argument cannot be adapted directly to the present case since the location of the background shock front can be arbitrary. Previous attempt going around this difficulty is pre-assuming artifically that the shock front goes through a fixed point on the wall of the nozzle as in \cite{ChenChenSong2006JDE,Chen_S2005TAMS,XinYanYin2011ARMA,XinYin2005CPAM,XinYin2008PJM}. However, this would make the transonic shock problem as formulated above ill-posed unless the boundary condition that the receiver pressure given at the exit of the nozzle is modified as in \cite{ChenChenSong2006JDE,Chen_S2005TAMS,XinYanYin2011ARMA,XinYin2005CPAM,XinYin2008PJM}.
In this paper, we propose an elaborate scheme to overcome this difficult and establish the existence of solutions to the transonic shock problem.
The key step in this scheme is to determine an approximating position of the
shock-front directly from the prescribed boundary data by designing
a free boundary problem for the linearized Euler system based on the
unperturbed normal shock solution $\left(\overline{U}_{+},\overline{U}_{-}\right)$,
in which the free boundary will be taken as the initial approximation of the perturbed shock front
( see Section 2 for detail formulation ). Since, as observed, for
instance, in \cite{Chen_S2005TAMS,LiXinYin2013ARMA,Serre1995ARMA},
the Euler system for subsonic flows is elliptic-hyperbolic composite
and it can be decomposed into an elliptic sub-system for the pressure
and the flow angle as well as two transport equations, a solvability
condition for the boundary data should be satisfied in order that
there exists a solution to the sub-problem of the elliptic sub-system,
as in \cite{LiXinYin2009CMP,LiXinYin2009MRL,LiXinYin2013ARMA}. It
turns out that, in the designed linear free boundary problem, the location
of the free boundary can be determined by this very solvability condition
for the linearized Euler system at the subsonic state $\overline{U}_{+}$.
Fortunately, it will be shown that the solvability condition can be formulated explicitly
as an equation solving the parameter representing the location of the
free boundary, the data on the boundary
geometry of the nozzle, the supersonic state at the entrance, and the receiver pressure at the exit, such that its solutions can
be analyzed clearly. In particular, as shown in Section
3, for either expanding or contracting nozzles, this
solvability condition yields a unique approximating location of the
shock-front; while for a general nozzle with both expanding
and contracting portions, there could exist more than one solutions
to the equation of the solvability condition for the same given receiver
pressure, and each one can be taken as an initial approximation of
the perturbed shock-front. Once this initial approximation is obtained, we can design a nonlinear iteration, similar as
in \cite{LiXinYin2009CMP,LiXinYin2009MRL,LiXinYin2013ARMA}, which
yields a transonic shock solution with its shock-front close
to the initial approximation, as long as the initial approximation
lies at the expanding or contracting portion and the perturbation
is suitably small. Thus, for a general nozzle, there may exist more
than one admissible locations of the shock-front satisfying both Rankine-Hugoniot
conditions and entropy conditions with the same prescribed receiver
pressure (see Figure \ref{fig:Nozzle_Wav}). This non-uniqueness of the shock
solutions also shows the instability of the unperturbed normal
shock solution under small perturbation of the nozzle and how this
instability may behave.

It is well-known that the study of gas flows in nozzles plays a fundamental
role in the operation of turbines, wind tunnels and rockets. In \cite{CourantFriedrichs1948},
Courant and Friedrichs first gave a systematic analysis mathematically
via nonlinear partial differential equations for various types of
steady inviscid flows in nozzles. The steady inviscid flow involving
a single shock-front, which enters the nozzle with a supersonic state
and leave with a subsonic one at the exit, is one of the typical flow
patterns they analyzed. They pointed out that the position of the
shock front cannot be uniquely determined unless additional conditions
are imposed at the exit and the pressure condition is suggested and
preferred (see \cite[Page 373-374]{CourantFriedrichs1948}). In particular, they established the unique existence of the transonic shock
solution in an expanding angular sector or a diverging cone with given constant pressure at the
exit and asked whether the similar patterns hold for general de Laval nozzles. However, such a theory may not be true for arbitrary nozzles.
It turns out that this is a challenging problem and the existence of piecewise smooth solutions separated by a single transonic shock in a nozzle depends sensitively on the conditions imposed at the exit of the nozzle.
Up to now, plenty of efforts for various models
and different boundary conditions at the exit of the nozzle have been made in order to establish
the rigorous mathematical analysis to a transonic shock in a nozzle. In 1980s, T.P.
Liu studied the unsteady transonic gas flows governed by quasi-one-dimensional
models in an infinite nozzle of varying area. In \cite{Liu1982ARMA,Liu1982CMP},
for certain given Cauchy data, he proved the existence of shocks in
the nozzle and showed that a transonic shock-front occured in the expanding
portion is dynamically stable, while the one in the contracting portion
is unstable. In \cite{EmbidGoodmanMajda1984}, Embid-Goodman-Majda
showed that there may exist multiple shock solutions for steady quasi-one-dimensional
flows in a nozzle with both expanding and contracting portions. See
also, for instance, \cite{ChenGengZhang2009SIAMJMA,GlazLiu1984AAM,RauchXieXin2013JMPA}
and references therein for literatures on quasi-one-dimensional nozzle
flows. As to the steady multi-dimensional models, thanks to continuous
efforts of many mathematicians, there have been substantial progresses\textcolor{blue}{{}
}in the past two decades, for instance, see \cite{BaeFeldman2011ARMA,ChenChenFeldman2007JMPA,ChenChenSong2006JDE,ChenFeldman2003JAMS,ChenFeldman2004CPAM,ChenFeldman2007ARMA,Chen_S2005TAMS,Chen_S2008TAMS,Chen_S2009CMP,ChenYuan2008ARMA,ChenYuan2013CPAA,CuiYin2008JPDE,DuanWeng2011JDE,FangLiuYuan2013ARMA,LiuXuYuan2016AdM,LiuYuan2008JHDE,LiuYuan2009SIAMJMA,LiXinYin2009CMP,LiXinYin2009MRL,LiXinYin2010JDE,LiXinYin2010PJM,LiXinYin2011PJM,LiXinYin2013ARMA,XieWang2007JDE,XinYanYin2011ARMA,XinYin2005CPAM,XinYin2008JDE,XinYin2008PJM,Yuan2008NARWA,Yuan2012NA}.
Mathematical theories for transonic shocks in a nozzle with various conditions
at the exit have been established. In particular, Chen-Feldman proved
in \cite{ChenFeldman2003JAMS} the existence of transonic shock solutions
in a finite nozzle for multi-dimensional potential flows with given
potential value at the exit. It is surprising to know that the potential
equation fails to give the unique transonic shock solution as shown
in \cite{XinYin2008PJM}. It is observed in \cite{Serre1995ARMA}
by D. Serre that the steady compressible Euler system for subsonic flows is elliptic-hyperbolic
composite and could be decomposed into an elliptic sub-system of first
order for the pressure and the flow angle, as well as transport equations.
Such a decomposition has been a major tool to analyze the transonic shock in a nozzle for steady Euler flows in \cite{Chen_S2005TAMS,LiXinYin2009CMP,LiXinYin2009MRL,LiXinYin2013ARMA}, which will also be used in this paper.
We will use the Lagrange transformation
to straighten the streamlines which was also first introduced in \cite{Chen_S2005TAMS}
and employed later in many literatures. It should be noted that, in
a series of papers \cite{LiXinYin2009CMP,LiXinYin2009MRL,LiXinYin2013ARMA},
with prescribed pressure at the exit, as conjectured by Courant and
Friedrichs, Li-Xin-Yin established the well-posedness of the shock solutions
for the steady Euler flows in a diverging nozzle, which is a generic perturbation of an expanding angular sector, which, in particular, yields the structual stability of the transonic shock solutions by Courant and
Friedrichs in \cite{CourantFriedrichs1948}.  
See also \cite{ChenYuan2013CPAA,CuiYin2008JPDE,FangLiuYuan2013ARMA,LiuYuan2009SIAMJMA,XinYanYin2011ARMA}
for studies on the uniqueness of the shock solution in general nozzles. Based on the theories established and techniques developed in
these literatures, in this paper, the steady 2-D Euler flows with
a single shock-front in an almost flat nozzle with prescribed receiver
pressure at the exit will be investigated and the admissible locations
of the shock-front will be determined under some natural sufficient conditions. 

\medskip{}

The paper is organized as follows. In Section 2, we formulate the
problem as a free boundary problem for the steady 2-D Euler system, which is then
reformulated via the Lagrange transformation. Then a free
boundary problem of the linearized Euler system based on the background
normal shock solution $\left(\overline{U}_{+},\overline{U}_{-}\right)$,
in which the free boundary will be taken as the initial approximation
for the shock-front, will be derived. The main results will also
be presented. In Section 3, we solve the proposed linear free boundary problem
and determine the location of the free boundary via the solvability
condition for the sub-problem of the elliptic sub-system of the linearized
Euler system at $\overline{U}_{+}$. With the initial approximation,
in Section 4 and 5, we can construct the nonlinear iteration scheme and
prove its convergence, which demonstrates the existence of a transonic shock-front
close to the initial approximation. Finally, in the appendix, we present the well-posed theory for the mixed boundary value problem of linear elliptic systems of first order with constant coefficients, which will be employed in the anaylysis for the elliptic sub-system of the Euler system for subsonic flows.

\section{Formulations and Main Results}

In this section, we first formulate the existence problem of the transonic
shock in an almost flat nozzle as a free boundary problem for the
steady 2-D Euler system, and then reformulate the problem by the Lagrange
transformation, and finally state the main results in this paper.

\subsection{The formulation of the free boundary problem}

The motion of the compressible fluid in the nozzle is governed by the full steady
2-D Euler system
\begin{align}
 & \partial_{x}\left(\rho u\right)+\partial_{y}\left(\rho v\right)=0,\label{eq:full-Euler-mass}\\
 & \partial_{x}\left(\rho u^{2}+p\right)+\partial_{y}\left(\rho uv\right)=0,\label{eq:full-Euler-momentum-x}\\
 & \partial_{x}\left(\rho uv\right)+\partial_{y}\left(\rho v^{2}+p\right)=0,\label{eq:full-Euler-momentum-y}\\
 & \partial_{x}\left(\rho u\Phi\right)+\partial_{y}\left(\rho v\Phi\right)=0,\label{eq:full-Euler-energy}
\end{align}
where $\left(u,\ v\right)^{\top}$ is the velocity field, $\left(p,\ S\ ,\rho\right)$
represents the pressure, the density and the entropy, of which only
two are independent and the equation of state is $\rho=\rho\left(p,S\right)$,
and $\Phi:=\frac{1}{2}\left(u^{2}+v^{2}\right)+i$ with $i:=e+p/\rho$
the enthalpy and $e$ the internal energy. In particular, for polytropic
gases, $\rho=A\left(S\right)p^{1/\gamma}$ with $\gamma$ the adiabatic
exponent, or specifically,
\begin{equation}
p=\left(\gamma-1\right)\me^{\left(S-S_{0}\right)/c_{\mr{v}}}\rho^{\gamma},\label{eq:polytropic}
\end{equation}
where $S_{0}$ is some constant, and $c_{\mr{v}}$ is the specific
heat at constant volume. Let $U=\left(p,\theta,q,S\right)^{\top}$represent
the state of the fluid, where $\theta=\arctan{\displaystyle \frac{v}{u}}$
and $q=\sqrt{u^{2}+v^{2}}$.

Across a shock-front, the Rankine-Hugoniot conditions are:
\begin{align}
 & \left[\rho u\right]-\varphi'\left[\rho v\right]=0,\label{eq:RH_Euler_1}\\
 & \left[\rho u^{2}+p\right]-\varphi'\left[\rho uv\right]=0,\label{eq:RH_Euler_2}\\
 & \left[\rho uv\right]-\varphi'\left[\rho v^{2}+p\right]=0,\label{eq:RH_Euler_3}\\
 & \left[\Phi\right]=0,\label{eq:RH_Euler_4}
\end{align}
where $x=\varphi\left(y\right)$ is the location of the shock-front,
$\varphi'\defs\displaystyle\frac{\dif}{\dif y}\varphi$, and $\left[\cdot\right]$
stands for the jump of the corresponding quantity across the shock-front.
Moreover, the entropy condition is that $[p]=p_{+}-p_{-}>0$.

Let 
\[
\fd=\set{\left(x,y\right)^{\top}\in\Real^{2}:\ 0<x<L,\ 0<y<\varphi_{\mr{w}}\left(x\right)}
\]
be the nozzle with length $ L $, lower wall $ W_1 $, upper wall $ W_2 $, entrance $ E_{1} $, and exit $ E_2 $. ( See Figure \ref{fig:Nozzle_waved}). 
Here for simplicity of the computation and the representation,
the lower wall of the nozzle is assumed to be flat. 

For a flat nozzle, one may assume $\varphi_{\mr{w}}\left(x\right)=\overline{\varphi}_{\mr{w}}\left(x\right)\equiv1$.
Then for a given uniform supersonic state $\overline{U}_{-}=\left(\overline{p}_{-},0,\overline{q}_{-},\overline{S}_{-}\right)^{\top}$,
there exists a unique subsonic state $\overline{U}_{+}=\left(\overline{p}_{+},0,\overline{q}_{+},\overline{S}_{+}\right)^{\top}$
such that
\begin{align}
 & \left[\overline{\rho}\cdot\overline{q}\right]=0,\label{eq:RH_bg_mass}\\
 & \left[\overline{\rho}\cdot\overline{q}^{2}+\overline{p}\right]=0,\label{eq:RH_bg_momentum}\\
 & \left[\ \overline{\Phi}\ \right]=0,\label{eq:RH_bg_Bernoulli}
\end{align}
which can be connected to $\overline{U}_{-}$ through a normal transonic
shock-front $x=\overline{x}_{\mr{s}}$, and $\overline{x}_{\mr{s}}$
can be any value in the interval $\left(0,\ L\right)$. See Figure
\ref{fig:normal_shocks}.

Without loss of generality, by \eqref{eq:RH_bg_mass}, one may assume
\begin{equation}
\overline{\rho}_{-}\overline{q}_{-}=\overline{\rho}_{+}\overline{q}_{+}=1.\label{eq:total_mass}
\end{equation}

In this paper, we are trying to solve the transonic shock problem with  a given receiver pressure at the exit for almost straight nozzles. This problem can be formulated as a free boundary problem as follows.

\medskip{}

\textbf{The Free Boundary Problem $\llbracket NS\rrbracket$. }

Let $P\in\mcc^{2+\alpha}(\overline{\Real_{+}})$ and $\Theta\in\mcc^{2+\alpha}[0,L]$
be given for some $0<\alpha<1$, satisfying
\begin{align}
 & \norm{P}_{\mcc^{2+\alpha}(\overline{\Real_{+}})}<+\infty,\qquad\norm{\Theta}_{\mcc^{2+\alpha}[0,L]}=1,\label{eq:norm_profiles}\\
 & \Theta(0)=\Theta'(0)=\Theta''(0)=0.\label{eq:upper_bdry_compatibility_2nd}
\end{align}
Assume that
\begin{eqnarray}
P_{e}\left(y\right) & \defs & \overline{p}_{+}+\sigma P\left(y\right),\label{eq:receiver_pressure}\\
\varphi_{\mr{w}}(x) & \defs & 1+\int_{0}^{x}\tan\left(\sigma\Theta(s)\right)\dif s,\label{eq:upper_wall}
\end{eqnarray}
where $\sigma>0$ is a sufficiently small constant. Then look for
a shock solution $\left(U_{-};U_{+};\varphi_{\mr{s}}\right)$ to the
Euler system \eqref{eq:full-Euler-mass}—\eqref{eq:full-Euler-energy}
in $\fd$ (see Figure \ref{fig:Nozzle_waved}), such that:
\begin{enumerate}
\item The location of the shock front is given by
\[
F_{\mr{s}}=\set{\left(x,y\right)^{\top}\in\Real^{2}:\ x=\varphi_{\mr{s}}\left(y\right),\ 0<y<Y_{\mr{s}}},
\]
where $(Y_{\mr{s}},\ \varphi_{\mr{s}}(Y_{\mr{s}}))$ is the intersection of the shock-front with the upper wall.
\item $U_{-}(x,y)$ satisfies the Euler system \eqref{eq:full-Euler-mass}—\eqref{eq:full-Euler-energy}
in the domain $\fd_{-}$ ahead of the shock-front
\[
\fd_{-}=\set{\left(x,y\right)^{\top}\in\Real^{2}:\ 0<x<\varphi_{\mr{s}}\left(y\right),\ 0<y<\varphi_{\mr{w}}\left(x\right)},
\]
coincides with the uniform supersonic state $\overline{U}_{-}$
at the entrance $E_{1}$ of the nozzle:
\begin{equation}
U_{-}=\overline{U}_{-},\text{ on }E_{1},\label{eq:entrance_cond}
\end{equation}
and satisfies the slip boundary condition on the walls of the nozzle:
\begin{align}
 & v_{-}=0, & \text{ on } & W_{1},\label{eq:nozzle_lower_cond_super}\\
 & v_{-}-\varphi_{\mr{w}}'\cdot u_{-}=0, & \text{ on } & W_{2}.\label{eq:nozzle_upper_cond_super}
\end{align}
\item $U_{+}(x,y)$ satisfies the Euler system \eqref{eq:full-Euler-mass}—\eqref{eq:full-Euler-energy}
in the domain $\fd_{+}$ behind the shock-front
\begin{eqnarray*}
\fd_{+} & = & \set{\left(x,y\right)^{\top}\in\Real^{2}:\ \varphi_{\mr{s}}\left(y\right)<x<L,\ 0<y<\varphi_{\mr{w}}\left(x\right)},
\end{eqnarray*}
the slip boundary condition on the walls of the nozzle:
\begin{align}
 & v_{+}=0, & \text{ on } & W_{1},\label{eq:nozzle_lower_cond_sub}\\
 & v_{+}-\varphi_{\mr{w}}'\cdot u_{+}=0, & \text{ on } & W_{2},\label{eq:nozzle_upper_cond_sub}
\end{align}
and the receiver pressure at the exit:
\begin{equation}
p_{+}=P_{e}\left(y\right),\text{ on }E_{2}.\label{eq:receiver_pressure_cond}
\end{equation}
\item Finally, on the shock front $F_{\mr{s}}$, $\left(U_{-};U_{+};\varphi_{\mr{s}}'\right)$
satisfies the R-H conditions \eqref{eq:RH_Euler_1}---\eqref{eq:RH_Euler_4}.
\end{enumerate}

\subsection{Reformulation via the Lagrange transformation}

For steady flows, the streamlines coincide with the characteristics
corresponding to the linearly degenerate eigenvalue of the Euler system.
And it is useful to employ the Lagrange transformation as in
\cite{Chen_S2005TAMS,LiXinYin2009CMP,LiXinYin2013ARMA}  to straighten the streamlines, which turns
out to be crucial in the regularity of solutions in the subsonic
region. Below we describe formally the Lagrange transformation for
convenience of the reader, and one is referred to \cite{Chen_S2005TAMS,LiXinYin2009CMP,LiXinYin2013ARMA}
for more rigorous derivations.

The Lagrange transformation is defined as:
\[
\begin{cases}
\xi=x,\\
\eta=\int_{\left(0,0\right)}^{\left(x,y\right)}\rho u(s,t)\dif t-\rho v(s,t)\dif s.
\end{cases}
\]
Then the Euler system \eqref{eq:full-Euler-mass}—\eqref{eq:full-Euler-energy}
changes to
\begin{align}
 & \partial_{\xi}\left(\frac{1}{\rho u}\right)-\partial_{\eta}\left(\frac{v}{u}\right)=0,\label{eq:Lagrange_Euler_mass}\\
 & \partial_{\xi}\left(u+\frac{p}{\rho u}\right)-\partial_{\eta}\left(\frac{pv}{u}\right)=0,\label{eq:Lagrange_Euler_momentum_x}\\
 & \partial_{\xi}v+\partial_{\eta}p=0,\label{eq:Lagrange_Euler_momentum_y}\\
 & \partial_{\xi}\Phi=0.\label{eq:Lagrange_Euler_energy}
\end{align}

Under the Lagrange transformation, the shock-front $F_{\mr{s}}$ becomes
\[
\Gamma_{\mr{s}}=\set{\left(\xi,\eta\right):\ \xi=\psi\left(\eta\right),\ 0<\eta<1},
\]
and the R-H conditions \eqref{eq:RH_Euler_1}—\eqref{eq:RH_Euler_4}
on $F_{\mr{s}}$ become: 
\begin{align}
 & \left[\frac{1}{\rho u}\right]+\psi'\left[\frac{v}{u}\right]=0,\label{eq:RH_Lagrange_1-o}\\
 & \left[u+\frac{p}{\rho u}\right]+\psi'\left[\frac{pv}{u}\right]=0,\label{eq:RH_Lagrange_2-o}\\
 & \left[v\right]-\psi'\left[p\right]=0,\label{eq:RH_Lagrange_3-o}\\
 & \left[\Phi\right]=0.\label{eq:RH_Lagrange_4-o}
\end{align}
Eliminating $\psi'$ in above conditions yields
\begin{eqnarray}
G_{1}\left(U_{+},U_{-}\right) & \defs & \left[\frac{1}{\rho u}\right]\left[p\right]+\left[\frac{v}{u}\right]\left[v\right]=0,\label{eq:RH_Lagrange_1}\\
G_{2}\left(U_{+},U_{-}\right) & \defs & \left[u+\frac{p}{\rho u}\right]\left[p\right]+\left[\frac{pv}{u}\right]\left[v\right]=0,\label{eq:RH_Lagrange_2}\\
G_{3}\left(U_{+},U_{-}\right) & \defs & \left[\frac{1}{2}q^{2}+i\right]=0,\label{eq:RH_Lagrange_3}
\end{eqnarray}
and \eqref{eq:RH_Lagrange_3-o} can be used to determine
the location of the shock: 
\begin{equation}
G_{4}\left(U_{+},U_{-};\psi'\right)\defs\left[v\right]-\psi'\left[p\right]=0,\label{eq:RH_Lagrange_4_bdry}
\end{equation}

Since the speed of the background shock solution $\left(\overline{U}_{+},\overline{U}_{-}\right)$ is zero,
it holds that 
\begin{align}
 & \left[\frac{1}{\overline{\rho}\cdot\overline{q}}\right]=0,\label{eq:RH_Lagrange_bg_1}\\
 & \left[\overline{q}+\frac{\overline{p}}{\overline{\rho}\cdot\overline{q}}\right]=0,\label{eq:RH_Lagrange_bg_2}\\
 & \left[\ \overline{\Phi}\ \right]=0.\label{eq:RH_Lagrange_bg_3}
\end{align}

Furthermore, in the Lagrange coordinates, the entance $E_{1}$
and the exit $E_{2}$ become, respectively,
\begin{align*}
\Gamma_{1} & =\set{\left(\xi,\eta\right):\ \xi=0,\ 0<\eta<1},\\
\Gamma_{3} & =\set{\left(\xi,\eta\right):\ \xi=L,\ 0<\eta<1},
\end{align*}
while the lower wall $W_{1}$ and the upper wall $W_{2}$ change to, respectively
\begin{align*}
\Gamma_{2} & =\set{\left(\xi,\eta\right):\ 0<\xi<L,\ \eta=0},\\
\Gamma_{4} & =\set{\left(\xi,\eta\right):\ 0<\xi<L,\ \eta=1}.
\end{align*}
Therefore, the whole nozzle $\fd$ becomes (see Figure \ref{fig:Domain_Lagrange})
\[
\Omega=\set{\left(\xi,\eta\right):\ 0<\xi<L,\ 0<\eta<1},
\]
with the supersonic region $\fd_{-}$, as well as the subsonic region
$\fd_{+}$ as, respectively, 
\begin{align*}
\Omega_{-} & =\set{\left(\xi,\eta\right):\ 0<\xi<\psi\left(\eta\right),\ 0<\eta<1},\\
\Omega_{+} & =\set{\left(\xi,\eta\right):\ \psi\left(\eta\right)<\xi<L,\ 0<\eta<1}.
\end{align*}

\begin{figure}[th]
	\centering
	\def\svgwidth{200pt}
	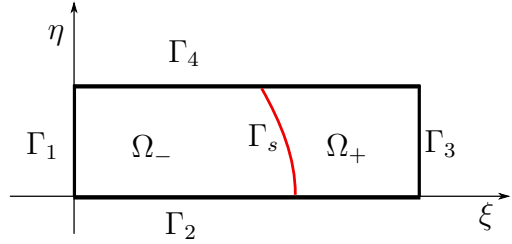
	
	\caption{The domain under the Lagrange transformation.\label{fig:Domain_Lagrange}}
\end{figure}

Thus, the free boundary problem $\llbracket NS\rrbracket$ can be reformulated as below.

\medskip{}

\textbf{The Free Boundary Problem $\llbracket NSL\rrbracket$. }

Let $P\in\mcc^{2+\alpha}(\overline{\Real_{+}})$ and $\Theta\in\mcc^{2+\alpha}[0,L]$
be given as in the free boundary problem $\llbracket NS\rrbracket$,
satisfying \eqref{eq:norm_profiles} and \eqref{eq:upper_bdry_compatibility_2nd}.
One looks for a shock solution $\left(U_{-};U_{+};\psi\right)$ to
the Euler system \eqref{eq:Lagrange_Euler_mass}—\eqref{eq:Lagrange_Euler_energy}
in $\Omega$ (see Figure \ref{fig:Domain_Lagrange}), such that:
\begin{enumerate}
\item The location of the shock front $\Gamma_{\mr{s}}$ is given by
\[
\Gamma_{\mr{s}}=\set{\left(\xi,\eta\right)^{\top}\in\Real^{2}:\ \xi=\psi\left(\eta\right),\ 0<\eta<1}.
\]
\item $U_{-}(\xi,\eta)$ solves the Euler system \eqref{eq:Lagrange_Euler_mass}—\eqref{eq:Lagrange_Euler_energy}
in the domain $\Omega_{-}$ ahead of the shock-front, coincides with
the uniform supersonic state $\overline{U}_{-}$ on $\Gamma_{1}$:
\begin{equation}
U_{-}=\overline{U}_{-},\text{ on }\Gamma_{1},\label{eq:Lagrange_entrance_cond}
\end{equation}
and satisfies the slip boundary condition on the walls of the nozzle:
\begin{align}
 & \theta_{-}=0, & \text{ on } & \Gamma_{2},\label{eq:nozzle_lower_cond_super-1}\\
 & \theta_{-}=\sigma\Theta(\xi), & \text{ on } & \Gamma_{4}.\label{eq:nozzle_upper_cond_super-1}
\end{align}
\item $U_{+}(\xi,\eta)$ solves the Euler system \eqref{eq:Lagrange_Euler_mass}—\eqref{eq:Lagrange_Euler_energy}
in the domain $\Omega_{+}$, and satisfies the slip boundary condition on the walls of the nozzle:
\begin{align}
 & \theta_{+}=0, & \text{ on } & \Gamma_{2},\label{eq:nozzle_lower_cond_sub-1}\\
 & \theta_{+}=\sigma\Theta(\xi), & \text{ on } & \Gamma_{4},\label{eq:nozzle_upper_cond_sub-1}
\end{align}
and the receiver pressure at the exit:
\begin{equation}
p_{+}=P_{e}\left(Y\left(L,\eta\right)\right),\text{ on }\Gamma_{3},\label{eq:receiver_pressure_cond-1}
\end{equation}
where 
\begin{eqnarray*}
Y\left(L,\eta\right) & = & \int_{0}^{\eta}\frac{1}{\left(\rho q\cos\theta\right)(L,s)}\dif s.
\end{eqnarray*}
\item Finally, on the shock front $\Gamma_{\mr{s}}$, $\left(U_{-};U_{+};\psi'\right)$
satisfies the R-H conditions \eqref{eq:RH_Lagrange_1}---\eqref{eq:RH_Lagrange_4_bdry}.
\end{enumerate}
\medskip{}

We are then going to solve the {free boundary problem $\llbracket NSL\rrbracket$}.
Once this is achieved, the {free boundary problem }$\llbracket NS\rrbracket$ is solved by inversing the Lagrange transformation.

\subsection{Characteristics of the Euler system}

Let $c=\sqrt{\partial_{\rho}p}$ be the sonic speed, and $M=q/c$
the Mach number. It follows from \eqref{eq:Lagrange_Euler_mass},
\eqref{eq:Lagrange_Euler_momentum_x} and \eqref{eq:Lagrange_Euler_momentum_y}
that
\begin{equation}
\partial_{\xi}\left(\frac{1}{2}q^{2}\right)+\frac{1}{\rho}\partial_{\xi}p\equiv q\partial_{\xi}q+\frac{1}{\rho}\partial_{\xi}p=0,\label{eq:Lagrange_pq}
\end{equation}
or equivalently,
\[
\rho q\partial_{\xi}q+\partial_{\xi}p=0.
\]
Since
\[
\dif i=\frac{1}{\rho}\dif p+T\dif S,
\]
where $T$ represents the temperature. Then \eqref{eq:Lagrange_Euler_energy} and \eqref{eq:Lagrange_pq} yield that
\begin{equation}
\partial_{\xi}S=0.\label{eq:Lagrange_entropy}
\end{equation}
Moreover, it follows from \eqref{eq:Lagrange_Euler_momentum_y} and \eqref{eq:Lagrange_pq} that
\begin{equation}
-\frac{\sin\theta}{\rho q}\partial_{\xi}p+q\cos\theta\partial_{\xi}\theta+\partial_{\eta}p=0.\label{eq:Lagrange_pw_1}
\end{equation}
Finally, direct calculations and
\eqref{eq:Lagrange_Euler_mass}, \eqref{eq:Lagrange_pq}
and \eqref{eq:Lagrange_entropy} imply
\begin{equation}
-\frac{\cos\theta}{\rho q}\left(M^{2}-1\right)\partial_{\xi}p+q\sin\theta\partial_{\xi}\theta-\rho q^{2}\partial_{\eta}\theta=0.\label{eq:Lagrange_pw_2}
\end{equation}

Thus, the steady Euler equations,
\eqref{eq:Lagrange_Euler_mass}---\eqref{eq:Lagrange_Euler_energy}
can be rewritten as
\begin{align}
 & \partial_{\eta}p-\frac{\sin\theta}{\rho q}\partial_{\xi}p+q\cos\theta\partial_{\xi}\theta=0,\label{eq:Euler_Lagrange_matrix_1}\\
 & \partial_{\eta}\theta-\frac{\sin\theta}{\rho q}\partial_{\xi}\theta-\frac{\cos\theta}{\rho q}\cdot\frac{1-M^{2}}{\rho q^{2}}\partial_{\xi}p=0,\label{eq:Euler_Lagrange_matrix_2}\\
 & \rho q\partial_{\xi}q+\partial_{\xi}p=0,\label{eq:Euler_Lagrange_matrix_3}\\
 & \partial_{\xi}S=0.\label{eq:Euler_Lagrange_matrix_4}
\end{align}
\eqref{eq:Euler_Lagrange_matrix_3} can be also replaced by
\begin{equation}
\partial_{\xi}\Phi=0.\label{eq:Euler_Lagrange_matrix_3-B}
\end{equation}
It is clear that \eqref{eq:Euler_Lagrange_matrix_3} 
and \eqref{eq:Euler_Lagrange_matrix_4} are simple hyperbolic equations.
Furthermore, the equations \eqref{eq:Euler_Lagrange_matrix_1} and
\eqref{eq:Euler_Lagrange_matrix_2} can be written as
\[
A\left(U\right)\partial_{\xi}\begin{bmatrix}p\\
\theta
\end{bmatrix}+\partial_{\eta}\begin{bmatrix}p\\
\theta
\end{bmatrix}=0,
\]
where 
\[
A\left(U\right)=\frac{1}{\rho q}\begin{bmatrix}-\sin\theta & \rho q^{2}\cos\theta\\
\frac{M^{2}-1}{\rho q^{2}}\cos\theta & -\sin\theta
\end{bmatrix}.
\]
The eigenvalues of $A\left(U\right)$ are
\[
\lambda_{\pm}=\frac{1}{\rho q}\left(-\sin\theta\pm\cos\theta\sqrt{M^{2}-1}\right).
\]
Hence, for supersonic flows, $M>1$, the eigenvalues are real numbers so that the equations \eqref{eq:Euler_Lagrange_matrix_1}
and \eqref{eq:Euler_Lagrange_matrix_2} form a hyperbolic system;
while for subsonic flows, $M<1$, the eigenvalues
are a pair of conjugate complex numbers so that the equations
\eqref{eq:Euler_Lagrange_matrix_1} and \eqref{eq:Euler_Lagrange_matrix_2}
form an elliptic system of first order. In conclusion, the Euler system
\eqref{eq:Euler_Lagrange_matrix_1}—\eqref{eq:Euler_Lagrange_matrix_4}
is a quasilinear hyperbolic system in the supersonic region, and
it is an elliptic-hyperbolic composite system in the subsonic region.

\subsection{The initial approximating locations of the shock-front}

For the {free boundary problem $\llbracket NSL\rrbracket$,} since the Euler system is hyperbolic for supersonic flows, the state of supersonic flow
field $U_{-}$ ahead the shock front can be determined by applying
the well-known theory for local classical solutions of quasilinear
hyperbolic systems (see \cite{LiYu1985}). Therefore, it suffices to determine the location of the shock front as well as the state
of the subsonic flow field $U_{+}$ behind the shock front. However, since the location of the unperturbed normal shock solution $\left(\overline{U}_{+},\overline{U}_{-}\right)$ can be anywhere in the nozzle, it  gives no information where the shock-front may appear as the nozzle's boundary is perturbed.
Therefore, a key difficulty is that how to determine the shock-front. To deal with this difficulty, in this paper, a free boundary problem for the linearized Euler system will be designed, in which the free boundary will be taken as an initial approximating location for the shock-front. With this initial approximation, a further nonlinear iteration scheme will be designed which lead to a transonic shock solution to the free boundary problem $\llbracket NSL\rrbracket$.

Let $\overline{\psi}(\eta)\equiv\overline{\xi}_{*}$, with $\overline{\xi}_{*}\in\left(0,L\right)$ an unknown constant to be determined, and
\[
\dot{\Gamma}_{\mr{s}}=\set{\left(\xi,\eta\right):\ \xi=\overline{\psi}(\eta),\ 0<\eta<1},
\]
which is taken as the initial approximating location of the shock-front.
Clearly, $\overline{\psi} '(\eta)\equiv 0$. Then, $\dot{\Gamma}_{\mr{s}}$ devides the domain $\Omega$ into two
parts $\dot{\Omega}_{-}$ and $\dot{\Omega}_{+}$ ( see Figure \ref{fig:Domain_First_approx}) as:
\begin{eqnarray*}
\dot{\Omega}_{-} & = & \set{\left(\xi,\eta\right):\ 0<\xi<\overline{\xi}_{*},\ 0<\eta<1},\\
\dot{\Omega}_{+} & = & \set{\left(\xi,\eta\right):\ \overline{\xi}_{*}<\xi<L,\ 0<\eta<1},
\end{eqnarray*}
and the boundaries $\Gamma_{2}$ and $\Gamma_{4}$  consist of
\begin{eqnarray*}
\dot{\Gamma}_{2}^{-} & = & \set{\left(\xi,\eta\right):\ 0<\xi<\overline{\xi}_{*},\ \eta=0},\\
\dot{\Gamma}_{2}^{+} & = & \set{\left(\xi,\eta\right):\ \overline{\xi}_{*}<\xi<L,\ \eta=0},\\
\dot{\Gamma}_{4}^{-} & = & \set{\left(\xi,\eta\right):\ 0<\xi<\overline{\xi}_{*},\ \eta=1},\\
\dot{\Gamma}_{4}^{+} & = & \set{\left(\xi,\eta\right):\ \overline{\xi}_{*}<\xi<L,\ \eta=1}.
\end{eqnarray*}

Assume that $\left(\overline{U}_{-};\overline{U}_{+}\right)$ is the unperturbed normal shock solution, with $\overline{U}_{-}=\left(\overline{p}_{-},0,\overline{q}_{-},\overline{S}_{-}\right)^{\top}$ and $\overline{U}_{+}=\left(\overline{p}_{+},0,\overline{q}_{+},\overline{S}_{+}\right)^{\top}$. 
Let $\dot{U}_{-} \defs \left(\dot{p}_{-},\dot{\theta}_{-},\dot{q}_{-},\dot{S}_{-}\right)^{\top}$  satisfy the the linearized Euler system at the uniform supersonic state $\overline{U}_{-}$ below, which will yield an initial approximation for the supersonic flow $U_{-}$:
\begin{align}
& \partial_{\eta}\dot{p}_{-}+\overline{q}_{-}\partial_{\xi}\dot{\theta}_{-}=0,\label{eq:first_approx_super_1}\\
& \partial_{\eta}\dot{\theta}_{-}-\frac{1}{\overline{\rho}_{-}\overline{q}_{-}}\cdot\frac{1-\overline{M}_{-}^{2}}{\overline{\rho}_{-}\overline{q}_{-}^{2}}\partial_{\xi}\dot{p}_{-}=0,\label{eq:first_approx_super_2}\\
& \overline{\rho}_{-}\overline{q}_{-}\partial_{\xi}\dot{q}_{-}+\partial_{\xi}\dot{p}_{-}=0,\label{eq:first_approx_super_3}\\
& \partial_{\xi}\dot{S}_{-}=0.\label{eq:first_approx_super_4}
\end{align}
Let $\dot{U}_{+} \defs \left(\dot{p}_{+},\dot{\theta}_{+},\dot{q}_{+},\dot{S}_{+}\right)^{\top}$ satisfy the linearized Euler system at the uniform subsonic state $\overline{U}_{+}$ below, which will give an initial approximation for  the subsonic flow $U_{+}$:
\begin{align}
& \partial_{\eta}\dot{p}_{+}+\overline{q}_{+}\partial_{\xi}\dot{\theta}_{+}=0,\label{eq:first_approx_sub_1}\\
& \partial_{\eta}\dot{\theta}_{+}-\frac{1}{\overline{\rho}_{+}\overline{q}_{+}}\cdot\frac{1-\overline{M}_{+}^{2}}{\overline{\rho}_{+}\overline{q}_{+}^{2}}\partial_{\xi}\dot{p}_{+}=0,\label{eq:first_approx_sub_2}\\
& \partial_{\xi}\left(\overline{q}_{+}\dot{q}_{+}+\frac{1}{\overline{\rho}_{+}}\dot{p}_{+}+\overline{T}_{+}\dot{S}_{+}\right)=0,\label{eq:first_approx_sub_3}\\
& \partial_{\xi}\dot{S}_{+}=0.\label{eq:first_approx_sub_4}
\end{align}

\begin{figure}[th]
\centering
\def\svgwidth{200pt}
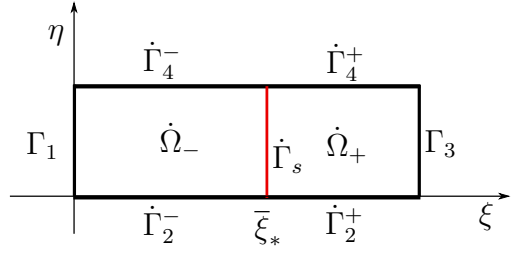

\caption{The domain for the problem on the initial approximating location of
the shock-front.\label{fig:Domain_First_approx}}
\end{figure}

Then the desired linear free boundary problem can be summarized as follows.
\medskip{}

\textbf{The Free Boundary Problem $\llbracket FL\rrbracket$ for the Linearization.}

Try to determine $\left(\dot{U}_{-};\dot{U}_{+};\dot{\psi}'|\ \overline{\xi}_{*}\right)$
such that
\begin{enumerate}
	\item $\dot{U}_{-}$ satisfies the linearized equations \eqref{eq:first_approx_super_1}-\eqref{eq:first_approx_super_4}
	in $\dot{\Omega}_{-}$ and the bounary conditions: 
	\begin{align}
	& \dot{U}_{-}\left(0,\eta\right)=0, & \text{ on } & \dot{\Gamma}_{1}, \label{eq:first_approx_super_bc-1}\\
	& \dot{\theta}_{-}=0, & \text{ on } & \dot{\Gamma}_{2}^{-}, \label{eq:first_approx_super_bc-2}\\
	& \dot{\theta}_{-}=\dot{\Theta}_{N}^{-}\left(\xi\right), & \text{ on } & \dot{\Gamma}_{4}^{-}, \label{eq:first_approx_super_bc-4}
	\end{align}
	where $\dot{\Theta}_{N}^{-}\left(\xi\right)=\sigma\Theta(\xi)$ with
	$\xi\in[0,\overline{\xi}_{*}]$;
	\item $\dot{U}_{+}$ satisfies the linearized equations \eqref{eq:first_approx_sub_1}-\eqref{eq:first_approx_sub_4}
	in $\dot{\Omega}_{+}$ and the bounary conditions: 
	\begin{align}
	& \dot{p}_{+}\left(L,\eta\right)=\dot{P}_{e}\left(\eta\right), & \text{ on } & \dot{\Gamma}_{3}, \label{eq:first_approx_sub_bc-3}\\
	& \dot{\theta}_{+}=0, & \text{ on } & \dot{\Gamma}_{2}^{+},\label{eq:first_approx_sub_bc-2}\\
	& \dot{\theta}_{+}=\dot{\Theta}_{N}^{+}\left(\xi\right), & \text{ on } & \dot{\Gamma}_{4}^{+},\label{eq:first_approx_sub_bc-4}
	\end{align}
	where $\dot{P}_{e}(\eta)=\sigma P(\eta)$ with $\eta\in[0,1]$, and $\dot{\Theta}_{N}^{+}\left(\xi\right)=\sigma\Theta(\xi)$ with
	$\xi\in[\overline{\xi}_{*},L]$;
	\item Across the free boundary $\dot{\Gamma}_{\mr{s}}$, where $\overline{\xi}_{*}\in(0,L)$ will be determined together with $\dot{U}_{-}$ and $\dot{U}_{+}$, $\left(\dot{U}_{-},\dot{U}_{+};\ \dot{\psi}' \right)$ satisfies the linearized R-H conditions 
	at $\left(\overline{U}_{+},\overline{U}_{-};\ \overline{\psi}'\right)$ below.
	\begin{align}
	& \beta_{j}^{+}\cdot\dot{U}_{+}+\beta_{j}^{-}\cdot\dot{U}_{-}=0,\quad j=1,2,3, & \text{ on } & \dot{\Gamma}_{s}\label{eq:first_approx_RH_1}\\
	& \beta_{4}^{+}\cdot\dot{U}_{+}+\beta_{4}^{-}\cdot\dot{U}_{-}-\left[\overline{p}\right]\dot{\psi}'=0, & \text{ on } & \dot{\Gamma}_{s}\label{eq:first_approx_RH_2}
	\end{align}
	where $\beta_{j}^{+}=\nabla_{U_{+}}G_{j}|_{\left(\overline{U}_{+},\overline{U}_{-}\right)}$,
	$\beta_{j}^{-}=\nabla_{U_{-}}G_{j}|_{\left(\overline{U}_{+},\overline{U}_{-}\right)}$
	$\left(j=1,2,3\right)$, $\beta_{4}^{+}=\nabla_{U_{+}}G_{4}|_{\left(\overline{U}_{+},\overline{U}_{-};\overline{\psi}'\right)}$,
	and $\beta_{4}^{-}=\nabla_{U_{-}}G_{4}|_{\left(\overline{U}_{+},\overline{U}_{-};\overline{\psi}'\right)}$.
\end{enumerate}
\medskip{}

It should be noted that the coefficients in the linearized R-H conditions \eqref{eq:first_approx_RH_1}-\eqref{eq:first_approx_RH_2} have explicit forms given below, which could be verified by direct computations.
\begin{lem}
\label{lem:first_approx_RH_coefficients}
\begin{eqnarray*}
\beta_{1}^{\pm}=\nabla_{U_{\pm}}G_{1}|_{\left(\overline{U}_{+},\overline{U}_{-}\right)} & = & \pm\frac{1}{\overline{\rho}_{\pm}\overline{q}_{\pm}}\left[\overline{p}\right]\cdot\left(-\frac{1}{\overline{\rho}_{\pm}\overline{c}_{\pm}^{2}},0,-\frac{1}{\overline{q}_{\pm}},\frac{1}{\gamma c_{\mr{v}}}\right)^{\top},\\
\beta_{2}^{\pm}=\nabla_{U_{\pm}}G_{2}|_{\left(\overline{U}_{+},\overline{U}_{-}\right)} & = & \pm\frac{1}{\overline{\rho}_{\pm}\overline{q}_{\pm}}\left[\overline{p}\right]\cdot\left(1-\frac{\overline{p}_{\pm}}{\overline{\rho}_{\pm}\overline{c}_{\pm}^{2}},0,\overline{\rho}_{\pm}\overline{q}_{\pm}-\frac{\overline{p}_{\pm}}{\overline{q}_{\pm}},\frac{\overline{p}_{\pm}}{\gamma c_{\mr{v}}}\right)^{\top},\\
\beta_{3}^{\pm}=\nabla_{U_{\pm}}G_{3}|_{\left(\overline{U}_{+},\overline{U}_{-}\right)} & = & \pm\left(\frac{1}{\overline{\rho}_{\pm}},0,\overline{q}_{\pm},\frac{1}{\left(\gamma-1\right)c_{\mr{v}}}\cdot\frac{\overline{p}_{\pm}}{\overline{\rho}_{\pm}}\right)^{\top},\\
\beta_{4}^{\pm}=\nabla_{U_{\pm}}G_{4}|_{\left(\overline{U}_{+},\overline{U}_{-};\overline{\psi}'\right)} & = & \pm\left(0,\overline{q}_{\pm},0,0\right)^{\top}.
\end{eqnarray*}
\end{lem}

\medskip{}

\begin{rem}
There are several remarks in order.
\begin{enumerate}
\item In the free boundary problem $\llbracket FL\rrbracket$, the linearized R-H condition \eqref{eq:first_approx_RH_2}
on $\dot{\Gamma}_{\mr{s}}$ will be used only in determining $\dot{\psi}'$
once $\dot{U}_{-}$ and $\dot{U}_{+}$ have been determined. And $\overline{\psi}'$ will be updated by  $\dot{\psi}'$ in the next step of the iteration.
\item In the free boundary problem $\llbracket FL\rrbracket$, $\dot{\Gamma}_{\mr{s}}$ is a free boundary,
which means that $\overline{\xi}_{*}$ is unknown and needs to be determined
together with $\dot{U}_{-}$ and $\dot{U}_{+}$. Indeed, since $\overline{U}_{+}$
is subsonic, the subsystem \eqref{eq:first_approx_sub_1}—\eqref{eq:first_approx_sub_2}
is an elliptic system of first order, thus there may exist no solutions
for the boundary value problem unless the boundary data satisfy certain
solvability condition which depends on the boundary geometry $\Theta$
and the perturbation of the receiver pressure $P$. Fortunately, as
will be shown later, the solvability condition can be formulated explicitely
and it will be employed to determine the value of $\overline{\xi}_{*}$.
Indeed, set
\begin{eqnarray}
R\left(\xi\right) & \defs & \int_{0}^{L}\Theta\left(\tau\right)\dif\tau-\dot{K}\int_{0}^{\xi}\Theta\left(\tau\right)\dif\tau,\qquad\xi\in(0,L),\label{eq:first_approx_criterion_nozzle_bdry}\\
\dot{P}_{*} & \defs & \frac{1}{\overline{\rho}_{+}\overline{q}_{+}}\cdot\frac{1-\overline{M}_{+}^{2}}{\overline{\rho}_{+}\overline{q}_{+}^{2}}\cdot\int_{0}^{1}P\left(\eta\right)\dif\eta,\label{eq:first_approx_criterion_pressure_exit}
\end{eqnarray}
where 
\[
\dot{K}\defs\left[\overline{p}\right]\cdot\left(\frac{\gamma-1}{\gamma\overline{p}_{+}}+\frac{1}{\overline{\rho}_{+}\overline{q}_{+}^{2}}\right)>0.
\]
Then the solvability condition turns out to be
\begin{equation}
R(\overline{\xi}_{*})=\dot{P}_{*}.\label{eq:first_approx_criterion_equation-1}
\end{equation}
Now let 
\begin{eqnarray*}
\underline{R} & \defs & \inf_{\xi\in\left(0,L\right)}R\left(\xi\right),\\
\overline{R} & \defs & \sup_{\xi\in\left(0,L\right)}R\left(\xi\right),
\end{eqnarray*}
and assume that
\begin{equation}
\underline{R}\leq\dot{P}_{*}\leq\overline{R},\label{eq:criterion_range_pressure_exit-1}
\end{equation}
then there exists at least one solution $\overline{\xi}_{*}\in[0,\ L]$
to the equation \eqref{eq:first_approx_criterion_equation-1} such
that there exists at least one solution to the free boundary problem $\llbracket FL\rrbracket$.
And obviously, for a general continuous function $\Theta$, there
may exist multiple solutions $\overline{\xi}_{*}\in[0,\ L]$ to the equation
\eqref{eq:first_approx_criterion_equation-1}.
\end{enumerate}
\end{rem}

\subsection{Main theorems and remarks}

Before describing the main theorems, we first introduce the function
spaces to be used in this paper and give some conventions on notations.
\begin{rem}[Function Spaces and Their Notations]
Different spaces for hyperbolic part and elliptic
part will be used.

\begin{enumerate}
\item For the hyperbolic part of the problem, it is natural to use the
classical H\"older spaces. Let $\Omega\subset\Real^{n}$ be a bounded
domain, $k\geq0$ be an interger, and $0<\alpha<1$. 
$\mcc^{k,\alpha}(\Omega)$ denote the classical H\"older spaces with the
index $(k,\alpha)$ for functions with continuous derivatives up to
$k$-th order, equipped with the classical $\mcc^{k,\alpha}(\Omega)$
norm:
\[
\norm{u}_{\mcc^{k,\alpha}(\Omega)}\defs\sum_{\abs{\mbnu}\leq k}\sup_{x\in\Omega}\abs{D^{\mbnu}u(x)}+\sum_{\abs{\mbnu}=k}\sup_{x,y\in\Omega;x\not=y}\frac{\abs{D^{\mbnu}u(x)-D^{{\bf \mbnu}}u(y)}}{\abs{x-y}^{\alpha}},
\]
where $\mbnu=(\nu_{1},\cdots,\nu_{n})$ is a multi-index and $\abs{\mbnu}=\sum\limits _{j=1}^{n}\nu_{j}$.
For the elliptic part of the problem, since the boundary of the
domain has corner singularities, the Sobolev spaces $W_{\beta}^{s}(\Omega)$
with $1\leq\beta\leq\infty$ will be employed. The index ``$s$''
will take real value, as defined in \cite{Grisvard1985}, for the
trace function on the boundary. Let $s=k+\alpha$, with $k\geq0$
an integer and $0<\alpha<1$. Define
\[
\norm{u}_{W_{\beta}^{s}(\Omega)}\defs\set{\norm{u}_{W_{\beta}^{k}(\Omega)}^{\beta}+\sum_{\abs{\mbnu}=k}\int\int_{\Omega\times\Omega}\frac{\abs{D^{\mbnu}u(x)-D^{{\bf \mbnu}}u(y)}^{\beta}}{\abs{x-y}^{n+\alpha\beta}}\dif x\dif y}^{1/\beta}.
\]
It should be noted that for any $u\in W_{\beta}^{1}(\Omega)$, its
trace on the boundary $u|_{\partial\Omega}\in W_{\beta}^{1-\frac{1}{\beta}}(\partial\Omega)$.
\item The Euler system for subsonic flows is elliptic-hyperbolic composite
and the Sobolev norms as well as the H\"older norms will be combined
to describe the flow behind the shock-front. Let $\overline{\psi}(\eta)\equiv\overline{\xi}_{*}$,
and $\Omega(\overline{\psi})$ be a rectangle:
\begin{align*}
\Omega(\overline{\psi})\defs & \set{\left(\xi,\eta\right):\ \overline{\psi}(\eta)<\xi<L,\ 0<\eta<1},\\
\Gamma(\overline{\psi})\defs & \set{\left(\xi,\eta\right):\ \xi=\overline{\psi}(\eta),\ 0<\eta<1}.
\end{align*}
For the state $U=(p,\theta,q,S)^{\top}$ of the fluid on $\Omega(\overline{\psi})$,
define 
\begin{align*}
\norm{U}_{(\Omega(\overline{\psi});\Gamma(\overline{\psi}))}\defs & \norm{p}_{W_{\beta}^{1}(\Omega(\overline{\psi}))}+\norm{\theta}_{W_{\beta}^{1}(\Omega(\overline{\psi}))}\\
 & +\norm{\left(q,S\right)}_{\mcc(\overline{\Omega(\overline{\psi})})}+\norm{\left(q,S\right)}_{W_{\beta}^{1-1/\beta}(\Gamma(\overline{\psi}))}.
\end{align*}
Since the shock-front is a free boundary, the domain 
\[
\Omega(\psi)\defs\set{\left(\xi,\eta\right):\ \psi(\eta)<\xi<L,\ 0<\eta<1}
\]
for the subsonic flow behind the shock-front needs to be transformed
to the fixed domain $\Omega(\overline{\psi})$. Suppose that $\norm{\psi-\overline{\psi}}_{\mcc^{1,\alpha}([0,1])}$
is small such that the transform $\Pi_{\psi}:\ \Omega(\psi)\sTo\Omega(\overline{\psi})$
defined below is invertible:
\begin{equation}
\Pi_{\psi}:\ \begin{cases}
\tilde{\xi}=L+\frac{L-\overline{\psi}(\eta)}{L-\psi(\eta)}\cdot(\xi-L),\\
\tilde{\eta}=\eta.
\end{cases}\label{eq:coordinate_tansform_fix_bdry}
\end{equation}
It is obvious that the free boundary 
\[
\Gamma(\psi)\defs\set{\left(\xi,\eta\right):\ \xi=\psi(\eta),\ 0<\eta<1}
\]
of $\Omega(\psi)$ becomes $\Gamma(\overline{\psi})$ under the transform
$\Pi_{\psi}$. Then define the norm of $U$ on the domain $\Omega(\psi)$
as
\[
\norm{U}_{(\Omega(\psi);\Gamma(\psi))}\defs\norm{U\circ\Pi_{\psi}^{-1}}_{(\Omega(\overline{\psi});\Gamma(\overline{\psi}))}.
\]
 
\end{enumerate}
\end{rem}


Let $\left(\overline{U}_{-};\overline{U}_{+}\right)$ be a normal
shock solution. Then the main theorem proved in this paper is as below.
\begin{thm}
\label{thm:main_thm_Lagrange}

Suppose that $\overline{\xi}_{*}\in(0,L)$ satisfies \eqref{eq:first_approx_criterion_equation-1}
and $\Theta(\overline{\xi}_{*})\not=0$. Then there exists a sufficiently
small constant $\sigma_{0}>0$, depending on $\overline{U}_{-}$,
$\overline{U}_{+}$, $L$, $\overline{\xi}_{*}$, and ${\displaystyle \frac{1}{\abs{\Theta(\overline{\xi}_{*})}}}$,
such that for any $0<\sigma<\sigma_{0}$, there exists a transonic
shock solution $\left(U_{-};U_{+};\psi\right)$ to the free boundary
problem $\llbracket NSL\rrbracket$, which satisfies the following
estimates, with $0<\alpha<1$, and $\beta>2$:
\begin{align}
 & \abs{\psi(1)-\overline{\xi}_{*}}\leq C_{s}\sigma,\qquad\norm{\psi'}_{W_{\beta}^{1-\frac{1}{\beta}}(\Gamma_{\mr{s}})}\leq C_{s}\sigma,\label{eq:solution_Lagrange_estimates_shock}\\
 & \norm{U_{-}-\overline{U}_{-}}_{\mcc^{2,\alpha}(\Omega_{-})}\leq C_{s}^{-}\sigma,\label{eq:solution_Lagrange_estimates_super}\\
 & \norm{U_{+}-\overline{U}_{+}}_{(\Omega_{+};\Gamma_{\mr{s}})}\leq C_{s}^{+}\sigma,\label{eq:solution_Lagrange_estimates_sub}
\end{align}
where $C_{s}$, $C_{s}^{-}$ and $C_{s}^{+}$ are constants depending
on $\overline{U}_{-}$, $\overline{U}_{+}$, $L$, $\overline{\xi}_{*}$,
and ${\displaystyle \frac{1}{\abs{\Theta(\overline{\xi}_{*})}}}$.

Furthermore, let $\left(\dot{U}_{-};\dot{U}_{+};\dot{\psi}'|\ \overline{\xi}_{*}\right)$
be a solution to the free boundary problem $\llbracket FL\rrbracket$.
Then it holds that:
\begin{align}
 & \norm{\psi'-\dot{\psi}'}_{W_{\beta}^{1-\frac{1}{\beta}}(\Gamma_{\mr{s}})}\leq\frac{1}{2}\sigma^{3/2},\label{eq:solution_Lagrange_estimates_shock-1}\\
 & \norm{U_{-}-\left(\overline{U}_{-}+\dot{U}_{-}\right)}_{\mcc^{1,\alpha}(\Omega)}\leq\frac{1}{2}\sigma^{3/2},\label{eq:solution_Lagrange_estimates_super-1}\\
 & \norm{U_{+}\circ\Pi_{\psi}^{-1}-\left(\overline{U}_{+}+\dot{U}_{+}\right)}_{(\dot{\Omega}_{+};\dot{\Gamma}_{\mr{s}})}\leq\frac{1}{2}\sigma^{3/2}.\label{eq:solution_Lagrange_estimates_sub-1}
\end{align}
\end{thm}
Two remarks are in order.
\begin{rem}
\label{rem:main_thm_L_regularity}

The regularity of the subsonic flow field behind the shock-front can
be raised to $\mcc^{1,\alpha}$ locally at any point in $\Omega_{+}$
except its corners. Indeed, it has been showed in \cite{XinYanYin2011ARMA}
that it cannot be $\mcc^{1}$ in general at the intersection point
between the shock-front and the nozzle wall. By applying similar arguments
as in \cite{LiXinYin2013ARMA}, the global regularity can be improved to be Lipschitz continuous. The regularity of the shock-front $\Gamma_{\mr{s}}$
can also be improved to be $\mcc^{2,\alpha}$ locally except two intersection
points between the shock-front and the nozzle walls.
\end{rem}
\medskip{}

\begin{rem}
\label{rem:multi_shock_solu}

For a expanding nozzle or a contracting nozzle ( see Figure \ref{fig:Nozzle_Div})
, as \eqref{eq:criterion_range_pressure_exit-1} holds, only one shock
solution can be established via the argument in this paper because
there exists only one solution $\overline{\xi}_{*}$ to the equation \eqref{eq:first_approx_criterion_equation-1}.
Nevertheless, the global uniqueness of such shock solution has not been proved. 

Moreover, for a general nozzle with both expanding and contracting
portions, there may exist multiple admissible shock solutions, all
satisfying the R-H conditions and the entropy conditions, depending
on the geometry of the nozzle wall ( see Figure \ref{fig:Nozzle_Wav}).
Indeed, if $\varphi_{\mr{w}}$ is not monotone, there may exists more
than one solutions $\overline{\xi}_{*}$ to the equation \eqref{eq:first_approx_criterion_equation-1}
with $\Theta(\overline{\xi}_{*})\not=0$, then for each such a solution $\overline{\xi}_{*}$,
there exists a shock solution to the free boundary problem $\llbracket NSL\rrbracket $
with the shock-front close to $\overline{\xi}_{*}$. A typical example
is, for any $k=1,2,\cdots$, let
\[
\Theta\left(x\right)\defs\Theta_{k}\left(x\right)=\sin^{2}\left(\frac{k\pi}{L}x\right),\qquad0<\xi<L,
\]
for which there exist $2k$ solutions $x_{*}^{j}\in(0,L),\ (j=1,2,\cdots,2k)$
to the equation \eqref{eq:first_approx_criterion_equation-1}, as
long as the given pressure function $P$ satisfies the range condition
\eqref{eq:criterion_range_pressure_exit-1}. Then for each $x_{*}^{j}$,
as long as $\Theta_{k}\left(x_{*}^{j}\right)\not=0$, an admissible
shock solution, satisfying the entropy condition, can be established
such that the shock-front $F_{\mr{s}}\defs\set{x=\varphi_{\mr{s}}(y)}$
lies near $x_{*}^{j}$:
\begin{equation}
\abs{\varphi_{\mr{s}}(Y_{\mr{s}})-x_{*}^{j}}\leq C_{*}^{j}\sigma,\label{eq:shock_front_location_estimate_multi}
\end{equation}
where $Y_{\mr{s}}$ is a solution to the following system: 
\[
\begin{cases}
y=\varphi_{\mr{w}}\left(x\right),\\
x=\varphi_{\mr{s}}\left(y\right),
\end{cases}
\]
and the constant $C_{*}^{j}$ depends on $\overline{U}_{-}$, $\overline{U}_{+}$,
$L$, and ${\displaystyle \frac{1}{\abs{\Theta(x_{*}^{j})}}}$. This implies that, for general nozzles, there would be no global uniqueness of the shock
solutions to the free boundary problem $\llbracket NS\rrbracket $. %
\end{rem}

\section{The Initial Approximating Location of the Shock-Front}

In this section, the linearized free boundary problem $\llbracket FL\rrbracket$
will be solved and the existence of the initial approximating locations
of the shock-front will be established.

\subsection{The solution $\dot{U}_{-}$ in $\Omega$}

In the free boundary problem $\llbracket FL\rrbracket$, it is easy to solve $\dot{U}_{-}$ in $\Omega$
since it satisfies a hyperbolic system with constant efficients. And
one obtains immediately the unique existence of the solutions by employing
the classical theory.
\begin{lem}
\label{lem:first_approx_super}

There exists a unique solution $\dot{U}_{-}$ which satisfies the
linearized equations \eqref{eq:first_approx_super_1}—\eqref{eq:first_approx_super_4}
in $\Omega$, the boundary conditions \eqref{eq:first_approx_super_bc-1}
on $\Gamma_{1}$, \eqref{eq:first_approx_super_bc-2} on $\Gamma_{2}$,
and \eqref{eq:first_approx_super_bc-4} on $\Gamma_{4}$. 

Moreover, in the whole domain $\Omega$, it holds that
\begin{align}
 & \dot{S}_{-}=0,\label{eq:first_approx_super_solution_1}\\
 & \dot{p}_{-}+\overline{\rho}_{-}\overline{q}_{-}\dot{q}_{-}=0,\label{eq:first_approx_super_solution_2}
\end{align}
and
\begin{eqnarray}
\norm{\dot{U}_{-}}_{\C^{2,\alpha}(\Omega)} & \leq & \dot{C}_{L}\norm{\dot{\Theta}_{N}}_{\C^{2,\alpha}(\Gamma_{4})}\leq\dot{C}_{L}\sigma,\label{eq:first_approx_super_estimate}
\end{eqnarray}
where $\dot{C}_{L}$ is a constant depending on $\overline{U}_{-}$
and $L$.

Finally, for any $\xi\in\left(0,L\right)$,
\begin{equation}
\frac{1}{\overline{\rho}_{-}\overline{q}_{-}}\cdot\frac{1-\overline{M}_{-}^{2}}{\overline{\rho}_{-}\overline{q}_{-}^{2}}\cdot\int_{0}^{1}\dot{p}_{-}\left(\xi,\eta\right)\dif\eta=\int_{0}^{\xi}\dot{\Theta}_{N}\left(\tau\right)\dif\tau=\sigma\int_{0}^{\xi}\Theta\left(\tau\right)\dif\tau.\label{eq:first_approx_super_int_p}
\end{equation}
\end{lem}
\begin{proof}
It suffices to prove \eqref{eq:first_approx_super_int_p}.

By \eqref{eq:first_approx_super_2}, there exists a function $\phi_{-}$,
such that
\[
\partial_{\xi}\phi_{-}=\dot{\theta}_{-},\qquad\partial_{\eta}\phi_{-}=\frac{1}{\overline{\rho}_{-}\overline{q}_{-}}\cdot\frac{1-\overline{M}_{-}^{2}}{\overline{\rho}_{-}\overline{q}_{-}^{2}}\cdot\dot{p}_{-}.
\]
Without loss of generality, one may assume $\phi_{-}\left(0,0\right)=0$.
Then the boundary conditions \eqref{eq:first_approx_super_bc-1} yield
\begin{equation}
\phi_{-}\left(0,\eta\right)=0,\qquad\partial_{\xi}\phi_{-}\left(0,\eta\right)=0.\label{eq:first_approx_super_wave_initial}
\end{equation}
Moreover, the boundary conditions \eqref{eq:first_approx_super_bc-2}
and \eqref{eq:first_approx_super_bc-4} become
\begin{align}
 & \phi_{-}=0, & \text{ on } & \Gamma_{2}\label{eq:first_approx_super_wave_bdry_lower}\\
 & \phi_{-}=\int_{0}^{\xi}\dot{\Theta}_{N}\left(\tau\right)\dif\tau. & \text{ on } & \Gamma_{4}\label{eq:first_approx_super_wave_bdry_upper}
\end{align}
Hence, for any $\xi\in\left(0,L\right)$, 
\begin{eqnarray*}
\frac{1}{\overline{\rho}_{-}\overline{q}_{-}}\cdot\frac{1-\overline{M}_{-}^{2}}{\overline{\rho}_{-}\overline{q}_{-}^{2}}\cdot\int_{0}^{1}\dot{p}_{-}\left(\xi,\eta\right)\dif\eta & = & \int_{0}^{1}\partial_{\eta}\phi_{-}\left(\xi,\eta\right)\dif\eta\\
 & = & \phi_{-}\left(\xi,1\right)-\phi_{-}\left(\xi,0\right)\\
 & = & \int_{0}^{\xi}\dot{\Theta}_{N}\left(\tau\right)\dif\tau,
\end{eqnarray*}
that is, \eqref{eq:first_approx_super_int_p} holds.
\end{proof}

\subsection{Reformulation of the linearized boundary conditions \eqref{eq:first_approx_RH_1}}

With $\dot{U}_{-}$ being determined, one then need to determine $\dot{U}_{+}$
and $\overline{\xi}_{*}$ with the prescribed information in the free boundary problem $\llbracket FL\rrbracket$. 
By Lemma \ref{lem:first_approx_RH_reformulated}, the linearized boundary conditions \eqref{eq:first_approx_RH_1} can be rewritten as
\begin{equation}
B_{s}\cdot\begin{bmatrix}\dot{p}_{+}\\
\dot{q}_{+}\\
\dot{S}_{+}
\end{bmatrix}=\begin{bmatrix}\dot{g}_{1}\\
\dot{g}_{2}\\
\dot{g}_{3}
\end{bmatrix},\label{eq:first_approx_RH_matrix}
\end{equation}
where $\dot{g}_{j}\defs-\beta_{j}^{-}\cdot\dot{U}_{-}$, $(j=1,2,3)$,
and the coefficient matrix is given by
\[
B_{s}\defs\frac{1}{\overline{\rho}_{+}\overline{q}_{+}}\left[\overline{p}\right]\cdot
\begin{bmatrix}-\displaystyle\frac{1}{\overline{\rho}_{+}\overline{c}_{+}^{2}} & -\displaystyle\frac{1}{\overline{q}_{+}} & \displaystyle\frac{1}{\gamma c_{\mr{v}}}\\
1-\displaystyle\frac{\overline{p}_{+}}{\overline{\rho}_{+}\overline{c}_{+}^{2}} & \overline{\rho}_{+}\overline{q}_{+}-\displaystyle\frac{\overline{p}_{+}}{\overline{q}_{+}} & \displaystyle\frac{\overline{p}_{+}}{\gamma c_{\mr{v}}}\\
\displaystyle\frac{\overline{q}_{+}}{\left[\overline{p}\right]} & \displaystyle\frac{\overline{\rho}_{+}\overline{q}_{+}^{2}}{\left[\overline{p}\right]} & \displaystyle\frac{1}{\left(\gamma-1\right)c_{\mr{v}}}\cdot\overline{p}_{+}\cdot\frac{\overline{q}_{+}}{\left[\overline{p}\right]}
\end{bmatrix}.
\]
Clearly, the following lemma holds true.
\begin{lem}
\label{lem:first_approx_RH_reformulated}

It follows from the boundary conditions \eqref{eq:first_approx_RH_matrix} that
\begin{eqnarray}
\det B_{s} & = & \frac{1}{\left(\gamma-1\right)c_{\mr{v}}}\cdot\frac{\left[\overline{p}\right]^{2}\cdot\overline{p}_{+}}{\left(\overline{\rho}_{+}\overline{q}_{+}\right)^{3}}\cdot\left(1-\overline{M}_{+}^{2}\right)\not=0,\qquad\text{as }\overline{M}_{+}\not=1.\label{eq:first_approx_RH_det}\\
\dot{p}_{+} & = & \dot{g}_{1}^{\#}\defs\frac{\overline{\rho}_{+}\overline{q}_{+}^{2}}{\overline{M}_{+}^{2}-1}\cdot\frac{\overline{M}_{-}^{2}-1}{\overline{\rho}_{-}\overline{q}_{-}^{2}}\cdot\dot{p}_{-}\left(1-\dot{K}\right),\label{eq:first_approx_p}\\
\dot{q}_{+} & = & \dot{g}_{2}^{\#}\defs\frac{\overline{M}_{-}^{2}-1}{\overline{\rho}_{-}\overline{q}_{-}^{2}}\cdot\dot{p}_{-}\left\{ \left[\overline{p}\right]-\frac{\overline{\rho}_{+}\overline{q}_{+}^{2}}{\overline{M}_{+}^{2}-1}\left(1-\dot{K}\right)\right\} ,\label{eq:first_approx_q}\\
\dot{S}_{+} & = & \dot{g}_{3}^{\#}\defs-\frac{\left(\gamma-1\right)c_{\mr{v}}}{\overline{p}_{+}}\cdot\frac{\overline{M}_{-}^{2}-1}{\overline{\rho}_{-}\overline{q}_{-}^{2}}\cdot\dot{p}_{-}\left[\overline{p}\right].\label{eq:first_approx_S}
\end{eqnarray}
with $\dot{K}\defs\left[\overline{p}\right]\cdot\left(\displaystyle\frac{\gamma-1}{\gamma\overline{p}_{+}}+\frac{1}{\overline{\rho}_{+}\overline{q}_{+}^{2}}\right)$.
\end{lem}
\begin{proof}
By Lemma \ref{lem:first_approx_RH_coefficients}, the boundary conditions
\eqref{eq:first_approx_RH_1} for $j=1$ can be written as 
\begin{align*}
 & \frac{1}{\overline{\rho}_{+}^{2}\overline{c}_{+}^{2}\overline{q}_{+}}\dot{p}_{+}+\frac{1}{\overline{\rho}_{+}\overline{q}_{+}^{2}}\dot{q}_{+}-\frac{1}{\gamma c_{\mr{v}}}\cdot\frac{1}{\overline{\rho}_{+}\overline{q}_{+}}\dot{S}_{+}\\
= & \frac{1}{\overline{\rho}_{-}^{2}\overline{c}_{-}^{2}\overline{q}_{-}}\dot{p}_{-}+\frac{1}{\overline{\rho}_{-}\overline{q}_{-}^{2}}\dot{q}_{-}-\frac{1}{\gamma c_{\mr{v}}}\cdot\frac{1}{\overline{\rho}_{-}\overline{q}_{-}}\dot{S}_{-},
\end{align*}
which yield, by  $\overline{\rho}_{+}\overline{q}_{+}=\overline{\rho}_{-}\overline{q}_{-}$,
\eqref{eq:first_approx_super_solution_1}, and \eqref{eq:first_approx_super_solution_2},
that
\begin{eqnarray}
\frac{1}{\overline{\rho}_{+}\overline{c}_{+}^{2}}\dot{p}_{+}+\frac{1}{\overline{q}_{+}}\dot{q}_{+}-\frac{1}{\gamma c_{\mr{v}}}\dot{S}_{+} & = & \frac{1}{\overline{\rho}_{-}\overline{c}_{-}^{2}}\dot{p}_{-}+\frac{1}{\overline{q}_{-}}\dot{q}_{-}\nonumber \\
 & = & \frac{1}{\overline{\rho}_{-}\overline{q}_{-}^{2}}\left(\overline{M}_{-}^{2}-1\right)\dot{p}_{-}.\label{eq:first_approx_RH_reduced_1}
\end{eqnarray}

Similarly, the boundary condition \eqref{eq:first_approx_RH_1} for $j=2$ leads to
\begin{align*}
 & \left(\dot{p}_{+}+\overline{\rho}_{+}\overline{q}_{+}\dot{q}_{+}\right)-\overline{p}_{+}\left(\frac{1}{\overline{\rho}_{+}\overline{c}_{+}^{2}}\dot{p}_{+}+\frac{1}{\overline{q}_{+}}\dot{q}_{+}-\frac{1}{\gamma c_{\mr{v}}}\dot{S}_{+}\right)\\
= & \left(\dot{p}_{-}+\overline{\rho}_{-}\overline{q}_{-}\dot{q}_{-}\right)-\overline{p}_{-}\left(\frac{1}{\overline{\rho}_{-}\overline{c}_{-}^{2}}\dot{p}_{-}+\frac{1}{\overline{q}_{-}}\dot{q}_{-}\right).
\end{align*}
Then, combining \eqref{eq:first_approx_super_solution_2} with \eqref{eq:first_approx_RH_reduced_1} yields
\begin{equation}
\dot{p}_{+}+\overline{\rho}_{+}\overline{q}_{+}\dot{q}_{+}=\frac{1}{\overline{\rho}_{-}\overline{q}_{-}^{2}}\left(\overline{M}_{-}^{2}-1\right)\dot{p}_{-}\cdot\left(\overline{p}_{+}-\overline{p}_{-}\right).\label{eq:first_approx_RH_reduced_2}
\end{equation}

The boundary condition \eqref{eq:first_approx_RH_1} for $j=3$ reads
\begin{align*}
 & \frac{1}{\overline{\rho}_{+}}\dot{p}_{+}+\overline{q}_{+}\dot{q}_{+}+\frac{1}{\left(\gamma-1\right)c_{\mr{v}}}\cdot\frac{\overline{p}_{+}}{\overline{\rho}_{+}}\dot{S}_{+}\\
= & \frac{1}{\overline{\rho}_{-}}\dot{p}_{-}+\overline{q}_{-}\dot{q}_{-}+\frac{1}{\left(\gamma-1\right)c_{\mr{v}}}\cdot\frac{\overline{p}_{-}}{\overline{\rho}_{-}}\dot{S}_{-},
\end{align*}
which yields, by \eqref{eq:first_approx_super_solution_1}
and \eqref{eq:first_approx_super_solution_2} again, that
\[
\left(\dot{p}_{+}+\overline{\rho}_{+}\overline{q}_{+}\dot{q}_{+}\right)+\frac{1}{\left(\gamma-1\right)c_{\mr{v}}}\cdot\overline{p}_{+}\cdot\dot{S}_{+}=\frac{\overline{\rho}_{+}}{\overline{\rho}_{-}}\left(\dot{p}_{-}+\overline{\rho}_{-}\overline{q}_{-}\dot{q}_{-}\right)=0.
\]
Hence, this and \eqref{eq:first_approx_RH_reduced_2} yield that
\begin{equation}
\frac{1}{c_{\mr{v}}}\cdot\dot{S}_{+}=-\left(\gamma-1\right)\cdot\frac{\overline{M}_{-}^{2}-1}{\overline{\rho}_{-}\overline{q}_{-}^{2}}\cdot\frac{\left[\overline{p}\right]}{\overline{p}_{+}}\dot{p}_{-},\label{eq:first_approx_RH_reduced_3}
\end{equation}
which is \eqref{eq:first_approx_S} exactly. 

Then, substituting \eqref{eq:first_approx_RH_reduced_3} into \eqref{eq:first_approx_RH_reduced_1}
gives
\begin{eqnarray*}
\frac{1}{\overline{\rho}_{+}\overline{c}_{+}^{2}}\dot{p}_{+}+\frac{1}{\overline{q}_{+}}\dot{q}_{+} & = & \frac{1}{\gamma c_{\mr{v}}}\dot{S}_{+}+\frac{1}{\overline{\rho}_{-}\overline{q}_{-}^{2}}\left(\overline{M}_{-}^{2}-1\right)\cdot\dot{p}_{-}\\
 & = & \left(1-\frac{\gamma-1}{\gamma}\cdot\frac{\left[\overline{p}\right]}{\overline{p}_{+}}\right)\cdot\frac{\overline{M}_{-}^{2}-1}{\overline{\rho}_{-}\overline{q}_{-}^{2}}\cdot\dot{p}_{-},
\end{eqnarray*}
which implies
\[
\frac{1}{\overline{\rho}_{+}\overline{q}_{+}^{2}}\left(\overline{M}_{+}^{2}\dot{p}_{+}+\overline{\rho}_{+}\overline{q}_{+}\dot{q}_{+}\right)=\left(1-\frac{\gamma-1}{\gamma}\cdot\frac{\left[\overline{p}\right]}{\overline{p}_{+}}\right)\cdot\frac{\overline{M}_{-}^{2}-1}{\overline{\rho}_{-}\overline{q}_{-}^{2}}\cdot\dot{p}_{-}.
\]
This, together with \eqref{eq:first_approx_RH_reduced_2}, gives
\begin{equation}
\frac{\overline{M}_{+}^{2}-1}{\overline{\rho}_{+}\overline{q}_{+}^{2}}\cdot\dot{p}_{+}=\frac{\overline{M}_{-}^{2}-1}{\overline{\rho}_{-}\overline{q}_{-}^{2}}\cdot\dot{p}_{-}\left\{ 1-\left[\overline{p}\right]\cdot\left(\frac{\gamma-1}{\gamma\overline{p}_{+}}+\frac{1}{\overline{\rho}_{+}\overline{q}_{+}^{2}}\right)\right\} ,\label{eq:first_approx_shock}
\end{equation}
which yields \eqref{eq:first_approx_p}.

Finally, substituting \eqref{eq:first_approx_p} into \eqref{eq:first_approx_RH_reduced_2}
shows \eqref{eq:first_approx_q}. Thus the lemma is proved.
\end{proof}

\subsection{Determine $\overline{\xi}_{*}$ and $\dot{U}_{+}$}

\textcolor{magenta}{}

Now we determine $\overline{\xi}_{*}$ and $\dot{U}_{+}$. It follows from
\eqref{eq:first_approx_sub_3} and \eqref{eq:first_approx_sub_4} that
\begin{align}
 & \left(\overline{q}_{+}\dot{q}_{+}+\frac{1}{\overline{\rho}_{+}}\dot{p}_{+}+\overline{T}_{+}\dot{S}_{+}\right)\Big|_{\left(\xi,\eta\right)}=\left(\overline{q}_{+}\dot{q}_{+}+\frac{1}{\overline{\rho}_{+}}\dot{p}_{+}+\overline{T}_{+}\dot{S}_{+}\right)\Big|_{\left(\overline{\xi}_{*},\eta\right)},\label{eq:first_approx_sub_hyper_3}\\
 & \dot{S}_{+}\left(\xi,\eta\right)=\dot{S}_{+}\left(\overline{\xi}_{*},\eta\right).\label{eq:first_approx_sub_hyper_4}
\end{align}
Thus, by Lemma \ref{lem:first_approx_RH_reformulated},
the left hand sides of \eqref{eq:first_approx_sub_hyper_3} and \eqref{eq:first_approx_sub_hyper_4}
are determined once $\overline{\xi}_{*}$ is known. Therefore, the key
step is determining $\overline{\xi}_{*}$ and $\left(\dot{p}_{+},\dot{\theta}_{+}\right)$
via the elliptic system of first order consisting in the equations
\eqref{eq:first_approx_sub_1}—\eqref{eq:first_approx_sub_2} as well
as the boundary conditions \eqref{eq:first_approx_sub_bc-3}—\eqref{eq:first_approx_sub_bc-4},
and \eqref{eq:first_approx_p}. Then Corollary \ref{lem:bvp_first_order_elliptic}
can be employed to yield the following lemma.
\begin{lem}
\label{lem:first_approx_criterion}

Given $\overline{\xi}_{*}\in(0,L)$, there exists a unique solution $\left(\dot{p}_{+},\dot{\theta}_{+}\right)$
to the boundary value problem consisting of the equations \eqref{eq:first_approx_sub_1}—\eqref{eq:first_approx_sub_2}
together with the boundary conditions \eqref{eq:first_approx_sub_bc-3}—\eqref{eq:first_approx_sub_bc-4},
and \eqref{eq:first_approx_p}, if and only if
\begin{equation}
\frac{1}{\overline{\rho}_{+}\overline{q}_{+}}\cdot\frac{1-\overline{M}_{+}^{2}}{\overline{\rho}_{+}\overline{q}_{+}^{2}}\cdot\int_{0}^{1}P\left(\eta\right)\dif\eta=\int_{0}^{L}\Theta\left(\xi\right)\dif\xi-\dot{K}\int_{0}^{\overline{\xi}_{*}}\Theta\left(\xi\right)\dif\xi,\label{eq:first_approx_criterion_exit_pressure}
\end{equation}
where 
\[
\dot{K}=\left[\overline{p}\right]\cdot\left(\frac{\gamma-1}{\gamma\overline{p}_{+}}+\frac{1}{\overline{\rho}_{+}\overline{q}_{+}^{2}}\right)>0.
\]
\end{lem}
\begin{proof}
Due to Corollary \ref{lem:bvp_first_order_elliptic}, it suffices to show
that the solvability condition \eqref{eq:first_order_elliptic_exist_cond}
yields \eqref{eq:first_approx_criterion_exit_pressure}. Actually,
for the problem of the system \eqref{eq:first_approx_sub_1}—\eqref{eq:first_approx_sub_2}
with the boundary conditions \eqref{eq:first_approx_sub_bc-3}—\eqref{eq:first_approx_sub_bc-4},
and \eqref{eq:first_approx_p}, the solvability condition \eqref{eq:first_order_elliptic_exist_cond}
reads
\begin{equation}
0=\frac{1}{\overline{\rho}_{+}\overline{q}_{+}}\cdot\frac{1-\overline{M}_{+}^{2}}{\overline{\rho}_{+}\overline{q}_{+}^{2}}\int_{0}^{1}\left(\dot{p}_{+}\left(\overline{\xi}_{*},\eta\right)-\dot{P}_{e}\left(\eta\right)\right)\dif\eta+\int_{\overline{\xi}_{*}}^{L}\dot{\Theta}_{N}^{+}\left(\xi\right)\dif\xi,\label{eq:criterion_compatibility_cond}
\end{equation}
By \eqref{eq:first_approx_p}, and recalling that $\overline{\rho}_{+}\overline{q}_{+}=\overline{\rho}_{-}\overline{q}_{-}$,
one gets that
\begin{eqnarray*}
\frac{1}{\overline{\rho}_{+}\overline{q}_{+}}\cdot\frac{1-\overline{M}_{+}^{2}}{\overline{\rho}_{+}\overline{q}_{+}^{2}}\int_{0}^{1}\dot{p}_{+}\left(\overline{\xi}_{*},\eta\right)\dif\eta & = & \left(1-\dot{K}\right)\frac{1}{\overline{\rho}_{-}\overline{q}_{-}}\cdot\frac{1-\overline{M}_{-}^{2}}{\overline{\rho}_{-}\overline{q}_{-}^{2}}\cdot\int_{0}^{1}\dot{p}_{-}\left(\overline{\xi}_{*},\eta\right)\dif\eta\\
 & = & \left(1-\dot{K}\right)\int_{0}^{\overline{\xi}_{*}}\sigma\Theta\left(\tau\right)\dif\tau,
\end{eqnarray*}
where \eqref{eq:first_approx_super_int_p} has been used for the second
equality. Therefore, \eqref{eq:criterion_compatibility_cond} yields,
since $\dot{P}_{e}\left(\eta\right)=\sigma P\left(\eta\right)$ and
$\dot{\Theta}_{N}^{+}\left(\xi\right)=\sigma\Theta\left(\xi\right)$,
that
\[
\frac{1}{\overline{\rho}_{+}\overline{q}_{+}}\cdot\frac{1-\overline{M}_{+}^{2}}{\overline{\rho}_{+}\overline{q}_{+}^{2}}\int_{0}^{1}P\left(\eta\right)\dif\eta=\left(1-\dot{K}\right)\int_{0}^{\overline{\xi}_{*}}\Theta\left(\tau\right)\dif\tau+\int_{\overline{\xi}_{*}}^{L}\Theta\left(\xi\right)\dif\xi,
\]
which is the condition \eqref{eq:first_approx_criterion_exit_pressure} exactly.
\end{proof}
\begin{rem}
By Lemma \ref{lem:first_approx_criterion}, for the existence of the
solution $\left(\dot{p}_{+},\dot{\theta}_{+}\right)$, the unknown
constant $\overline{\xi}_{*}$, which determines the location of the free
boundary $\dot{\Gamma}_{s}$, should satisfy \eqref{eq:first_approx_criterion_exit_pressure}.
Hence, the value of $\overline{\xi}_{*}$ should be determined by the shape
of the nozzle described by $\Theta$ and the pressure $P$ prescribed
at the exit via the condition \eqref{eq:first_approx_criterion_exit_pressure}.
For any $\xi\in(0,L)$, define 
\begin{equation}
R\left(\xi\right)\defs\int_{0}^{L}\Theta\left(\tau\right)\dif\tau-\dot{K}\int_{0}^{\xi}\Theta\left(\tau\right)\dif\tau,\label{eq:first_approx_criterion_nozzle}
\end{equation}
and denote its infimum and supremum respectively by 
\begin{eqnarray*}
\underline{R} & \defs & \inf_{\xi\in\left(0,L\right)}R\left(\xi\right),\\
\overline{R} & \defs & \sup_{\xi\in\left(0,L\right)}R\left(\xi\right).
\end{eqnarray*}
Then the condition \eqref{eq:first_approx_criterion_exit_pressure}
can be regarded as an equation for $\overline{\xi}_{*}$ as
\begin{equation}
R(\overline{\xi}_{*})=\dot{P}_{*}\defs\frac{1}{\overline{\rho}_{+}\overline{q}_{+}}\cdot\frac{1-\overline{M}_{+}^{2}}{\overline{\rho}_{+}\overline{q}_{+}^{2}}\cdot\int_{0}^{1}P\left(\eta\right)\dif\eta,\label{eq:first_approx_criterion_equation}
\end{equation}
and obviously, it has at least a solution $\overline{\xi}_{*}\in[0,L]$
if and only if
\begin{equation}
\underline{R}\leq\dot{P}_{*}\leq\overline{R}.\label{eq:criterion_range_pressure}
\end{equation}
And as \eqref{eq:criterion_range_pressure} holds, any solution $\overline{\xi}_{*}\in(0,L)$
to the equation \eqref{eq:first_approx_criterion_equation} with
\begin{equation}
R'\left(\overline{\xi}_{*}\right)=-\dot{K}\Theta\left(\overline{\xi}_{*}\right)\not=0,\label{eq:first_approx_location_ncd}
\end{equation}
 can be taken as an initial approximating location of the shock-front
which will lead to a shock solution via a nonlinear iteration scheme as shown later.
\end{rem}
\begin{figure}[th]
\centering
\def\svgwidth{200pt}
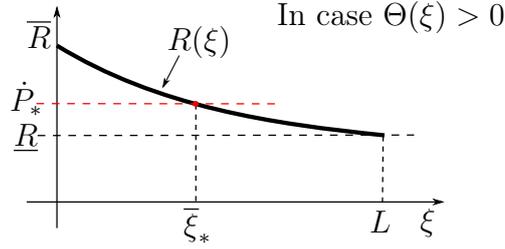

\caption{The initial approximating location of the shock-front for a expanding
finite nozzle.\label{fig:Criterion_Div}}
\end{figure}

\begin{lem}
\label{lem:first_approx_location_monotone}

Assume that for any $\xi\in(0,L)$, 
\begin{equation}
\Theta\left(\xi\right)\not=0,\qquad\text{ for any }\xi\in(0,L),\label{eq:nozzle_monotone}
\end{equation}
 which means that $\varphi_{\mr{w}}$ is a strictly monotone function
in $(0,L)$, and the nozzle is either expanding or contracting. Then
there exists a unique solution $\overline{\xi}_{*}\in(0,L)$ to the equation
\eqref{eq:first_approx_criterion_exit_pressure}, provided that \eqref{eq:criterion_range_pressure}
holds. Furthermore, the solution $\overline{\xi}_{*}$ satisfies \eqref{eq:first_approx_location_ncd}.
\end{lem}
\begin{proof}
Due to \eqref{eq:nozzle_monotone}, then it holds that either 
\begin{equation}
\Theta\left(\xi\right)>0,\qquad\text{for all }\xi\in\left(0,L\right),\label{eq:diverging_nozzle}
\end{equation}
such that the nozzle is expanding, or
\begin{equation}
\Theta\left(\xi\right)<0,\qquad\text{for all }\xi\in\left(0,L\right),\label{eq:converging_nozzle}
\end{equation}
such that the nozzle is contracting.

Note that
\[
R'\left(\xi\right)=-\dot{K}\Theta\left(\xi\right).
\]
Hence, in the case \eqref{eq:diverging_nozzle}, $R\left(\xi\right)$
is strictly decreasing, and 
\begin{eqnarray*}
\underline{R} & = & R\left(L\right)=\left(1-\dot{K}\right)\int_{0}^{L}\Theta\left(\xi\right)\dif\xi,\\
\overline{R} & = & R\left(0\right)=\int_{0}^{L}\Theta\left(\xi\right)\dif\xi.
\end{eqnarray*}
Moreover, for any $P(\eta)$ satisfying 
\begin{equation}
\underline{R}<\dot{P}_{*}<\overline{R},\label{eq:criterion_range_pressure_monotone}
\end{equation}
it is then clear that there exists a unique $\overline{\xi}_{*}\in(0,L)$
such that
\[
R\left(\overline{\xi}_{*}\right)=\dot{P}_{*}.
\]
Also, 
\[
R'\left(\overline{\xi}_{*}\right)=-\dot{K}\Theta\left(\overline{\xi}_{*}\right)<0.
\]

Similarly, in case that \eqref{eq:converging_nozzle} holds, $R\left(\xi\right)$
is strictly increasing, and 
\begin{eqnarray*}
\underline{R} & = & R\left(0\right)=\int_{0}^{L}\Theta\left(\xi\right)\dif\xi,\\
\overline{R} & = & R\left(L\right)=\left(1-\dot{K}\right)\int_{0}^{L}\Theta\left(\xi\right)\dif\xi.
\end{eqnarray*}
Then for any $P(\eta)$ satisfying \eqref{eq:criterion_range_pressure_monotone},
there exists a unique $\overline{\xi}_{*}\in(0,L)$ such that
\[
R\left(\overline{\xi}_{*}\right)=\dot{P}_{*},
\]
with 
\[
R'\left(\overline{\xi}_{*}\right)=-\dot{K}\Theta\left(\overline{\xi}_{*}\right)>0.
\]
\end{proof}
\begin{figure}[th]
\centering
\def\svgwidth{200pt}
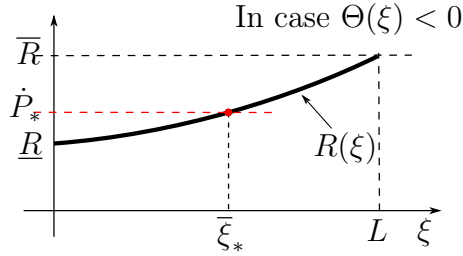

\caption{The initial approximating location of the shock-front for a contracting
finite nozzle.\label{fig:Criterion_Con}}
\end{figure}

Then, with $\overline{\xi}_{*}$ being so determined, by employing Corollary
\ref{lem:bvp_first_order_elliptic}, one can detemine $\dot{U}_{+}$
in $\dot{\Omega}_{+}$ immediately. 
\begin{lem}
\label{lem:first_approx_sub}

Assume that \eqref{eq:criterion_range_pressure} and \eqref{eq:diverging_nozzle}
hold.
Then there exists a unique solution $\dot{U}_{+}$ satisfying the
equations \eqref{eq:first_approx_sub_1}—\eqref{eq:first_approx_sub_4}
in $\dot{\Omega}_{+}$ as well as the boundary conditions \eqref{eq:first_approx_sub_bc-3}—\eqref{eq:first_approx_sub_bc-4},
and \eqref{eq:first_approx_p}—\eqref{eq:first_approx_S}. Furthermore, it holds that
\begin{eqnarray}
\norm{\dot{U}_{+}}_{(\dot{\Omega}_{+};\dot{\Gamma}_{\mr{s}})} & \leq & \dot{C}_{+}\left(\norm{\dot{\Theta}_{N}}_{\C^{2,\alpha}(\Gamma_{4})}+\norm{\dot{P}_{e}}_{\C^{2,\alpha}(\Gamma_{3})}\right)\leq\dot{C}_{+}\sigma,\label{eq:first_approx_sub_estimate}
\end{eqnarray}
with $\beta>2$, and the constant $\dot{C}_{+}$ depending on $\overline{U}_{+}$ and
$L$.
\end{lem}
\begin{proof}
By Corollary \ref{lem:bvp_first_order_elliptic}, there exists a unique solution $\left(\dot{p}_{+},\dot{\theta}_{+}\right)$
to the boundary value problem consisting of the equations \eqref{eq:first_approx_sub_1}—\eqref{eq:first_approx_sub_2}
together with the boundary conditions \eqref{eq:first_approx_sub_bc-3}—\eqref{eq:first_approx_sub_bc-4},
and \eqref{eq:first_approx_p}. Moreover, it holds that
\begin{equation}
\norm{\dot{p}_{+}}_{W_{\beta}^{1}(\dot{\Omega}_{+})}+\norm{\dot{\theta}_{+}}_{W_{\beta}^{1}(\dot{\Omega}_{+})}\le C\left\{ \norm{\dot{\Theta}_{N}}_{\C^{2,\alpha}(\Gamma_{4})}+\norm{\dot{P}_{e}}_{\C^{2,\alpha}(\Gamma_{3})}+\norm{\dot{g}_{1}^{\#}}_{W_{\beta}^{1-1/\beta}(\dot{\Gamma}_{s})}\right\} ,\label{eq:first_approx_sub_pw}
\end{equation}
where the constant $C$ depends on $\overline{U}_{+}$ and $L$. Then
the lemma follows from \eqref{eq:first_approx_sub_hyper_3}, \eqref{eq:first_approx_sub_hyper_4}
and Lemma \ref{lem:first_approx_super} - \ref{lem:first_approx_RH_reformulated}.
\end{proof}

\subsection{Existence of solutions to the problem $\llbracket FL\rrbracket$ and
some remarks.}

As a consequence of the above arguments, we obtain immediately the existence of the unique solution to problem \textbf{$\llbracket FL\rrbracket$} as follows:
\begin{thm}
\label{thm:first_approx_existence}

Assume that \eqref{eq:criterion_range_pressure} and \eqref{eq:diverging_nozzle}
hold. Then there exists a unique solution $\left(\dot{U}_{-},\ \dot{U}_{+};\dot{\psi}'|\ \overline{\xi}_{*}\right)$
to the problem $\llbracket FL\rrbracket$, where $\overline{\xi}_{*}$,
the free boundary $\dot{\Gamma}_{\mr{s}}$, is determined by the equation
\eqref{eq:first_approx_criterion_exit_pressure}, as in Lemma \ref{lem:first_approx_location_monotone}.
Moreover, it holds that
\begin{eqnarray}
\norm{\dot{U}_{-}}_{\C^{2,\alpha}(\dot{\Omega}_{-})} & \leq & C_{L}\norm{\dot{\Theta}_{N}^{-}}_{\C^{2,\alpha}(\dot{\Gamma}_{4}^{-})},\label{eq:first_approx_super_estimate-1}\\
\norm{\dot{U}_{+}}_{(\dot{\Omega}_{+};\dot{\Gamma}_{s})} & \leq & \dot{C}_{+}\left(\norm{\dot{\Theta}_{N}^{+}}_{\C^{2,\alpha}(\dot{\Gamma}_{4}^{+})}+\norm{\dot{P}_{e}}_{\C^{2,\alpha}(\Gamma_{3})}\right),\label{eq:first_approx_sub_estimate-1}\\
\norm{\dot{\psi}'}_{W_{\beta}^{1-1/\beta}(\dot{\Gamma}_{s})} & \leq & \dot{C}_{+}\left(\norm{\dot{\Theta}_{N}}_{\C^{2,\alpha}(\Gamma_{4})}+\norm{\dot{P}_{e}}_{\C^{2,\alpha}(\Gamma_{3})}\right).\label{eq:first_approx_front_estimate}
\end{eqnarray}
\end{thm}
The solution $\left(\dot{U}_{-},\ \dot{U}_{+};\dot{\psi}'|\ \overline{\xi}_{*}\right)$
will lead to an initial approximating shock solution. Then one is able to carry out a nonlinear iteration scheme to obtain an admissible
shock solution to the free boundary problem $\llbracket NSL\rrbracket$
near $\left(\overline{U}_{-} + \dot{U}_{-},\ \overline{U}_{+} + \dot{U}_{+};\dot{\psi}'|\ \overline{\xi}_{*}\right)$.
\begin{rem}
\label{rem:first_approx_multi_solu}

In the case \eqref{eq:nozzle_monotone} that does not hold, there may exist
multiple solutions for $\overline{\xi}_{*}$ to the equation \eqref{eq:first_approx_criterion_equation},
provided \eqref{eq:criterion_range_pressure} holds ( see Figure \ref{fig:Criterion_Gen}).
As an example, for any $k=1,2,\cdots$, if
\[
\Theta\left(\xi\right)\defs\Theta_{k}\left(\xi\right)=\sin^{2}\left(\frac{k\pi}{L}\xi\right),\qquad0<\xi<L,
\]
then there exist $2k$ solutions $\overline{\xi}_{*}^{j}(j=1,2,\cdots,2k)$
to the equation \eqref{eq:first_approx_criterion_equation}, as long
as the given pressure function $P$ satisfies the range condition
\eqref{eq:criterion_range_pressure}. For each $\overline{\xi}_{*}^{j}$,
there exists a unique solution $\left(\dot{U}_{-},\ \dot{U}_{+}^{j};(\dot{\psi}^{j})'|\ \overline{\xi}_{*}^{j}\right)$
to the problem $\llbracket FL\rrbracket$ and
it satisfies the estimates \eqref{eq:first_approx_super_estimate-1}–\eqref{eq:first_approx_front_estimate}.
Then each solution $\left(\dot{U}_{-},\ \dot{U}_{+}^{j};(\dot{\psi}^{j})'|\ \overline{\xi}_{*}^{j}\right)$
can lead to an initial approximating shock solution, which will
give an admissible shock solution by the iteration arguments in
the following sections. This fact results in the non-uniqueness of
the shock solutions to the problem $\llbracket NSL\rrbracket$.
\end{rem}
\begin{figure}[th]
\centering
\def\svgwidth{200pt}
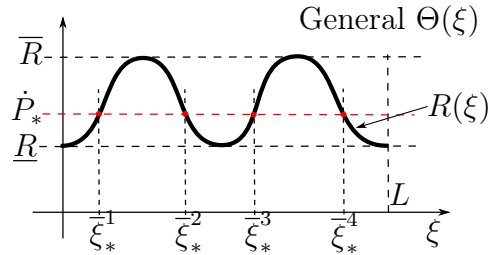

\caption{Multiple initial approximating locations of the shock-fronts for a
general finite nozzle.\label{fig:Criterion_Gen}}
\end{figure}

\textcolor{blue}{}

\section{The Nonlinear Iteration Scheme and the Linearized Problem}

\textcolor{red}{}

The free boundary $\dot{\Gamma}_{s}$ determined in the problem
$\llbracket FL\rrbracket$ will be taken as the initial approximating
location of the shock-front and a nonlinear iteration scheme will
be carried out to determine the shock solution to the problem $\llbracket NSL\rrbracket$. 

\subsection{The supersonic flow in the whole nozzle.}

\textcolor{magenta}{}

Since the Euler system is hyperbolic in supersonic regions, so
applying the theory in \cite{LiYu1985} by Li and Yu yields the supersonic flow field in the nozzle as follows: 
\begin{thm} 
 \label{thm:supersonic_flow}

Suppose that \eqref{eq:upper_bdry_compatibility_2nd}
holds. Then there exists a positive constant $\sigma_{L}$ depending
on $\overline{U}_{-}$ and $L$, such that for any $0<\sigma\leq\sigma_{L}$,
the initial-boundary value problem
\begin{align*}
 & \text{The Euler system \eqref{eq:Euler_Lagrange_matrix_1}---\eqref{eq:Euler_Lagrange_matrix_4},} & \text{ in } & \Omega\\
 & U_{-}\left(0,\eta\right)=\overline{U}_{-}, & \text{ on } & \Gamma_{1}\\
 & \theta_{-}=0, & \text{ on } & \Gamma_{2}\\
 & \theta_{-}=\sigma\Theta\left(\xi\right), & \text{ on } & \Gamma_{4}
\end{align*}
has a unique solution $U_{-}\in\C^{2,\alpha}\left(\overline{\Omega}\right)$ such that
\begin{equation}
\norm{U_{-}-\overline{U}_{-}}_{\C^{2,\alpha}\left(\overline{\Omega}\right)}\leq C_{L}\sigma.\label{eq:estimate_perturbation_supersonic}
\end{equation}
Finally, let $\delta U_{-}\defs U_{-}-\overline{U}_{-}$, then there
exists a constant $\hat{C}_{L}$, depending on $L$ and $\overline{U}_{-}$,
such that
\begin{equation}
\norm{\delta U_{-}-\dot{U}_{-}}_{\C^{1,\alpha}\left(\overline{\Omega}\right)}\leq C_{\overline{U}_{-}}\norm{\partial_{\xi}\delta U_{-}}_{\C^{1,\alpha}\left(\overline{\Omega}\right)}\cdot\norm{\delta U_{-}}_{\C^{1,\alpha}\left(\overline{\Omega}\right)}\leq\hat{C}_{L}\sigma^{2}.\label{eq:estimate_supersonic_difference}
\end{equation}
\end{thm}
\begin{proof}
The existence of the unique solution $U_{-}\in\mcc^{2,\alpha}\left(\overline{\Omega}\right)$
can be obtained by applying the theory in \cite{LiYu1985}. Thus, it suffices to verify \eqref{eq:estimate_supersonic_difference}.

For the simplicity of notations, we rewrite the Euler system \eqref{eq:Euler_Lagrange_matrix_1}—\eqref{eq:Euler_Lagrange_matrix_4} as
\begin{equation}
\mca_{0}\partial_{\eta}U+\mca_{1}(U)\partial_{\xi}U=0,\label{eq:Euler_Lagrange_matrix}
\end{equation}
where
\[
\mca_{0}=\begin{bmatrix}1 & 0 & 0 & 0\\
0 & 1 & 0 & 0\\
0 & 0 & 0 & 0\\
0 & 0 & 0 & 0
\end{bmatrix},\qquad\mca_{1}\left(U\right)=\begin{bmatrix}-\frac{\sin\theta}{\rho q} & q\cos\theta & 0 & 0\\
-\frac{\cos\theta}{\rho q}\cdot\frac{1-M^{2}}{\rho q^{2}} & -\frac{\sin\theta}{\rho q} & 0 & 0\\
1 & 0 & \rho q & 0\\
0 & 0 & 0 & 1
\end{bmatrix}.
\]
Then $U_{-}\in\mcc^{2,\alpha}\left(\overline{\Omega}\right)$ solves
the following initial-boundary vlue problem:
\begin{align*}
 & \mca_{0}\partial_{\eta}U_{-}+\mca_{1}(U_{-})\partial_{\xi}U_{-}=0, & \text{ in } & \Omega\\
 & U_{-}=\overline{U}_{-}, & \text{ on } & \Gamma_{1}\\
 & \theta_{-}=0, & \text{ on } & \Gamma_{2}\\
 & \theta_{-}=\sigma\Theta\left(\xi\right). & \text{ on } & \Gamma_{4}
\end{align*}
Therefore, $\delta U_{-}-\dot{U}_{-}$ satisfies the following problem:
\begin{align*}
 & \mca_{0}\partial_{\eta}\left(\delta U_{-}-\dot{U}_{-}\right)+\mca_{1}\left(\overline{U}_{-}\right)\partial_{\xi}\left(\delta U_{-}-\dot{U}_{-}\right)=\mcf_{-}\left(\delta U_{-}\right), & \text{ in } & \Omega\\
 & \left(\delta U_{-}-\dot{U}_{-}\right)=0, & \text{ on } & \Gamma_{1}\\
 & \delta\theta_{-}-\dot{\theta}_{-}=0, & \text{ on } & \Gamma_{2}\\
 & \delta\theta_{-}-\dot{\theta}_{-}=0, & \text{ on } & \Gamma_{4}
\end{align*}
where 
\[
\mcf_{-}\left(\delta U_{-}\right)\defs\left(\mca_{1}\left(\overline{U}_{-}\right)-\mca_{1}\left(U_{-}\right)\right)\partial_{\xi}\delta U_{-}+\mca_{0}\partial_{\eta}\delta U_{-}.
\]
Hence, there exists a constant $C_{\overline{U}_{-}}$ depending on
$\overline{U}_{-}$, such that
\begin{eqnarray*}
\norm{\delta U_{-}-\dot{U}_{-}}_{\C^{1,\alpha}\left(\overline{\Omega}\right)} & \leq & C_{\overline{U}_{-}}\norm{\mcf_{-}\left(\delta U_{-}\right)}_{\C^{1,\alpha}\left(\overline{\Omega}\right)}\\
 & \leq & C_{\overline{U}_{-}}\norm{\partial_{\xi}\delta U_{-}}_{\C^{1,\alpha}\left(\overline{\Omega}\right)}\cdot\norm{\delta U_{-}}_{\C^{1,\alpha}\left(\overline{\Omega}\right)},
\end{eqnarray*}
which proves \eqref{eq:estimate_supersonic_difference}.
\end{proof}

\subsection{The shock-front and subsonic flow behind it}

Now we assume that, for the given pressure $P_{e}$ at the exit, there
appears a shock-front $\Gamma_{\mr{s}}$ whose location is close to
$\dot{\Gamma}_{s}$:
\[
\Gamma_{\mr{s}}\defs\set{\left(\xi,\eta\right):\ \xi=\psi\left(\eta\right)\defs\overline{\xi}_{*}+\delta\psi\left(\eta\right),\ 0<\eta<1}.
\]
Then the region of the subsonic flow behind it is 
\[
\Omega_{+}=\set{\left(\xi,\eta\right):\ \psi\left(\eta\right)<\xi<L,\ 0<\eta<1},
\]
and the state of the flow $U_{+}$ is assumed to be close to $\overline{U}_{+}$.
Then $\left(U_{+};\psi\right)$ solves the following free boundary
value problem $\llbracket NSL-Sub\rrbracket$:
\begin{align*}
 & \mca_{0}\partial_{\eta}U_{+}+\mca_{1}(U_{+})\partial_{\xi}U_{+}=0, & \text{ in } & \Omega_{+}\\
 & \text{The R-H conditions \eqref{eq:RH_Lagrange_1}---\eqref{eq:RH_Lagrange_4_bdry}}, & \text{ on } & \Gamma_{\mr{s}}\\
 & \theta_{+}=0, & \text{ on } & \Gamma_{2}\\
 & p_{+}=P_{e}\left(Y\left(L,\eta\right)\right), & \text{ on } & \Gamma_{3}\\
 & \theta_{+}=\sigma\Theta\left(\xi\right), & \text{ on } & \Gamma_{4}
\end{align*}
where, in the R-H conditions on the shock-front $\Gamma_{\mr{s}}$,
$U_{-}$ are given by the supersonic flow determined in Theorem \ref{thm:supersonic_flow}.
Thus, the key step here is to solve the free boundary value problem $\llbracket NSL-Sub\rrbracket$
near $\left(\overline{U}_{+};\overline{\psi}\right)$. It should be noted here that
the free boundary $\psi$ will be determined by two independent
information: the shape of the shock-front $\psi'$ will be determined
by the R-H conditions, and an exact position on the nozzle $\xi_{*}\defs\psi(1)$
where the shock-front passes through will be determined by the solvability
condition for the existence of the solution $ {U}_{+} $. Therefore, we shall rewrite
$\psi$ as below:
\[
\psi(\eta)=\overline{\xi}_{*}+\delta\xi_{*}-\int_{\eta}^{1}\delta\psi'(s)\dif s,
\]
where $\delta\xi_{*}\defs\xi_{*}-\overline{\xi}_{*}=\psi(1)-\overline{\xi}_{*}$.

\begin{figure}[th]
\centering
\def\svgwidth{200pt}
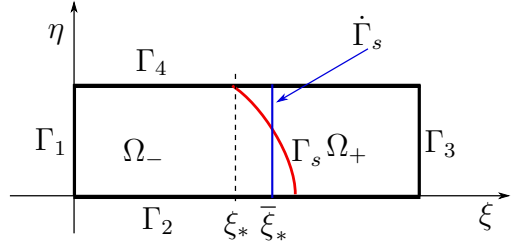

\caption{The shock-front and its initial approximating location.\label{fig:Shock_front}}
\end{figure}

First, introduce a transformation:
\[
\Pi_{\psi}:\ \begin{cases}
\tilde{\xi}=L+\displaystyle\frac{L-\overline{\xi}_{*}}{L-\psi(\eta)}(\xi-L),\\
\tilde{\eta}=\eta,
\end{cases}\ \text{with its inverse }\Pi_{\psi}^{-1}:\ \begin{cases}
\xi=L+\displaystyle\frac{L-\psi(\tilde{\eta})}{L-\overline{\xi}_{*}}(\tilde{\xi}-L),\\
\eta=\tilde{\eta},
\end{cases}
\]
under which the region $\Omega_{+}$, $\Gamma_{\mr{s}}$, and $\Gamma_{j}(j=2,3,4)$ become, respectively, 
\begin{eqnarray*}
\tilde{\Omega} & \defs & \set{\left(\tilde{\xi},\tilde{\eta}\right):\ \overline{\xi}_{*}<\tilde{\xi}<L,\ 0<\tilde{\eta}<1}, \\
\tilde{\Gamma}_{s} & \defs & \set{\left(\tilde{\xi},\tilde{\eta}\right):\ \tilde{\xi}=\overline{\xi}_{*},\ 0<\tilde{\eta}<1}, \\
\tilde{\Gamma}_{2} & \defs & \set{\left(\tilde{\xi},\tilde{\eta}\right):\ \overline{\xi}_{*}<\tilde{\xi}<L,\ \tilde{\eta}=0},\\
\tilde{\Gamma}_{3} & \defs & \set{\left(\tilde{\xi},\tilde{\eta}\right):\ \tilde{\xi}=L,\ 0<\tilde{\eta}<1},\\
\tilde{\Gamma}_{4} & \defs & \set{\left(\tilde{\xi},\tilde{\eta}\right):\ \overline{\xi}_{*}<\tilde{\xi}<L,\ \tilde{\eta}=1}.
\end{eqnarray*}

Let $\tilde{U}(\tilde{\xi},\tilde{\eta})\defs U_{+}\circ\Pi_{\psi}^{-1}(\tilde{\xi},\tilde{\eta})$.
Then $\tilde{U}$ satisfies the following equations in $\tilde{\Omega}$:
\begin{equation}
\mca_{0}\partial_{\tilde{\eta}}\tilde{U}+\mca_{1}(\tilde{U})\partial_{\tilde{\xi}}\tilde{U}=-\frac{L-\tilde{\xi}}{L-\psi(\tilde{\eta})}\cdot\psi'(\tilde{\eta})\mca_{0}\partial_{\tilde{\xi}}\tilde{U}+\frac{\psi(\tilde{\eta})-\overline{\xi}_{*}}{L-\psi(\tilde{\eta})}\mca_{1}(\tilde{U})\partial_{\tilde{\xi}}\tilde{U}.\label{eq:fix_bdry_system}
\end{equation}
Moreover, the boundary conditions on $\Gamma_{j}(j=2,3,4)$ become
\begin{align}
 & \tilde{\theta}=0, & \text{ on } & \tilde{\Gamma}_{2}\label{eq:fix_bdry_bc_2}\\
 & \tilde{p}=P_{e}\left(Y\left(L,\tilde{\eta}\right)\right), & \text{ on } & \tilde{\Gamma}_{3}\label{eq:fix_bdry_bc_3}\\
 & \tilde{\theta}=\sigma\Theta\circ\Pi_{\psi}^{-1}(\tilde{\xi},1), & \text{ on } & \tilde{\Gamma}_{4}\label{eq:fix_bdry_bc_4}
\end{align}
and the R-H conditions \eqref{eq:RH_Lagrange_1}—\eqref{eq:RH_Lagrange_4_bdry}
become
\begin{align}
 & G_{j}\left(\tilde{U},U_{-}^{(\psi',\xi_{*})}\right)=0,\qquad j=1,2,3 & \text{ on } & \tilde{\Gamma}_{s}\label{eq:fix_bdry_bc_RH_1}\\
 & G_{4}\left(\tilde{U},U_{-}^{(\psi',\xi_{*})};\psi'\right)=0, & \text{ on } & \tilde{\Gamma}_{s}\label{eq:fix_bdry_bc_RH_2}
\end{align}
where 
\[
U_{-}^{(\psi',\xi_{*})}\defs U_{-}\left(\psi(\tilde{\eta}),\tilde{\eta}\right).
\]

Thus, the free boundary value problem $\llbracket NSL-Sub \rrbracket$
is reduced to a fixed boundary value problem \eqref{eq:fix_bdry_system}—\eqref{eq:fix_bdry_bc_RH_2}.
Then we shall employ a nonlinear iteration scheme to prove the existence
of its solutions.

For simplicity of the notations, we shall drop `` $\tilde{}$ ''
in the following argument.

\subsection{The linearized problem for the iteration}

In this subsection, we shall describe the linearized problem for the
nonlinear iteration scheme which will be used to prove the existence
of the solution to the problem \eqref{eq:fix_bdry_system}—\eqref{eq:fix_bdry_bc_RH_2}.
Since the location of the shock-front $ \Gamma_{\mr{s}}$ is expected to be close to the initial approximation $ \dot{\Gamma}_{s} $ obtained in Section 3, and the state $ U $ of the subsonic flow behind the shock-front is expected to be perturbation of the state $ \overline{U}_{+} $, the iteration scheme would be designed as follows, in which the right-hand sides for the perturbation terms were quantities of the order $ \sigma^{2} $.

\medskip{}

\textbf{The iteration scheme. }

Let $U=\overline{U}_{+}+\delta U$ be a given approximating state with $\delta U\defs(\delta p,\ \delta\theta,\ \delta q,\ \delta S)^{\top}$ close to $\dot{U}_{+}$, and let $\psi(\eta)=\overline{\psi}(\eta) + \delta\psi(\eta)$ be an approximating location of the shock-front with
\begin{equation}
\psi(\eta)=\psi(1)-\int_{\eta}^{1}\delta\psi'(\tau)\dif\tau=\xi_{*}-\int_{\eta}^{1}\delta\psi'(\tau)\dif\tau,\label{eq:linearized_shock_front_updating}
\end{equation}
where $\delta\psi'$ is a given function close to $ \dot{\psi}' $ while $ \psi(1)=\xi_{*}\defs \overline{\xi}_{*}+\delta\xi_{*}$ is unknown.

Then try to determine $ \delta\xi_{*} $ and update $ (\delta U;\ \delta\psi') $ by  $ (\delta U^{*};\ (\delta\psi^{*})') $ such that
\begin{enumerate}
\item $\delta U^{*}\defs(\delta p^{*},\ \delta\theta^{*},\ \delta q^{*},\ \delta S^{*})^{\top}$ satisfies the following linearized Euler equations
in $\Omega$:
\begin{align}
 & \partial_{\eta}\delta p^{*}+\overline{q}_{+}\partial_{\xi}\delta\theta^{*}=f_{1}\left(\delta U;\delta\psi',\delta\xi_{*}\right),\label{eq:linearized_system_sub_elliptic_1}\\
 & \partial_{\eta}\delta\theta^{*}-\frac{1}{\overline{\rho}_{+}\overline{q}_{+}}\cdot\frac{1-\overline{M}_{+}^{2}}{\overline{\rho}_{+}\overline{q}_{+}^{2}}\partial_{\xi}\delta p^{*}=f_{2}\left(\delta U;\delta\psi',\delta\xi_{*}\right),\label{eq:linearized_system_sub_elliptic_2}\\
 & \partial_{\xi}\left(\overline{q}_{+}\delta q^{*}+\frac{1}{\overline{\rho}_{+}}\delta p^{*}+\overline{T}_{+}\delta S^{*}\right)=\partial_{\xi}f_{3}\left(\delta U\right),\label{eq:linearized_system_sub_hyper_1}\\
 & \partial_{\xi}\delta S^{*}=0,\label{eq:linearized_system_sub_hyper_2}
\end{align}
where
\begin{eqnarray*}
f_{1}\left(\delta U;\delta\psi',\delta\xi_{*}\right) & \defs & \left(\partial_{\eta}\delta p+\overline{q}_{+}\partial_{\xi}\delta\theta\right)-\left(\partial_{\eta}p-\frac{\sin\theta}{\rho q}\partial_{\xi}p+q\cos\theta\partial_{\xi}\theta\right)\\
 &  & -\frac{L-\xi}{L-\psi(\eta)}\cdot\delta\psi'(\eta)\partial_{\xi}p\\
 &  & +\frac{\delta\xi_{*}-\int_{\eta}^{1}\delta\psi'(\tau)\dif\tau}{L-\psi(\eta)}\left(\frac{\sin\theta}{\rho q}\partial_{\xi}p-q\cos\theta\partial_{\xi}\theta\right),\\
f_{2}\left(\delta U;\delta\psi',\delta\xi_{*}\right) & \defs & \left(\partial_{\eta}\delta\theta-\frac{1}{\overline{\rho}_{+}\overline{q}_{+}}\cdot\frac{1-\overline{M}_{+}^{2}}{\overline{\rho}_{+}\overline{q}_{+}^{2}}\partial_{\xi}\delta p\right)\\
 &  & -\left(\partial_{\eta}\theta-\frac{\sin\theta}{\rho q}\partial_{\xi}\theta-\frac{\cos\theta}{\rho q}\cdot\frac{1-M^{2}}{\rho q^{2}}\partial_{\xi}p\right)\\
 &  & -\frac{L-\xi}{L-\psi(\eta)}\cdot\delta\psi'(\eta)\partial_{\xi}\theta\\
 &  & +\frac{\delta\xi_{*}-\int_{\eta}^{1}\delta\psi'(\tau)\dif\tau}{L-\psi(\eta)}\left(\frac{\sin\theta}{\rho q}\partial_{\xi}\theta+\frac{\cos\theta}{\rho q}\cdot\frac{1-M^{2}}{\rho q^{2}}\partial_{\xi}p\right),\\
f_{3}\left(\delta U\right) & \defs & \left(\overline{q}_{+}\delta q+\frac{1}{\overline{\rho}_{+}}\delta p+\overline{T}_{+}\delta S\right)-\Phi\left(U\right).
\end{eqnarray*}
\item On the nozzle boundaries $\Gamma_{2}$ and $\Gamma_{4}$,
\begin{align}
 & \delta\theta^{*}=0, & \text{ on } & \Gamma_{2}\label{eq:linearized_bc_nozzle_2}\\
 & \delta\theta^{*}=\delta\Theta_{4}\left(\xi;\delta\xi_{*}\right), & \text{ on } & \Gamma_{4}\label{eq:linearized_bc_nozzle_4}
\end{align}
where 
\[
\delta\Theta_{4}\left(\xi;\delta\xi_{*}\right)\defs\sigma\Theta\circ\Pi_{\psi}^{-1}(\xi,1)=\sigma\Theta\left(\frac{L-\xi_{*}}{L-\overline{\xi}_{*}}\xi+\frac{\delta\xi_{*}}{L-\overline{\xi}_{*}}L\right).
\]
\item On the exit of the nozzle $\Gamma_{3}$,
\begin{equation}
\delta p^{*}=\delta P_{3}(\eta;\delta U),\qquad\text{ on }\Gamma_{3}\label{eq:linearized_bc_nozzle_3}
\end{equation}
where
\begin{eqnarray*}
\delta P_{3}(\eta;\delta U) & \defs & \sigma P\left(Y\left(L,\eta;\delta U\right)\right),\\
Y\left(L,\eta;\delta U\right) & \defs & \int_{0}^{\eta}\frac{1}{\left(\rho q\cos\theta\right)(L,s)}\dif s.
\end{eqnarray*}
\item On the fixed shock-front $\Gamma_{\mr{s}}$, the R-H conditions are
linearized as
\begin{align}
 & \beta_{j}^{+}\cdot\delta U^{*}=g_{j}\left(\delta U,\delta U_{-};\delta\psi',\delta\xi_{*}\right),\quad j=1,2,3, & \text{ on } & \Gamma_{\mr{s}}\label{eq:linearized_bc_RH_1}\\
 & \beta_{4}^{+}\cdot\delta U^{*}-\left[\overline{p}\right](\delta\psi^{*})'=g_{4}\left(\delta U,\delta U_{-};\delta\psi',\delta\xi_{*}\right), & \text{ on } & \Gamma_{\mr{s}}\label{eq:linearized_bc_RH_2}
\end{align}
where 
\begin{eqnarray*}
g_{j}\left(\delta U,\delta U_{-};\delta\psi',\delta\xi_{*}\right) & \defs & \beta_{j}^{+}\cdot\delta U-G_{j}\left(U,U_{-}^{(\psi',\xi_{*})}\right),\qquad j=1,2,3\\
g_{4}\left(\delta U,\delta U_{-};\delta\psi',\delta\xi_{*}\right) & \defs & \left(\beta_{4}^{+}\cdot\delta U-\left[\overline{p}\right]\delta\psi'\right)-G_{4}\left(U,U_{-}^{(\psi',\xi_{*})};\psi'\right).
\end{eqnarray*}
\end{enumerate}
\medskip{}

Obviously, one needs to find an appropriate set for $\left(\delta U;\delta\psi'\right)$
such that $\delta\xi_{*}$ can be determined, and the iteration mapping
\[
\mcj_{s}:\left(\delta U;\delta\psi'\right)\mapsto\left(\delta U^{*};(\delta\psi^{*})'\right)
\]
would be well-defined and contractive. The crucial step is to establish
the well-posed theory for the linearized problem \eqref{eq:linearized_system_sub_elliptic_1}—\eqref{eq:linearized_bc_RH_2}.
\begin{defn}
\label{def:linearized_problem_operator}

For simplicity of notations, define the solution $ \left(\delta U^{*};\ (\delta\psi^{*})'\ |\delta\xi_{*} \right) $ to the linearized problem \eqref{eq:linearized_system_sub_elliptic_1}-\eqref{eq:linearized_bc_RH_2} near $ \left(\dot{U}_{+};\ \dot{\psi}'\ |\  0\right) $
as an operator:
\[
\left(\delta U^{*};\ (\delta\psi^{*})'\ |\delta\xi_{*} \right) = \mcp_{a}\left(\mcf;\mcg;\delta P_{3};\delta\Theta_{4}\right),
\]
where $\mcf\defs\left(f_{1},f_{2},f_{3}\right)$, and $\mcg\defs\left(g_{1},g_{2},g_{3},g_{4}\right)$. When $ \delta\xi_{*} $ is omitted, it will also be denoted by 
\[
\left(\delta U^{*};\ (\delta\psi^{*})'\right) = \mcp\left(\mcf;\mcg;\delta P_{3};\delta\Theta_{4}\right).
\]
\end{defn}
Since $ \det B_{s} \not= 0 $, the boundary conditions \eqref{eq:linearized_bc_RH_1} can be rewritten
as
\begin{eqnarray}
\delta p^{*} & = & g_{1}^{\sharp},\label{eq:linearized_bc_RH_reform_p}\\
\delta q^{*} & = & g_{2}^{\sharp},\label{eq:linearized_bc_RH_reform_w}\\
\delta S^{*} & = & g_{3}^{\sharp},\label{eq:linearized_bc_RH_reform_S}
\end{eqnarray}
where $g_{j}^{\sharp}=g_{j}^{\sharp}\left(\delta U,\delta U_{-};\delta\psi',\delta\xi_{*}\right)$,
$(j=1,2,3)$, is determined by 
\[
\begin{bmatrix}g_{1}^{\sharp}\\
g_{2}^{\sharp}\\
g_{3}^{\sharp}
\end{bmatrix}=B_{s}^{-1}\begin{bmatrix}g_{1}\\
g_{2}\\
g_{3}
\end{bmatrix}.
\]
Then the equations \eqref{eq:linearized_system_sub_elliptic_1}—\eqref{eq:linearized_system_sub_elliptic_2}
with the boundary conditions \eqref{eq:linearized_bc_nozzle_2}—\eqref{eq:linearized_bc_nozzle_3},
and \eqref{eq:linearized_bc_RH_reform_p} form a closed boundary value
problem of a first-order elliptic system for $\left(\delta p^{*},\delta\theta^{*}\right)^{\top}$.
By Lemma \ref{lem:bvp_first_order_elliptic}, there exists a unique
solution to this problem if
\begin{alignat}{1}
\int_{\Omega}f_{2}\left(\delta U;\delta\psi',\delta\xi_{*}\right)\dif\xi\dif\eta & =\int_{\overline{\xi}_{*}}^{L}\delta\Theta_{4}(\xi;\delta\xi_{*})\dif\xi\nonumber \\
+\frac{1}{\overline{\rho}_{+}\overline{q}_{+}}\cdot\frac{1-\overline{M}_{+}^{2}}{\overline{\rho}_{+}\overline{q}_{+}^{2}} & \int_{0}^{1}\left\{ g_{1}^{\sharp}\left(\delta U,\delta U_{-};\delta\psi',\delta\xi_{*}\right)-\delta P_{3}(\eta;\delta U)\right\} \dif\eta,\label{eq:linearized_solvability_cond}
\end{alignat}
which will be used to determine $\delta\xi_{*}$. In addition, if the condition \eqref{eq:linearized_solvability_cond} holds, the solution
$\left(\delta p^{*},\delta\theta^{*}\right)^{\top}$ satisfies the
estimate that, for any $\beta>2$,
\begin{align}
 & \norm{\delta p^{*}}_{W_{\beta}^{1}(\Omega)}+\norm{\delta\theta^{*}}_{W_{\beta}^{1}(\Omega)}\nonumber \\
\leq & C\left(\sum_{j=1}^{2}\norm{f_{j}}_{L^{\beta}(\Omega)}+\norm{g_{1}^{\sharp}}_{W^{1-\frac{1}{\beta},\beta}(\Gamma_{\mr{s}})}+\norm{\delta P_{3}}_{W^{1-\frac{1}{\beta},\beta}(\Gamma_{3})}+\norm{\delta\Theta_{4}}_{W^{1-\frac{1}{\beta},\beta}(\Gamma_{4})}\right),\label{eq:linearized_estimates_pw}
\end{align}
where the constant $C$ depends on $\beta$. It follows from \eqref{eq:linearized_system_sub_hyper_2} and \eqref{eq:linearized_bc_RH_reform_S} that
\begin{equation}
\delta S^{*}\left(\xi,\eta\right)=\delta S^{*}\left(\overline{\xi}_{*},\eta\right)=g_{3}^{\sharp}(\eta),\label{eq:linearized_solu_S}
\end{equation}
which yields the estimate
\begin{equation}
\norm{\delta S^{*}}_{W^{1-\frac{1}{\beta},\beta}(\Gamma_{\mr{s}})}+\norm{\delta S^{*}}_{\mcc^{0}(\Omega)}\leq C\norm{g_{3}^{\sharp}}_{W^{1-\frac{1}{\beta},\beta}(\Gamma_{\mr{s}})},\label{eq:linearized_estimate_S}
\end{equation}
with the constant $C$ depending also on $\beta$. Once
$\left(\delta p^{*},\delta\theta^{*},\delta S^{*}\right)^{\top}$ is obtained,
one can obtain $\delta q^{*}$ from \eqref{eq:linearized_system_sub_hyper_1}
and \eqref{eq:linearized_bc_RH_reform_w} as follows:
\begin{align}
 & \left(\overline{q}_{+}\delta q^{*}+\frac{1}{\overline{\rho}_{+}}\delta p^{*}+\overline{T}_{+}\delta S^{*}\right)\Big|_{\left(\xi,\eta\right)}\nonumber \\
= & \left(\overline{q}_{+}\delta q^{*}+\frac{1}{\overline{\rho}_{+}}\delta p^{*}+\overline{T}_{+}\delta S^{*}\right)\Big|_{\left(\overline{\xi}_{*},\eta\right)}+\left(f_{3}(\xi,\eta)-f_{3}(\overline{\xi}_{*},\eta)\right).\label{eq:linearized_solu_q}
\end{align}
Therefore, $\delta q^{*}$ can be estimated as 
\begin{align}
 & \norm{\delta q^{*}}_{W^{1-\frac{1}{\beta},\beta}(\Gamma_{\mr{s}})}+\norm{\delta q^{*}}_{\mcc^{0}(\Omega)}\label{eq:linearized_estimate_q}\\
\leq & C\left(\sum_{j=1}^{3}\norm{g_{j}^{\sharp}}_{W^{1-\frac{1}{\beta},\beta}(\Gamma_{\mr{s}})}+\norm{f_{3}-f_{3}|_{\Gamma_{\mr{s}}}}_{\mcc^{0}(\Omega)}+\norm{\delta p^{*}}_{\mcc^{0}(\Omega)}\right)\nonumber \\
\leq & C\left(\sum_{j=1}^{3}\norm{g_{j}^{\sharp}}_{W^{1-\frac{1}{\beta},\beta}(\Gamma_{\mr{s}})}+\sum_{j=1}^{2}\norm{f_{j}}_{L^{\beta}(\Omega)}+\norm{\delta P_{3}}_{W^{1-\frac{1}{\beta},\beta}(\Gamma_{3})}\right. \\
& \left.\qquad +\norm{\delta\Theta_{4}}_{W^{1-\frac{1}{\beta},\beta}(\Gamma_{4})}+\norm{f_{3}-f_{3}|_{\Gamma_{\mr{s}}}}_{\mcc^{0}(\Omega)}\right)\nonumber 
\end{align}
Finally, it follows from \eqref{eq:linearized_bc_RH_2} that
\begin{equation}
(\delta\psi^{*})'=\frac{1}{\left[\overline{p}\right]}\left(\beta_{4}^{+}\cdot\delta U^{*}-g_{4}\right),\label{eq:linearized_solu_shock_shape}
\end{equation}
which yields
\begin{equation}
\norm{(\delta\psi^{*})'}_{W^{1-\frac{1}{\beta},\beta}(\Gamma_{\mr{s}})}\leq C\left(\norm{g_{4}}_{W^{1-\frac{1}{\beta},\beta}(\Gamma_{\mr{s}})}+\norm{\delta U^{*}}_{W^{1-\frac{1}{\beta},\beta}(\Gamma_{\mr{s}})}\right).\label{eq:linearized_estimate_shock_shape}
\end{equation}
Consequently, the following theorem holds:
\begin{thm}
\label{thm:linearized_problem}

Assume that, for given $\left(\delta U;\delta\psi'\right)$, there
exists a $\delta\xi_{*}$ such that \eqref{eq:linearized_solvability_cond}
holds. Then there exists a solution $\left(\delta U^{*};(\delta\psi^{*})',\delta\xi_{*}\right)$
to the linearized problem \eqref{eq:linearized_system_sub_elliptic_1}—\eqref{eq:linearized_bc_RH_2}.

Moreover, the solution $\left(\delta U^{*};(\delta\psi^{*})'\right)$
satisfies the following estimates, with $\beta>2$:
\begin{eqnarray}
 &  & \norm{\delta p^{*}}_{W_{\beta}^{1}(\Omega)}+\norm{\delta\theta^{*}}_{W_{\beta}^{1}(\Omega)}\\
 & \leq & C\left(\sum_{j=1}^{2}\norm{f_{j}}_{L^{\beta}(\Omega)}+\norm{g_{1}^{\sharp}}_{W^{1-\frac{1}{\beta},\beta}(\Gamma_{\mr{s}})}+\norm{\delta P_{3}}_{W^{1-\frac{1}{\beta},\beta}(\Gamma_{3})}+\norm{\delta\Theta_{4}}_{W^{1-\frac{1}{\beta},\beta}(\Gamma_{4})}\right),\nonumber \\
 &  & \norm{\delta q^{*}}_{W^{1-\frac{1}{\beta},\beta}(\Gamma_{\mr{s}})}+\norm{\delta q^{*}}_{\mcc^{0}(\Omega)}+\norm{\delta S^{*}}_{W^{1-\frac{1}{\beta},\beta}(\Gamma_{\mr{s}})}+\norm{\delta S^{*}}_{\mcc^{0}(\Omega)}\\
 & \leq & C\left(\sum_{j=1}^{3}\norm{g_{j}^{\sharp}}_{W^{1-\frac{1}{\beta},\beta}(\Gamma_{\mr{s}})}+\norm{f_{3}(\xi,\eta)-f_{3}(\overline{\xi}_{*},\eta)}_{\mcc^{0}(\Omega)}+\norm{\delta p^{*}}_{\mcc^{0}(\Omega)}\right),\nonumber \\
 &  & \norm{(\delta\psi^{*})'}_{W^{1-\frac{1}{\beta},\beta}(\Gamma_{\mr{s}})} \leq  C\left(\norm{g_{4}}_{W^{1-\frac{1}{\beta},\beta}(\Gamma_{\mr{s}})}+\norm{\delta U^{*}}_{W^{1-\frac{1}{\beta},\beta}(\Gamma_{\mr{s}})}\right).
\end{eqnarray}
\end{thm}
\begin{rem}
\label{rem:linearized_first_approx}

The argument in Section 3 indicates that
\[
\left(\dot{U}_{+};\ \dot{\psi}'\ |\ 0\right) = \mcp_{a}\left(\dot{\mcf};\dot{\mcg};\dot{P}_{e};\dot{\Theta}_{N}\right),
\]
where $\dot{\mcf}\defs\left(0,0,-\Phi(\overline{U}_{+})\right)$,
and $\dot{\mcg}\defs\left(\dot{g}_{1},\dot{g}_{2},\dot{g}_{3},\dot{g}_{4}\right)$.
Moreover, $\delta\xi_{*}=0$ is a solution to the equation \eqref{eq:linearized_solvability_cond}.
\end{rem}
\bigskip{}

\section{The Well-Posedness and Contractiveness of the Iteration Mapping}


By Theorem \ref{thm:linearized_problem}, in order to carry out the
iteration scheme, one needs to show that, for given $\left(\delta U;\delta\psi'\right)$,
$\delta\xi_{*}$ can be determined by the condition \eqref{eq:linearized_solvability_cond}.
To this end, one needs to specify a suitable set for $\left(\delta U;\delta\psi'\right)$.

Let $\epsilon>0$, $\beta>2$, and define
\[
\fk(\epsilon)\defs\set{\left(\delta U;\delta\psi'\right):\ \norm{\delta U}_{(\Omega;\Gamma_{\mr{s}})}+\norm{\delta\psi'}_{W^{1-\frac{1}{\beta},\beta}(\Gamma_{\mr{s}})}\leq\epsilon}.
\]
We are going to show that $\delta\xi_{*}$ can be determined by \eqref{eq:linearized_solvability_cond}
for sufficiently small $\epsilon$. In fact, the following lemma holds.
\begin{lem}
\label{lem:iteration_approx_location}

There exists a suitably small positive constant $\sigma_{1}$ such that for
any $0<\sigma\leq\sigma_{1}$, if $\left(\delta U-\dot{U}_{+};\delta\psi'-\dot{\psi}'\right)\in\fk(\frac{1}{2}\sigma^{3/2})$,
there exists a unique solution $\delta\xi_{*}$ to the equation \eqref{eq:linearized_solvability_cond} with the estimate
\begin{equation}
\abs{\delta\xi_{*}}\leq C_{*}\sigma,\label{eq:iteration_approx_location_estimate}
\end{equation}
where the constant $C_{*}$ depends on $\overline{\xi}_{*}$ and the value
of $ \displaystyle \frac{1}{\abs{\Theta(\overline{\xi}_{*})}} $.
\end{lem}
\begin{proof}
For any $\left(\delta U-\dot{U}_{+};\delta\psi'-\dot{\psi}'\right)\in\fk(\frac{1}{2}\sigma^{3/2})$,
it follows from Theorem \ref{thm:first_approx_existence} that
\begin{equation}
\norm{\delta U}_{(\Omega;\Gamma_{\mr{s}})}+\norm{\delta\psi'}_{W^{1-\frac{1}{\beta},\beta}(\Gamma_{\mr{s}})}\leq C_{+}\sigma,\label{eq:iteration_approx_pre_states}
\end{equation}
where $C_{+}=\dot{C}_{+}+1$.

Set
\begin{align}
\mci(\delta\xi_{*};\delta U,\delta\psi';\delta U_{-}) & \defs-\int_{\Omega}f_{2}\left(\delta U;\delta\psi',\delta\xi_{*}\right)\dif\xi\dif\eta+\int_{\overline{\xi}_{*}}^{L}\delta\Theta_{4}(\xi;\delta\xi_{*})\dif\xi\nonumber \\
 & +\frac{1}{\overline{\rho}_{+}\overline{q}_{+}}\cdot\frac{1-\overline{M}_{+}^{2}}{\overline{\rho}_{+}\overline{q}_{+}^{2}}\int_{0}^{1}\left\{ g_{1}^{\sharp}\left(\delta U,\delta U_{-};\delta\psi',\delta\xi_{*}\right)-\delta P_{3}(\eta;\delta U)\right\} \dif\eta.\label{eq:iteration_mci}
\end{align}
Then it can be checked easily that 
\begin{equation}
\mci(0;0,0;\dot{U}_{-})=0.\label{eq:iteration_approx_location_bg}
\end{equation}
We claim that 
\begin{equation}
\frac{\partial\mci}{\partial(\delta\xi_{*})}(0;0,0;\dot{U}_{-})\not=0.\label{eq:iteration_approx_location_update_cond}
\end{equation}
Thus, by applying the implicit function theorem, there exists a solution
$\delta\xi_{*}$ to the equation \eqref{eq:linearized_solvability_cond}.

To prove \eqref{eq:iteration_approx_location_update_cond}, one needs to expand
$\mci$ near $(0;0,0;\dot{U}_{-})$. First, one estimates
the terms in $\mci$ one by one.

Since  $\overline{\rho}_{+}\overline{q}_{+}=1$, so
\begin{align*}
\delta P_{3}(\eta;\delta U) 
 & =\sigma P(\eta)+\sigma\set{P\left(\int_{0}^{\eta}\frac{1}{\left(\rho q\cos\theta\right)(L,s)}\dif s\right)-P\left(\int_{0}^{\eta}\frac{1}{\left(\overline{\rho}_{+}\overline{q}_{+}\right)(L,s)}\dif s\right)},
\end{align*}
which implies
\[
\delta P_{3}(\eta;\delta U)=\sigma P(\eta)+\mco(1)\sigma^{2},
\]
where $\mco(1)$ depends on $C_{+}$ and $\norm{P'}_{L^{\infty}(\Gamma_{3})}$.
Thus, 
\begin{equation}
\int_{0}^{1}\delta P_{3}(\eta;\delta U)\dif\eta=\sigma\int_{0}^{1}P(\eta)\dif\eta+\mco(1)\sigma^{2}.\label{eq:iteration_mci_P}
\end{equation}

Moreover, direct computations yield 
\begin{equation}
\int_{\overline{\xi}_{*}}^{L}\delta\Theta_{4}(\xi;\delta\xi_{*})\dif\xi=\sigma\int_{\overline{\xi}_{*}}^{L}\Theta(\zeta)\dif\zeta+\sigma\int_{\overline{\xi}_{*}+\delta\xi_{*}}^{\overline{\xi}_{*}}\Theta(\zeta)\dif\zeta+\sigma\cdot\frac{\delta\xi_{*}}{L-\xi_{*}}\int_{\overline{\xi}_{*}+\delta\xi_{*}}^{L}\Theta(\zeta)\dif\zeta.\label{eq:iteration_mci_Theta}
\end{equation}

To estimate $g_{1}^{\sharp}$, one notes for any $j=1,2,3$,
\begin{align*}
 & g_{j}(\delta U,\delta U_{-};\delta\psi',\delta\xi_{*})\\
= & \beta_{j}^{+}\cdot\delta U-G_{j}\left(U,U_{-}^{(\psi',\xi_{*})}\right)\\
= & \left(\beta_{j}^{+}\cdot\delta U+\beta_{j}^{-}\cdot\dot{U}_{-}|_{(\overline{\xi}_{*},\eta)}-G_{j}\left(U,U_{-}^{(\psi',\xi_{*})}\right)\right)+\dot{g}_{j}(\dot{U}_{-}|_{(\overline{\xi}_{*},\eta)})\\
= & \left(\beta_{j}^{+}\cdot\delta U+\beta_{j}^{-}\cdot\delta U_{-}^{(\delta\psi',\delta\xi_{*})}-G_{j}\left(U,U_{-}^{(\psi',\xi_{*})}\right)\right)\\
 & -\beta_{j}^{-}\cdot\left(\delta U_{-}^{(\delta\psi',\delta\xi_{*})}-\dot{U}_{-}|_{(\overline{\xi}_{*},\eta)}\right)+\dot{g}_{j},
\end{align*}
where $\delta U_{-}^{(\delta\psi',\delta\xi_{*})}\defs\delta U_{-}|_{\left(\overline{\xi}_{*}+\delta\xi_{*}-\int_{\eta}^{1}\delta\psi'(s)\dif s,\eta\right)}$.
Since
\begin{align*}
 & \beta_{j}^{+}\cdot\delta U+\beta_{j}^{-}\cdot\delta U_{-}^{(\delta\psi',\delta\xi_{*})}-G_{j}\left(U,U_{-}^{(\psi',\xi_{*})}\right)\\
= & \frac{1}{2}\int_{0}^{1}D^{2}G_{j}\left(\overline{U}_{+}+s\delta U;\overline{U}_{-}+s\delta U_{-}\right)\dif s\left(\delta U;\delta U_{-}\right)^{2}\\
= & \mco(1)\sigma^{2},
\end{align*}
where $\mco(1)$ depends on $C_{+}$, and 
\begin{align*}
 & \delta U_{-}^{(\delta\psi',\delta\xi_{*})}-\dot{U}_{-}|_{(\overline{\xi}_{*},\eta)}\\
= & \left(\delta U_{-}^{(\delta\psi',\delta\xi_{*})}-\dot{U}_{-}|_{\left(\overline{\xi}_{*}+\delta\xi_{*}-\int_{\eta}^{1}\delta\psi'(s)\dif s,\eta\right)}\right)\\
 & +\left(\dot{U}_{-}|_{\left(\overline{\xi}_{*}+\delta\xi_{*}-\int_{\eta}^{1}\delta\psi'(s)\dif s,\eta\right)}-\dot{U}_{-}|_{\left(\overline{\xi}_{*}+\delta\xi_{*},\eta\right)}\right)\\
 & +\left(\dot{U}_{-}|_{\left(\overline{\xi}_{*}+\delta\xi_{*},\eta\right)}-\dot{U}_{-}|_{\left(\overline{\xi}_{*},\eta\right)}\right)\\
= & \mco(1)\sigma^{2}+\left(\dot{U}_{-}|_{\left(\overline{\xi}_{*}+\delta\xi_{*},\eta\right)}-\dot{U}_{-}|_{\left(\overline{\xi}_{*},\eta\right)}\right),
\end{align*}
where $\mco(1)$ depends on $\hat{C}_{L}$, $C_{L}$ and $C_{+}$,
one gets that
\begin{align*}
 & g_{j}(\delta U,\delta U_{-};\delta\psi',\delta\xi_{*})\\
= & \mco(1)\sigma^{2}-\beta_{j}^{-}\cdot\left(\dot{U}_{-}|_{\left(\overline{\xi}_{*}+\delta\xi_{*},\eta\right)}-\dot{U}_{-}|_{\left(\overline{\xi}_{*},\eta\right)}\right)+\dot{g}_{j}(\dot{U}_{-}|_{(\overline{\xi}_{*},\eta)})\\
= & \mco(1)\sigma^{2}-\beta_{j}^{-}\cdot\dot{U}_{-}|_{\left(\overline{\xi}_{*}+\delta\xi_{*},\eta\right)}.
\end{align*}
Therefore, it holds that
\[
g_{1}^{\sharp}\left(\delta U,\delta U_{-};\delta\psi',\delta\xi_{*}\right)|_{\left(\overline{\xi}_{*},\eta\right)}=\dot{g}_{1}^{\sharp}|_{\left(\overline{\xi}_{*}+\delta\xi_{*},\eta\right)}+\mco(1)\sigma^{2},
\]
where $\mco(1)$ depends on $\hat{C}_{L}$, $C_{L}$ , $C_{+}$ and
$\norm{\Theta}_{L^{\infty}(\Gamma_{4})}$. Thus,
\begin{equation}
\int_{0}^{1}g_{1}^{\sharp}\left(\delta U,\delta U_{-};\delta\psi',\delta\xi_{*}\right)\dif\eta=\int_{0}^{1}\dot{g}_{1}^{\sharp}|_{\left(\overline{\xi}_{*}+\delta\xi_{*},\eta\right)}\dif\eta+\mco(1)\sigma^{2}.\label{eq:iteration_mci_g}
\end{equation}

It remains to estimate $f_{2}$. Note that 
\begin{align*}
 & f_{2}\left(\delta U;\delta\psi',\delta\xi_{*}\right)\\
= & \frac{\sin\theta}{\rho q}\partial_{\xi}\delta\theta+\left(\frac{\cos\theta}{\rho q}\cdot\frac{1-M^{2}}{\rho q^{2}}-\frac{1}{\overline{\rho}_{+}\overline{q}_{+}}\cdot\frac{1-\overline{M}_{+}^{2}}{\overline{\rho}_{+}\overline{q}_{+}^{2}}\right)\partial_{\xi}\delta p\\
 & -\frac{L-\xi}{L-\psi(\eta)}\cdot\delta\psi'(\eta)\partial_{\xi}\delta\theta+\frac{\delta\xi_{*}-\int_{\eta}^{1}\delta\psi'(\tau)\dif\tau}{L-\psi(\eta)}\cdot\frac{\sin\theta}{\rho q}\partial_{\xi}\delta\theta\\
 & -\frac{\int_{\eta}^{1}\delta\psi'(\tau)\dif\tau}{L-\psi(\eta)}\cdot\frac{\cos\theta}{\rho q}\cdot\frac{1-M^{2}}{\rho q^{2}}\partial_{\xi}\delta p+\frac{\delta\xi_{*}}{L-\psi(\eta)}\cdot\frac{\cos\theta}{\rho q}\cdot\frac{1-M^{2}}{\rho q^{2}}\partial_{\xi}\delta p.
\end{align*}
It follows from \eqref{eq:iteration_approx_pre_states} that
\begin{align}
 & \abs{\int_{\Omega}\frac{\sin\theta}{\rho q}\partial_{\xi}\delta\theta+\left(\frac{\cos\theta}{\rho q}\cdot\frac{1-M^{2}}{\rho q^{2}}-\frac{1}{\overline{\rho}_{+}\overline{q}_{+}}\cdot\frac{1-\overline{M}_{+}^{2}}{\overline{\rho}_{+}\overline{q}_{+}^{2}}\right)\partial_{\xi}\delta p\dif\xi\dif\eta}\nonumber \\
\leq & \mco(1)\norm{\delta U}_{W_{\beta}^{1}(\Omega)}^{2}=\mco(1)\sigma^{2},\label{eq:iteration_mci_f-1}
\end{align}
where $\mco(1)$ depends on $\beta>2$, $\overline{U}_{+}$ and $C_{+}$.
Also, it holds that
\begin{align}
 & \abs{\int_{\Omega}-\frac{L-\xi}{L-\psi(\eta)}\cdot\delta\psi'(\eta)\partial_{\xi}\delta\theta+\frac{\delta\xi_{*}-\int_{\eta}^{1}\delta\psi'(\tau)\dif\tau}{L-\psi(\eta)}\cdot\frac{\sin\theta}{\rho q}\partial_{\xi}\delta\theta\dif\xi\dif\eta}\nonumber \\
\leq & \mco(1)\norm{\delta\psi'}_{L^{\infty}}\norm{\delta\theta}_{W_{\beta}^{1}(\Omega)}+\mco(1)\norm{\delta\psi'}_{L^{\infty}}\norm{\delta U}_{W_{\beta}^{1}(\Omega)}^{2}+\mco(1)\norm{\delta\theta}_{W_{\beta}^{1}(\Omega)}^{2}\delta\xi_{*}\nonumber \\
\leq & \mco(1)\sigma^{2}+\mco(1)\sigma^{2}\cdot\delta\xi_{*},\label{eq:iteration_mci_f-2}
\end{align}
where $\mco(1)$ depends on $\overline{\xi}_{*}$, $\overline{U}_{+}$,
and $C_{+}$. Moreover,
\begin{align}
 & \abs{\int_{\Omega}-\frac{\int_{\eta}^{1}\delta\psi'(\tau)\dif\tau}{L-\psi(\eta)}\cdot\frac{\cos\theta}{\rho q}\cdot\frac{1-M^{2}}{\rho q^{2}}\partial_{\xi}\delta p\dif\xi\dif\eta}\nonumber \\
\leq & \mco(1)\norm{\delta\psi'}_{L^{\infty}}\norm{\delta U}_{W_{\beta}^{1}(\Omega)}=\mco(1)\sigma^{2},\label{eq:iteration_mci_f-3}
\end{align}
where $\mco(1)$ depends on $\overline{\xi}_{*}$, $\overline{U}_{+}$,
and $C_{+}$. Finally, since
\begin{align*}
 & \frac{\delta\xi_{*}}{L-\psi(\eta)}\cdot\frac{\cos\theta}{\rho q}\cdot\frac{1-M^{2}}{\rho q^{2}}\partial_{\xi}\delta p\\
= & \frac{\delta\xi_{*}}{L-\psi(\eta)}\cdot\left(\frac{\cos\theta}{\rho q}\cdot\frac{1-M^{2}}{\rho q^{2}}-\frac{1}{\overline{\rho}_{+}\overline{q}_{+}}\cdot\frac{1-\overline{M}_{+}^{2}}{\overline{\rho}_{+}\overline{q}_{+}^{2}}\right)\partial_{\xi}\delta p\\
 & +\frac{\delta\xi_{*}}{L-\psi(\eta)}\cdot\frac{1}{\overline{\rho}_{+}\overline{q}_{+}}\cdot\frac{1-\overline{M}_{+}^{2}}{\overline{\rho}_{+}\overline{q}_{+}^{2}}\partial_{\xi}\left(\delta p-\dot{p}_{+}\right)\\
 & -\frac{\delta\xi_{*}\cdot\int_{\eta}^{1}\delta\psi'(\tau)\dif\tau}{\left(L-\psi(\eta)\right)\left(L-\xi_{*}\right)}\cdot\frac{1}{\overline{\rho}_{+}\overline{q}_{+}}\cdot\frac{1-\overline{M}_{+}^{2}}{\overline{\rho}_{+}\overline{q}_{+}^{2}}\partial_{\xi}\dot{p}_{+}\\
 & +\frac{\delta\xi_{*}}{L-\xi_{*}}\cdot\frac{1}{\overline{\rho}_{+}\overline{q}_{+}}\cdot\frac{1-\overline{M}_{+}^{2}}{\overline{\rho}_{+}\overline{q}_{+}^{2}}\partial_{\xi}\dot{p}_{+},
\end{align*}
it then follows from \eqref{eq:first_approx_sub_2} that, 
\[
\int_{\Omega}\frac{1}{\overline{\rho}_{+}\overline{q}_{+}}\cdot\frac{1-\overline{M}_{+}^{2}}{\overline{\rho}_{+}\overline{q}_{+}^{2}}\partial_{\xi}\dot{p}_{+}\dif\xi\dif\eta=\int_{\overline{\xi}_{*}}^{L}\int_{0}^{1}\partial_{\eta}\dot{\theta}_{+}\dif\eta\dif\xi=\sigma\int_{\overline{\xi}_{*}}^{L}\Theta(\xi)\dif\xi,
\]
thus
\begin{align}
 & \int_{\Omega}\frac{\delta\xi_{*}}{L-\psi(\eta)}\cdot\frac{\cos\theta}{\rho q}\cdot\frac{1-M^{2}}{\rho q^{2}}\partial_{\xi}\delta p\dif\xi\dif\eta\nonumber \\
= & \mco(1)\norm{\delta U}_{L^{\infty}}\norm{\delta p}_{W_{\beta}^{1}(\Omega)}+\mco(1)\norm{\delta p-\dot{p}_{+}}_{W_{\beta}^{1}(\Omega)}\cdot\delta\xi_{*}\nonumber \\
 & +\mco(1)\norm{\delta\psi'}_{L^{\infty}}\norm{\dot{p}_{+}}_{W_{\beta}^{1}(\Omega)}\cdot\delta\xi_{*}+\frac{\delta\xi_{*}}{L-\xi_{*}}\cdot\sigma\int_{\overline{\xi}_{*}}^{L}\Theta(\xi)\dif\xi\nonumber \\
= & \frac{\delta\xi_{*}}{L-\xi_{*}}\cdot\sigma\int_{\overline{\xi}_{*}}^{L}\Theta(\xi)\dif\xi+\mco(1)\sigma^{\frac{3}{2}}\cdot\delta\xi_{*}+\mco(1)\sigma^{2},\label{eq:iteration_mci_f-4}
\end{align}
where $\mco(1)$ depends on $\overline{\xi}_{*}$, $\beta$, $\overline{U}_{+}$,
and $C_{+}$. Therefore, integrating the estimates \eqref{eq:iteration_mci_f-1}—\eqref{eq:iteration_mci_f-4} yields that
\begin{equation}
\int_{\Omega}f_{2}\left(\delta U;\delta\psi',\delta\xi_{*}\right)\dif\xi\dif\eta=\frac{\delta\xi_{*}}{L-\xi_{*}}\cdot\sigma\int_{\overline{\xi}_{*}}^{L}\Theta(\xi)\dif\xi+\mco(1)\sigma^{\frac{3}{2}}\cdot\delta\xi_{*}+\mco(1)\sigma^{2},\label{eq:iteration_mci_f}
\end{equation}
where $\mco(1)$ depends on $\overline{\xi}_{*}$, $\beta$, $\overline{U}_{+}$,
and $C_{+}$.

Substituting \eqref{eq:iteration_mci_P}, \eqref{eq:iteration_mci_Theta},
\eqref{eq:iteration_mci_g}, and \eqref{eq:iteration_mci_f} into
\eqref{eq:iteration_mci} gives
\begin{align*}
 & \mci(\delta\xi_{*};\delta U,\delta\psi';\delta U_{-})\\
= & -\frac{\delta\xi_{*}}{L-\xi_{*}}\cdot\sigma\int_{\overline{\xi}_{*}}^{L}\Theta(\xi)\dif\xi+\mco(1)\sigma^{\frac{3}{2}}\cdot\delta\xi_{*}+\mco(1)\sigma^{2}\\
 & +\sigma\int_{\overline{\xi}_{*}}^{L}\Theta(\zeta)\dif\zeta+\sigma\int_{\overline{\xi}_{*}+\delta\xi_{*}}^{\overline{\xi}_{*}}\Theta(\zeta)\dif\zeta+\sigma\cdot\frac{\delta\xi_{*}}{L-\xi_{*}}\int_{\overline{\xi}_{*}+\delta\xi_{*}}^{L}\Theta(\zeta)\dif\zeta\\
 & +\frac{1}{\overline{\rho}_{+}\overline{q}_{+}}\cdot\frac{1-\overline{M}_{+}^{2}}{\overline{\rho}_{+}\overline{q}_{+}^{2}}\int_{0}^{1}\dot{g}_{1}^{\sharp}|_{\left(\overline{\xi}_{*}+\delta\xi_{*},\eta\right)}\dif\eta+\mco(1)\sigma^{2}\\
 & -\frac{1}{\overline{\rho}_{+}\overline{q}_{+}}\cdot\frac{1-\overline{M}_{+}^{2}}{\overline{\rho}_{+}\overline{q}_{+}^{2}}\sigma\int_{0}^{1}P(\eta)\dif\eta+\mco(1)\sigma^{2}\\
= & \frac{1}{\overline{\rho}_{+}\overline{q}_{+}}\cdot\frac{1-\overline{M}_{+}^{2}}{\overline{\rho}_{+}\overline{q}_{+}^{2}}\int_{0}^{1}\left(\dot{g}_{1}^{\sharp}|_{\left(\overline{\xi}_{*},\eta\right)}-\sigma P(\eta)\right)\dif\eta+\sigma\int_{\overline{\xi}_{*}}^{L}\Theta(\zeta)\dif\zeta\\
 & +\frac{1}{\overline{\rho}_{+}\overline{q}_{+}}\cdot\frac{1-\overline{M}_{+}^{2}}{\overline{\rho}_{+}\overline{q}_{+}^{2}}\int_{0}^{1}\left(\dot{g}_{1}^{\sharp}|_{\left(\overline{\xi}_{*}+\delta\xi_{*},\eta\right)}-\dot{g}_{1}^{\sharp}|_{\left(\overline{\xi}_{*},\eta\right)}\right)\dif\eta+\sigma\int_{\overline{\xi}_{*}+\delta\xi_{*}}^{\overline{\xi}_{*}}\Theta(\zeta)\dif\zeta\\
 & +\sigma\cdot\frac{\delta\xi_{*}}{L-\xi_{*}}\left(-\int_{\overline{\xi}_{*}}^{L}\Theta(\xi)\dif\xi+\int_{\overline{\xi}_{*}+\delta\xi_{*}}^{L}\Theta(\zeta)\dif\zeta\right)+\mco(1)\sigma^{\frac{3}{2}}\cdot\delta\xi_{*}+\mco(1)\sigma^{2}.
\end{align*}
Then \eqref{eq:first_approx_criterion_exit_pressure} in Lemma \ref{lem:first_approx_criterion}, \eqref{eq:first_approx_super_int_p} and \eqref{eq:first_approx_p}  yield that
\begin{align*}
 & \mci(\delta\xi_{*};\delta U,\delta\psi';\delta U_{-})\\
= & -\sigma\cdot\frac{1}{\overline{\rho}_{+}\overline{q}_{+}}\cdot\frac{1-\overline{M}_{+}^{2}}{\overline{\rho}_{+}\overline{q}_{+}^{2}}\int_{0}^{1}P(\eta)\dif\eta+\sigma\left(1-\dot{K}\right)\int_{0}^{\overline{\xi}_{*}}\Theta(\zeta)\dif\zeta+\sigma\int_{\overline{\xi}_{*}}^{L}\Theta(\zeta)\dif\zeta\\
 & +\sigma\left(1-\dot{K}\right)\left(\int_{0}^{\overline{\xi}_{*}+\delta\xi_{*}}\Theta(\zeta)\dif\zeta-\int_{0}^{\overline{\xi}_{*}}\Theta(\zeta)\dif\zeta\right)+\sigma\int_{\overline{\xi}_{*}+\delta\xi_{*}}^{\overline{\xi}_{*}}\Theta(\zeta)\dif\zeta\\
 & +\sigma\cdot\frac{\delta\xi_{*}}{L-\xi_{*}}\int_{\overline{\xi}_{*}+\delta\xi_{*}}^{\overline{\xi}_{*}}\Theta(\zeta)\dif\zeta+\mco(1)\sigma^{\frac{3}{2}}\cdot\delta\xi_{*}+\mco(1)\sigma^{2}\\
= & -\sigma\dot{K}\int_{\overline{\xi}_{*}}^{\overline{\xi}_{*}+\delta\xi_{*}}\Theta(\zeta)\dif\zeta+\sigma\cdot\frac{\delta\xi_{*}}{L-\xi_{*}}\int_{\overline{\xi}_{*}+\delta\xi_{*}}^{\overline{\xi}_{*}}\Theta(\zeta)\dif\zeta+\mco(1)\sigma^{\frac{3}{2}}\cdot\delta\xi_{*}+\mco(1)\sigma^{2}\\
= & \left(-\sigma\dot{K}\Theta(\overline{\xi}_{*})\delta\xi_{*}+\mco(1)\sigma\cdot\delta\xi_{*}^{2}\right)+\sigma\cdot\frac{\delta\xi_{*}}{L-\xi_{*}}\left(-\Theta(\overline{\xi}_{*})\delta\xi_{*}+\mco(1)\delta\xi_{*}^{2}\right)\\
 & +\mco(1)\sigma^{\frac{3}{2}}\cdot\delta\xi_{*}+\mco(1)\sigma^{2}\\
= & \left(-\sigma\dot{K}\Theta(\overline{\xi}_{*})+\mco(1)\sigma^{\frac{3}{2}}\right)\delta\xi_{*}+\mco(1)\sigma\cdot\delta\xi_{*}^{2}+\mco(1)\sigma^{2},
\end{align*}
that is,
\begin{equation}
\mci(\delta\xi_{*};\delta U,\delta\psi';\delta U_{-})=\left(-\sigma\dot{K}\Theta(\overline{\xi}_{*})+\mco(1)\sigma^{\frac{3}{2}}\right)\delta\xi_{*}+\mco(1)\sigma\cdot\delta\xi_{*}^{2}+\mco(1)\sigma^{2},\label{eq:iteration_mci_expansion}
\end{equation}
where $\mco(1)$ depends on $\overline{\xi}_{*}$, $\beta$, $\overline{U}_{+}$,
$\hat{C}_{L}$, $C_{L}$ , $C_{+}$, and $\norm{P'}_{L^{\infty}(\Gamma_{3})}$.

The expansion \eqref{eq:iteration_mci_expansion} of $\mci(\delta\xi_{*};\delta U,\delta\psi';\delta U_{-})$
indicates that
\begin{equation}
\frac{\partial\mci}{\partial(\delta\xi_{*})}(0;0,0;\dot{U}_{-})=-\sigma\dot{K}\Theta(\overline{\xi}_{*})+\mco(1)\sigma^{\frac{3}{2}},\label{eq:iteration_approx_functional_derivative}
\end{equation}
and \eqref{eq:iteration_approx_location_update_cond} holds as long
as $\Theta(\overline{\xi}_{*})\not=0$ and $\sigma$ is sufficiently small.
Then the implicit function theorem implies the existence of the solution
$\delta\xi_{*}$ to the equation \eqref{eq:linearized_solvability_cond}.
And the expansion \eqref{eq:iteration_mci_expansion} further yields
that the solution $\delta\xi_{*}$ satisfies the estimate \eqref{eq:iteration_approx_location_estimate}.
Thus, we complete the proof of the lemma.
\end{proof}
Lemma \ref{lem:iteration_approx_location}, together with Theorem
\ref{thm:linearized_problem}, shows that there exists a solution
$(\delta U^{*};\ (\delta\psi^{*})'|\ \delta\xi_{*})$ to the linearized
problem \eqref{eq:linearized_system_sub_elliptic_1}-\eqref{eq:linearized_bc_RH_2}
as $\left(\delta U-\dot{U}_{+};\delta\psi'-\dot{\psi}'\right)\in\fk(\frac{1}{2}\sigma^{3/2})$
with sufficiently small $\sigma$. Moreover, it could be shown that the solution $\left(\delta U^{*};(\delta\psi^{*})'\right)$
satisfies that $\left(\delta U^{*}-\dot{U}_{+};(\delta\psi^{*})'-\dot{\psi}'\right)\in\fk(\frac{1}{2}\sigma^{3/2})$
as long as $\sigma$ is sufficiently small. That is, the iteration
mapping
\[
\mcj_{s}:\left(\delta U;\delta\psi'\right)\mapsto\left(\delta U^{*};(\delta\psi^{*})'\right)
\]
is well-defined in 
\[
\fk_{\sigma}(\dot{U}_{+};\dot{\psi}')\defs\set{\left(\delta U;\delta\psi'\right):\ \left(\delta U-\dot{U}_{+};\delta\psi'-\dot{\psi}'\right)\in\fk(\frac{1}{2}\sigma^{3/2})}
\]
for any sufficiently small $\sigma$. Indeed, the following lemma holds.
\begin{lem}
\label{lem:iteration_mapping_well-defined}

There exists a positive constant $\sigma_{2}\ll1$, such that for
any $0<\sigma\leq\sigma_{2}$, the mapping $\mcj_{s}$ is well-defined
in $\fk_{\sigma}(\dot{U}_{+};\dot{\psi}')$.
\end{lem}
\begin{proof}
It suffices to verify that $\left(\delta U^{*};(\delta\psi^{*})'\right)\in\fk_{\sigma}(\dot{U}_{+};\dot{\psi}')$
when $\sigma$ is small.

Since
\begin{eqnarray*}
\left(\delta U^{*};(\delta\psi^{*})'|\ \delta\xi_{*}\right) & = & \mcp_{a}\left(\mcf;\mcg;\delta P_{3};\delta\Theta_{4}\right),\\
\left(\dot{U}_{+};\dot{\psi}'|\ 0\right) & = & \mcp_{a}\left(\dot{\mcf};\dot{\mcg};\dot{P}_{e};\dot{\Theta}_{N}\right),
\end{eqnarray*}
where $\mcf=\mcf(\delta U;\delta\psi',\delta\xi_{*})\defs(f_{1}(\delta U;\delta\psi',\delta\xi_{*}),\ f_{2}(\delta U;\delta\psi',\delta\xi_{*}),\ f_{3}(\delta U))$,
$\mcg=\mcg(\delta U,\delta U_{-};\delta\psi',\delta\xi_{*})\defs(g_{j}(\delta U,\delta U_{-};\delta\psi',\delta\xi_{*});\ j=1,2,3,4)$,
$\dot{\mcf}\defs\left(0,0,-\Phi(\overline{U}_{+})\right)$, and $\dot{\mcg}\defs\left(\dot{g}_{1},\dot{g}_{2},\dot{g}_{3},\dot{g}_{4}\right)$,
so it holds that
\begin{equation}
\left(\delta U^{*}-\dot{U}_{+};(\delta\psi^{*})'-\dot{\psi}'\right)=\mcp\left(\mcf-\dot{\mcf};\mcg-\dot{\mcg};\delta P_{3}-\dot{P}_{e};\delta\Theta_{4}-\dot{\Theta}_{N}\right).\label{eq:iteration_diff_approx_problem}
\end{equation}
Then, it follows from Theorem \ref{thm:linearized_problem} that
\begin{align}
 & \norm{\delta U^{*}-\dot{U}_{+}}_{(\Omega;\Gamma_{\mr{s}})}+\norm{(\delta\psi^{*})'-\dot{\psi}'}_{W^{1-\frac{1}{\beta},\beta}(\Gamma_{\mr{s}})}\label{eq:iteration_diff_approx_estimates}\\
\leq & C\left(\sum_{j=1}^{2}\norm{f_{j}}_{L^{\beta}(\Omega)}+\norm{f_{3}+\Phi(\overline{U}_{+})}_{\mcc^{0}(\Omega)}+\sum_{j=1}^{4}\norm{g_{j}-\dot{g}_{j}}_{W^{1-\frac{1}{\beta},\beta}(\Gamma_{\mr{s}})}\right)\nonumber \\
 & +C\left(\norm{\delta P_{3}-\dot{P}_{e}}_{W^{1-\frac{1}{\beta},\beta}(\Gamma_{3})}+\norm{\delta\Theta_{4}-\dot{\Theta}_{N}}_{W^{1-\frac{1}{\beta},\beta}(\Gamma_{4})}\right).\nonumber 
\end{align}
It remains to estimate all the terms on the right hand side above.

Since for any $\beta>2$, $W_{\beta}^{1}(\Omega)\subset L^{\infty}(\Omega)$,
and $W^{1-\frac{1}{\beta},\beta}(\Gamma_{j})\subset L^{\infty}(\Gamma_{j})$,
then 
\begin{align*}
 & \norm{f_{1}(\delta U;\delta\psi',\delta\xi_{*})}_{L^{\beta}(\Omega)}\\
\leq & \norm{\left(\overline{q}_{+}-q\cos\theta\right)\partial_{\xi}\delta\theta+\frac{\sin\delta\theta}{\rho q}\partial_{\xi}\delta p}_{L^{\beta}(\Omega)}+\norm{\frac{L-\xi}{L-\psi(\eta)}\cdot\delta\psi'(\eta)\partial_{\xi}\delta p}_{L^{\beta}(\Omega)}\\
 & +\norm{\frac{\delta\xi_{*}-\int_{\eta}^{1}\delta\psi'(\tau)\dif\tau}{L-\psi(\eta)}\left(\frac{\sin\delta\theta}{\rho q}\partial_{\xi}\delta p-q\cos\delta\theta\partial_{\xi}\delta\theta\right)}_{L^{\beta}(\Omega)}\\
\leq & C\norm{\delta U}_{L^{\infty}(\Omega)}\norm{(\delta\theta,\delta p)}_{W_{\beta}^{1}(\Omega)}+\norm{\delta\psi'}_{L^{\infty}(\Gamma_{\mr{s}})}\norm{\delta p}_{W_{\beta}^{1}(\Omega)}\\
 & +C(\overline{\xi}_{*})\abs{\delta\xi_{*}}\cdot\norm{(\delta\theta,\delta p)}_{W_{\beta}^{1}(\Omega)}+C(\overline{\xi}_{*})\norm{\delta\psi'}_{L^{\infty}(\Gamma_{\mr{s}})}\norm{(\delta\theta,\delta p)}_{W_{\beta}^{1}(\Omega)}\\
\leq & C\cdot\left(C_{+}\sigma\right)^{2}+\left(C_{+}\sigma\right)^{2}+C(\overline{\xi}_{*})\cdot C_{*}\sigma\cdot C_{+}\sigma+C(\overline{\xi}_{*})\cdot\left(C_{+}\sigma\right)^{2}\\
\leq & C_{1}\sigma^{2},
\end{align*}
where the constant $C_{1}$ depends on $\overline{\xi}_{*}$, $\overline{U}_{+}$,
and $C_{*}$.

Similarly, one has
\[
\norm{f_{2}(\delta U;\delta\psi',\delta\xi_{*})}_{L^{\beta}(\Omega)}\leq C_{2}\sigma^{2}.
\]

Next, since
\begin{align*}
 & f_{3}(\delta U)+\Phi(\overline{U}_{+})\\
= & \Phi(\overline{U}_{+})+(\overline{q}_{+}\delta q+\frac{1}{\overline{\rho}_{+}}\delta p+\overline{T}_{+}\delta S)-\Phi\left(U\right)\\
= & -\int_{0}^{1}D_{U}^{2}\Phi(\overline{U}_{+}+t\delta U)\dif t\cdot(\delta U)^{2},
\end{align*}
it holds that 
\[
\norm{f_{3}+\Phi(\overline{U}_{+})}_{\mcc^{0}(\Omega)}\leq C(C_{+}\sigma)^{2}\defs C_{3}\sigma^{2}.
\]

Moreover, in the proof of Lemma \ref{lem:iteration_approx_location},
it has been already shown that, for $j=1,2,3$, 
\[
g_{j}(\delta U,\delta U_{-};\delta\psi',\delta\xi_{*})-\dot{g}_{j}=-\beta_{j}^{-}\cdot\left(\dot{U}_{-}|_{\left(\overline{\xi}_{*}+\delta\xi_{*},\eta\right)}-\dot{U}_{-}|_{\left(\overline{\xi}_{*},\eta\right)}\right)+\mco(1)\sigma^{2},
\]
hence, 
\begin{align*}
 & \norm{g_{j}(\delta U,\delta U_{-};\delta\psi',\delta\xi_{*})-\dot{g}_{j}}_{W^{1-\frac{1}{\beta},\beta}(\Gamma_{\mr{s}})}\\
\leq & C\cdot\left(C_{*}\sigma\right)\norm{\partial_{\xi}\dot{U}_{-}}_{\mcc^{1,\alpha}(\Omega_{L})}+\mco(1)\sigma^{2}\\
\leq & C_{j}^{b}\sigma^{2}.
\end{align*}
Analogous computations show also that
\[
\norm{g_{4}(\delta U,\delta U_{-};\delta\psi',\delta\xi_{*})-\dot{g}_{4}}_{W^{1-\frac{1}{\beta},\beta}(\Gamma_{\mr{s}})}\leq C_{4}^{b}\sigma^{2}.
\]

Furthermore, the proof of Lemma \ref{lem:iteration_approx_location} gives that
\[
\delta P_{3}(\eta;\delta U)=\sigma P(\eta)+\mco(1)\sigma^{2}=\dot{P}_{e}+\mco(1)\sigma^{2}.
\]
Thus,
\[
\norm{\delta P_{3}-\dot{P}_{e}}_{W^{1-\frac{1}{\beta},\beta}(\Gamma_{3})}\leq\mco(1)\sigma^{2}.
\]

Finally, since
\begin{align*}
\delta\Theta_{4}-\dot{\Theta}_{N}= & \sigma\Theta(\xi+\frac{L-\xi}{L-\overline{\xi}_{*}}\cdot\delta\xi_{*})-\sigma\Theta(\xi)\\
= & \sigma\int_{0}^{1}\Theta'(\xi+s\frac{L-\xi}{L-\overline{\xi}_{*}}\cdot\delta\xi_{*})\dif s\cdot\frac{L-\xi}{L-\overline{\xi}_{*}}\cdot\delta\xi_{*},
\end{align*}
it holds that 
\[
\norm{\delta\Theta_{4}-\dot{\Theta}_{N}}_{W^{1-\frac{1}{\beta},\beta}(\Gamma_{4})}\leq\sigma\cdot\norm{\Theta'}_{\mcc^{1,\alpha}(\Gamma_{4})}\cdot(C_{*}\sigma)\leq C_{*}\sigma^{2}.
\]

Collecting the above computations and estimates, we conclude that there
exists a constant $C_{\mcj}$, such that for sufficiently small $\sigma$,
\begin{equation}
\norm{\delta U^{*}-\dot{U}_{+}}_{(\Omega;\Gamma_{\mr{s}})}+\norm{(\delta\psi^{*})'-\dot{\psi}'}_{W^{1-\frac{1}{\beta},\beta}(\Gamma_{\mr{s}})}\leq C_{\mcj}\sigma^{2}\leq\frac{1}{2}\sigma^{\frac{3}{2}}.\label{eq:iteration_well-posedness_estimate}
\end{equation}
That is, $\left(\delta U^{*};(\delta\psi^{*})'\right)\in\fk_{\sigma}(\dot{U}_{+};\dot{\psi}')$
and the proof is complete.
\end{proof}
Then, the\textcolor{black}{{} Theorem \ref{thm:main_thm_Lagrange} will
be proved as long as the mapping }$\mcj_{s}$ is
contractive, which is the conclusion of the following lemma.
\begin{lem}
\label{lem:iteration_mapping_contraction}

There exists a positive constant $\sigma_{3}\ll1$, such that for
any $0<\sigma\leq\sigma_{3}$, the mapping $\mcj_{s}$ is contractive
in $\fk_{\sigma}(\dot{U}_{+};\dot{\psi}')$.
\end{lem}
\begin{proof}
Assume that $\left(\delta U_{k};\delta\psi_{k}'\right)\in\fk_{\sigma}(\dot{U}_{+};\dot{\psi}')$,
$k=1,2$, then, by Lemma \ref{lem:iteration_approx_location} and
Lemma \ref{lem:iteration_mapping_well-defined}, there exist $\delta\xi_{k*}$
satisfying the estimate \eqref{eq:iteration_approx_location_estimate}
and $\left(\delta U_{k}^{*};\delta\psi_{k}^{*'}\right)\in\fk_{\sigma}(\dot{U}_{+};\dot{\psi}')$,
such that 
\begin{eqnarray}
\left(\delta U_{k}^{*};\delta\psi_{k}^{*'}|\ \delta\xi_{k*}\right) & = & \mcp_{a}\left(\mcf_{k};\mcg_{k};\delta P_{3}(\eta;\delta U_{k});\delta\Theta_{4}(\xi;\delta\xi_{k*})\right),\qquad k=1,2,\label{eq:iteration_contractive_problems}
\end{eqnarray}
where $\mcf_{k}=\mcf(\delta U_{k};\delta\psi_{k}',\delta\xi_{k*})\defs(f_{1}(\delta U_{k};\delta\psi_{k}',\delta\xi_{k*}),\ f_{2}(\delta U_{k};\delta\psi_{k}',\delta\xi_{k*}),\ f_{3}(\delta U_{k}))$,
and $\mcg_{k}=\mcg(\delta U_{k},\delta U_{-};\delta\psi_{k}',\delta\xi_{k*})\defs(g_{j}(\delta U_{k},\delta U_{-};\delta\psi_{k}',\delta\xi_{k*});\ j=1,2,3,4)$.
Then, to show that the mapping $\mcj_{s}$ is contractive, it suffices
to verify that, for sufficiently small $\sigma>0$,
\begin{equation}
\begin{split}
&\norm{\delta U_{2}^{*}-\delta U_{1}^{*}}_{(\Omega;\Gamma_{\mr{s}})}+\norm{\delta\psi_{2}^{*'}-\delta\psi_{1}^{*'}}_{W^{1-\frac{1}{\beta},\beta}(\Gamma_{\mr{s}})} \\
\leq&\frac{1}{2}\left(\norm{\delta U_{2}-\delta U_{1}}_{(\Omega;\Gamma_{\mr{s}})}+\norm{\delta\psi_{2}'-\delta\psi_{1}'}_{W^{1-\frac{1}{\beta},\beta}(\Gamma_{\mr{s}})}\right).\label{eq:iteration_contraction_estimate_toprove}
\end{split}
\end{equation}
It follows from \eqref{eq:iteration_contractive_problems} that
\begin{align}
 & \left(\delta U_{2}^{*}-\delta U_{1}^{*};\delta\psi_{2}^{*'}-\delta\psi_{1}^{*'}\right)\label{eq:iteration_contractive_difference_problem}\\
= & \mcp\left(\mcf_{2}-\mcf_{1};\mcg_{2}-\mcg_{1};\delta P_{3}(\eta;\delta U_{2})-\delta P_{3}(\eta;\delta U_{1});\delta\Theta_{4}(\xi;\delta\xi_{2*})-\delta\Theta_{4}(\xi;\delta\xi_{1*})\right).\nonumber
\end{align}
Then one needs to estimate the terms in the right hand side. Since the right hand side of \eqref{eq:iteration_contractive_difference_problem}
involves $\delta\xi_{k*}$, which is determined by \eqref{eq:linearized_solvability_cond}
with given $\left(\delta U_{k};\delta\psi_{k}'\right)\in\fk_{\sigma}(\dot{U}_{+};\dot{\psi}')$,
one has to estimate $\abs{\delta\xi_{2*}-\delta\xi_{1*}}$ first.

It follows from \eqref{eq:linearized_solvability_cond} that
\begin{align*}
0= & \mci(\delta\xi_{2*};\delta U_{2},\delta\psi_{2}';\delta U_{-})-\mci(\delta\xi_{1*};\delta U_{1},\delta\psi_{1}';\delta U_{-})\\
= & \mci(\delta\xi_{2*};\delta U_{2},\delta\psi_{2}';\delta U_{-})-\mci(\delta\xi_{1*};\delta U_{2},\delta\psi_{2}';\delta U_{-})\\
 & +\mci(\delta\xi_{1*};\delta U_{2},\delta\psi_{2}';\delta U_{-})-\mci(\delta\xi_{1*};\delta U_{1},\delta\psi_{1}';\delta U_{-})\\
= & \int_{0}^{1}\frac{\partial\mci}{\partial(\delta\xi_{*})}(\delta\xi_{t*};\delta U_{2},\delta\psi_{2}';\delta U_{-})\dif t\cdot(\delta\xi_{2*}-\delta\xi_{1*})\\
 & +\int_{0}^{1}\nabla_{(\delta U,\delta\psi')}\mci(\delta\xi_{1*};\delta U_{t},\delta\psi_{t}';\delta U_{-})\dif t\cdot\left(\delta U_{2}-\delta U_{1};\delta\psi_{2}'-\delta\psi_{1}'\right),
\end{align*}
where $\delta\xi_{t*}\defs t\delta\xi_{2*}+(1-t)\delta\xi_{1*}$,
$\delta U_{t}\defs t\delta U_{2}+(1-t)\delta U_{1}$, and $\delta\psi_{t}'\defs t\delta\psi_{2}'+(1-t)\psi_{1}'$.
Then similar computations as in Lemma \ref{lem:iteration_approx_location} lead to
\begin{align*}
 & \frac{\partial\mci}{\partial(\delta\xi_{*})}(\delta\xi_{t*};\delta U_{2},\delta\psi_{2}';\delta U_{-})\\
= & \frac{\partial\mci}{\partial(\delta\xi_{*})}(0;0,0;\dot{U}_{-})+\int_{0}^{1}\nabla_{(\delta\xi_{*};\delta U,\delta\psi';\delta U_{-})}\frac{\partial\mci}{\partial(\delta\xi_{*})}(s\delta\xi_{t*};s\delta U_{2},s\delta\psi_{2}';\delta U_{-}^{s})\dif s\\
 & \qquad\qquad\qquad\qquad\qquad\qquad\qquad\cdot(\delta\xi_{t*};\delta U_{2},\delta\psi_{2}';\delta U_{-}-\dot{U}_{-})\\
= & \left(-\sigma\dot{K}\Theta(\overline{\xi}_{*})+\mco(1)\sigma^{\frac{3}{2}}\right)+\mco(1)\sigma^{2},
\end{align*}
and that
\[
\nabla_{(\delta U,\delta\psi')}\mci(\delta\xi_{1*};\delta U_{t},\delta\psi_{t}';\delta U_{-})=\mco(1)\sigma,
\]
where $\mco(1)$ is a constant depending on the parameters of the
background shock solution.

Therefore, as $\sigma$ is sufficiently small depending on $\dot{K}\Theta(\overline{\xi}_{*})$
and other parameters of the background solution, it holds that
\begin{equation}
\abs{\delta\xi_{2*}-\delta\xi_{1*}}\leq\mco(1)(\norm{\delta U_{2}-\delta U_{1}}_{(\Omega;\Gamma_{\mr{s}})}+\norm{\delta\psi_{2}'-\delta\psi_{1}'}_{W^{1-\frac{1}{\beta},\beta}(\Gamma_{\mr{s}})}),\label{eq:iteration_approx_location_diff}
\end{equation}
where $\mco(1)$ is a constant depending on $\dot{K}\Theta(\overline{\xi}_{*})$
and other parameters of the background solution.

Then, applying Theorem \ref{thm:linearized_problem} to the problem
\eqref{eq:iteration_contractive_difference_problem}, and by analogous
computations as in Lemma \ref{lem:iteration_mapping_well-defined},
with help of the estimate \eqref{eq:iteration_approx_location_diff},
we can show that the inequality \eqref{eq:iteration_contraction_estimate_toprove}
holds for sufficiently small $\sigma$, which shows the validity
of the lemma.
\end{proof}
\newpage
\appendix

\section{A Boundary Value Problem for Linear Elliptic Systems of First Order
with Constant Coefficients}

In this section, we are going to establish the well-posedness theory for
boundary value problems of elliptic systems of first order with a
particular form. It should be remarked that the notations used in this section
are independent and have no relations to the ones in other parts of
the paper.

\begin{figure}[th]
\centering
\def\svgwidth{280pt}
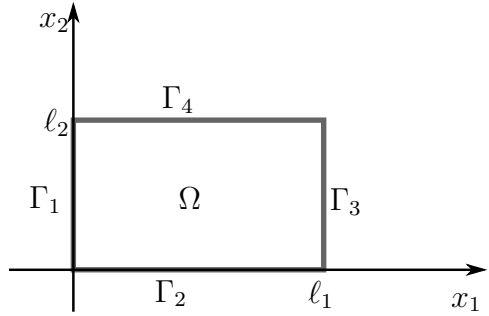

\caption{The domain of the boundary value problem.\label{fig:domain}}
\end{figure}

Let $\ell_{1}$, $\ell_{2}$ be two positive constants, and 
\[
\Omega:=\left\{ x=\left(x_{1},x_{2}\right)\in\Real^{2}|\ 0<x_{1}<\ell_{1},\ 0<x_{2}<\ell_{2}\right\} 
\]
 be a rectangle with the boundaries
\begin{eqnarray*}
\Gamma_{1} & = & \left\{ x_{1}=0,\ 0<x_{2}<\ell_{2}\right\} ,\\
\Gamma_{2} & = & \left\{ 0<x_{1}<\ell_{1},\ x_{2}=0\right\} ,\\
\Gamma_{3} & = & \left\{ x_{1}=\ell_{1},\ 0<x_{2}<\ell_{2}\right\} ,\\
\Gamma_{4} & = & \left\{ 0<x_{1}<\ell_{1},\ x_{2}=\ell_{2}\right\} .
\end{eqnarray*}

\subsection{Cauchy-Riemann equations}

Consider the following boundary value problem: 
\begin{align}
 & \partial_{x_{2}}v_{1}+\partial_{x_{1}}v_{2}=f_{1}\left(x\right), & \text{ in } & \Omega\label{eq:cauchy_riemann_eq_1}\\
 & \partial_{x_{2}}v_{2}-\partial_{x_{1}}v_{1}=f_{2}\left(x\right), & \text{ in } & \Omega\label{eq:cauchy_riemann_eq_2}\\
 & v_{1}=g_{1}\left(x_{2}\right), & \text{ on } & \Gamma_{1}\label{eq:cauchy_riemann_bc_1}\\
 & v_{2}=g_{2}\left(x_{1}\right), & \text{ on } & \Gamma_{2}\label{eq:cauchy_riemann_bc_2}\\
 & v_{1}=g_{3}\left(x_{2}\right), & \text{ on } & \Gamma_{3}\label{eq:cauchy_riemann_bc_3}\\
 & v_{2}=g_{4}\left(x_{1}\right). & \text{ on } & \Gamma_{4}\label{eq:cauchy_riemann_bc_4}
\end{align}

\begin{lem}
\label{lem:cauchy_riemann}

Let $q>2$. Assume that $f_{j}\in L^{q}\left(\Omega\right)\left(j=1,2\right)$,
and $g_{j}\in W_{q}^{1-1/q}\left(\Gamma_{j}\right)$, $\left(j=1,2,3,4\right)$.
Then there exists a unique solution $v\defs\left(v_{1},v_{2}\right)^{\top}\in\left(W_{q}^{1}\left(\Omega\right)\right)^{2}$
to the boundary value problem \eqref{eq:cauchy_riemann_eq_1}—\eqref{eq:cauchy_riemann_bc_4},
provided that
\begin{equation}
\int_{\Omega}f_{2}\left(x\right)\dif x=\int_{0}^{\ell_{1}}\left(g_{4}\left(x_{1}\right)-g_{2}\left(x_{1}\right)\right)\dif x_{1}+\int_{0}^{\ell_{2}}\left(g_{1}\left(x_{2}\right)-g_{3}\left(x_{2}\right)\right)\dif x_{2}.\label{eq:cauchy_riemann_exist_cond}
\end{equation}
Moreover, if \eqref{eq:cauchy_riemann_exist_cond} holds, then
the following estimate is valid:
\begin{equation}
\sum_{j=1}^{2}\norm{v_{j}}_{W_{q}^{1}\left(\Omega\right)}\leq C\left(\sum_{j=1}^{2}\norm{f_{j}}_{L^{q}\left(\Omega\right)}+\sum_{j=1}^{4}\norm{g_{j}}_{W_{q}^{1-1/q}\left(\Gamma_{j}\right)}\right),\label{eq:cauchy_riemann_estimate_v}
\end{equation}
where the constant $C$ depends only on $ q $ and $ \Omega $, but is independent of $v_{j}$. 
\end{lem}
\begin{proof}
We decompose the problem \eqref{eq:cauchy_riemann_eq_1}—\eqref{eq:cauchy_riemann_bc_4}
into two boundary value problems with different inhomogeneous terms
as follows. Let $\left(v_{1},v_{2}\right)^{\top}=\left(V_{1},V_{2}\right)^{\top}+\left(W_{1},W_{2}\right)^{\top}$,
where $\left(V_{1},V_{2}\right)^{\top}$ is the solution to the problem
\begin{align}
 & \partial_{x_{2}}V_{1}+\partial_{x_{1}}V_{2}=f_{1}\left(x\right), & \text{ in } & \Omega\label{eq:cauchy_riemann_aug_V_eq_1}\\
 & \partial_{x_{2}}V_{2}-\partial_{x_{1}}V_{1}=0, & \text{ in } & \Omega\label{eq:cauchy_riemann_aug_V_eq_2}\\
 & V_{1}=0, & \text{ on } & \Gamma_{1}\cup\Gamma_{3}\label{eq:cauchy_riemann_aug_V_bc_13}\\
 & V_{2}=0, & \text{ on } & \Gamma_{2}\cup\Gamma_{4}\label{eq:cauchy_riemann_aug_V_bc_24}
\end{align}
while $\left(W_{1},W_{2}\right)^{\top}$ solves the following problem
\begin{align}
 & \partial_{x_{2}}W_{1}+\partial_{x_{1}}W_{2}=0, & \text{ in } & \Omega\label{eq:cauchy_riemann_aug_W_eq_1}\\
 & \partial_{x_{2}}W_{2}-\partial_{x_{1}}W_{1}=f_{2}\left(x\right), & \text{ in } & \Omega\label{eq:cauchy_riemann_aug_W_eq_2}\\
 & W_{1}=g_{1}\left(x_{2}\right), & \text{ on } & \Gamma_{1}\label{eq:cauchy_riemann_aug_W_bc_1}\\
 & W_{2}=g_{2}\left(x_{1}\right), & \text{ on } & \Gamma_{2}\label{eq:cauchy_riemann_aug_W_bc_2}\\
 & W_{1}=g_{3}\left(x_{2}\right), & \text{ on } & \Gamma_{3}\label{eq:cauchy_riemann_aug_W_bc_3}\\
 & W_{2}=g_{4}\left(x_{1}\right). & \text{ on } & \Gamma_{4}\label{eq:cauchy_riemann_aug_W_bc_4}
\end{align}

By \eqref{eq:cauchy_riemann_aug_V_eq_2}, there exists a function
$\Phi\left(x\right)$ such that $\nabla\Phi\defs\left(\partial_{x_{1}}\Phi,\partial_{x_{2}}\Phi\right)^{\top}=\left(V_{2},V_{1}\right)^{\top}$.
Then the problem \eqref{eq:cauchy_riemann_aug_V_eq_1}—\eqref{eq:cauchy_riemann_aug_V_bc_24}
for $\left(V_{1},V_{2}\right)^{\top}$ is rewritten as
\begin{align}
 & \triangle\Phi=f_{1}\left(x\right), & \text{ in } & \Omega\label{eq:cauchy_riemann_aug_V_eq_1-1}\\
 & \partial_{x_{2}}\Phi=0, & \text{ on } & \Gamma_{1}\cup\Gamma_{3}\label{eq:cauchy_riemann_aug_V_bc_13-1}\\
 & \partial_{x_{1}}\Phi=0. & \text{ on } & \Gamma_{2}\cup\Gamma_{4}\label{eq:cauchy_riemann_aug_V_bc_24-1}
\end{align}
Since the derivatives in the boundary conditions \eqref{eq:cauchy_riemann_aug_V_bc_13-1}
and \eqref{eq:cauchy_riemann_aug_V_bc_24-1} are exactly tangential
derivatives, they can be reformulated as Dirichlet boundary conditions
\begin{equation}
\Phi=0,\qquad\text{on }\partial\Omega\label{eq:cauchy_riemann_aug_V_bc_D}
\end{equation}
if we impose an additional condition $\Phi(0,0)=0$. Then by Corollary
4.4.3.8 ( or Theorem 4.4.3.7, P233-234) in the book \cite[P233-234]{Grisvard1985}
by P. Grisvard, we obtain that there exists a unique solution $\Phi\in W_{q}^{2}\left(\Omega\right)$
to the boundary value problem \eqref{eq:cauchy_riemann_aug_V_eq_1-1}
and \eqref{eq:cauchy_riemann_aug_V_bc_D}, and there exists a constant
$C$ independent of $\Phi$, such that
\begin{equation}
\norm{\Phi}_{W_{q}^{2}\left(\Omega\right)}\leq C\norm{f_{1}}_{L^{q}\left(\Omega\right)}.\label{eq:estimate_phi}
\end{equation}
Hence, there exists a unique solution$\left(V_{1},V_{2}\right)^{\top}$
to the problem \eqref{eq:cauchy_riemann_aug_V_eq_1}—\eqref{eq:cauchy_riemann_aug_V_bc_24},
and the following estimate holds:
\begin{equation}
\sum_{j=1}^{2}\norm{V_{j}}_{W_{q}^{1}\left(\Omega\right)}\leq C\norm{f_{1}}_{L^{q}\left(\Omega\right)}.\label{eq:cauchy_riemann_estimate_aug_V}
\end{equation}

Analogously, by \eqref{eq:cauchy_riemann_aug_W_eq_1}, there exists
a function $\Psi\left(x\right)$ such that $\nabla\Psi\defs\left(\partial_{x_{1}}\Psi,\partial_{x_{2}}\Psi\right)^{\top}=\left(-W_{1},W_{2}\right)^{\top}$.
Thus, the problem \eqref{eq:cauchy_riemann_aug_W_eq_1}—\eqref{eq:cauchy_riemann_aug_W_bc_4}
for $\left(W_{1},W_{2}\right)^{\top}$ is rewritten as 
\begin{align}
 & \Delta\Psi=f_{2}\left(x\right), & \text{ in } & \Omega\label{eq:cauchy_riemann_aug_W_eq}\\
 & -\partial_{x_{1}}\Psi=g_{1}\left(x_{2}\right), & \text{ on } & \Gamma_{1}\label{eq:cauchy_riemann_aug_W_bc_1-1}\\
 & \partial_{x_{2}}\Psi=g_{2}\left(x_{1}\right), & \text{ on } & \Gamma_{2}\label{eq:cauchy_riemann_aug_W_bc_2-1}\\
 & -\partial_{x_{1}}\Psi=g_{3}\left(x_{2}\right), & \text{ on } & \Gamma_{3}\label{eq:cauchy_riemann_aug_W_bc_3-1}\\
 & \partial_{x_{2}}\Psi=g_{4}\left(x_{1}\right). & \text{ on } & \Gamma_{4}\label{eq:cauchy_riemann_aug_W_bc_4-1}
\end{align}
Then by Corollary 4.4.3.8 in the book \cite[P233-234]{Grisvard1985},
there exists a unique solution $\Psi\in W_{q}^{2}\left(\Omega\right)$,
up to an additive constant, to the problem \eqref{eq:cauchy_riemann_aug_W_eq_1}—\eqref{eq:cauchy_riemann_aug_W_bc_4},
provided that the solvability condition \eqref{eq:cauchy_riemann_exist_cond}
holds. Furthermore, because it is the first derivatives of $\Psi$
that are really needed, the additive constant is not important and
can be chosen such that $\int_{\Omega}\Psi=0$. Therefore, 
by Poincare inequality, there exists a constant $C_{\left(\Omega\right)}$
such that
\begin{equation}
\norm{\Psi}_{L^{2}\left(\Omega\right)}\leq C_{\left(\Omega\right)}\norm{\nabla\Psi}_{L^{2}\left(\Omega\right)}.\label{eq:PoincareInequality}
\end{equation}
Then by standard energy estimates, that is, multiplying
$\Psi$ on both sides of the equation \eqref{eq:cauchy_riemann_aug_W_eq},
integrate over $\Omega$ and then employing the formula of integration
by parts, one can obtain
\[
\norm{\nabla\Psi}_{L^{2}\left(\Omega\right)}^{2}\leq\norm{f_{2}}_{L^{2}}\cdot\norm{\Psi}_{L^{2}\left(\Omega\right)}+C_{\left(q,\Omega\right)}\max\set{\norm{g_{j}}_{L^{\infty}\left(\Gamma_{j}\right)}\left(j=1,2,3,4\right)}\cdot\norm{\Psi}_{L^{2}\left(\partial\Omega\right)}.
\]
Applying the inequality \eqref{eq:PoincareInequality} yields
immediately that
\begin{equation}
\norm{\Psi}_{H^{1}\left(\Omega\right)}\leq C\norm{\nabla\Psi}_{L^{2}\left(\Omega\right)}\leq C\cdot F_{\Psi},\label{eq:estimate_psi_H1}
\end{equation}
where 
\[
F_{\Psi}\defs\norm{f_{2}}_{L^{q}\left(\Omega\right)}+\sum_{j=1}^{4}\norm{g_{j}}_{W_{q}^{1-\frac{1}{q}}\left(\Gamma_{j}\right)}.
\]
Thus the embedding theorem yields
\[
\norm{\Psi}_{L^{q}\left(\Omega\right)}\leq C\cdot F_{\Psi}.
\]
Then by trace theorems, Theorem 4.3.2.4 and Remark 4.3.2.5 in the
book \cite{Grisvard1985}, one gets that
\begin{equation}
\norm{\Psi}_{W_{q}^{2}\left(\Omega\right)}\leq C\cdot F_{\Psi}.\label{eq:estimate_psi}
\end{equation}
Hence, there exists a unique solution$\left(W_{1},W_{2}\right)^{\top}$
to the problem \eqref{eq:cauchy_riemann_aug_W_eq_1}-\eqref{eq:cauchy_riemann_aug_W_bc_4} satisfying the estimate:
\begin{equation}
\sum_{j=1}^{2}\norm{W_{j}}_{W_{q}^{1}\left(\Omega\right)}\leq C\cdot F_{\Psi}.\label{eq:cauchy_riemann_estimate_aug_W}
\end{equation}

Hence the Lemma holds.
\end{proof}

\subsection{Elliptic systems of first order with constant coefficients}

The above results can be easily extended to general elliptic systems
of first order with constant coefficients. 

Consider the following boundary value problem on $\Omega$:
\begin{align}
 & \partial_{x_{2}}u_{1}+a_{1}\partial_{x_{1}}u_{2}=f_{1}\left(x\right), & \text{ in } & \Omega\label{eq:first_order_elliptic_eq_1}\\
 & \partial_{x_{2}}u_{2}-a_{2}\partial_{x_{1}}u_{1}=f_{2}\left(x\right), & \text{ in } & \Omega\label{eq:first_order_elliptic_eq_2}\\
 & u_{1}=g_{1}\left(x_{2}\right), & \text{ on } & \Gamma_{1}\label{eq:first_order_elliptic_bc_1}\\
 & u_{2}=g_{2}\left(x_{1}\right), & \text{ on } & \Gamma_{2}\label{eq:first_order_elliptic_bc_2}\\
 & u_{1}=g_{3}\left(x_{2}\right), & \text{ on } & \Gamma_{3}\label{eq:first_order_elliptic_bc_3}\\
 & u_{2}=g_{4}\left(x_{1}\right), & \text{ on } & \Gamma_{4}\label{eq:first_order_elliptic_bc_4}
\end{align}
where $a_{1}$ and $a_{2}$ are two constants satisfying $a_{1}a_{2}>0$.

Indeed, under the following transformation of the coordinates and
the unknown functions:
\begin{align}
\left(\ft\right): &  & \begin{cases}
y_{1} & =\displaystyle\frac{1}{\sqrt{a_{1}a_{2}}}x_{1},\\
y_{2} & =x_{2},
\end{cases} &  &  & \begin{cases}
v_{1} & =\sqrt{\displaystyle\frac{a_{2}}{a_{1}}}u_{1},\\
v_{2} & =u_{2},
\end{cases}
\label{eq:transform_elliptic_sys}
\end{align}
the domain $\Omega$ becomes
\[
\Omega^{*}:=\left\{ y=\left(y_{1},y_{2}\right)\in\Real^{2}|\ 0<y_{1}<\ell_{1}^{*},\ 0<y_{2}<\ell_{2}^{*}\right\} ,
\]
where $\ell_{1}^{*}=\displaystyle\frac{1}{\sqrt{a_{1}a_{2}}}\ell_{1}$ and $\ell_{2}^{*}=\ell_{2}$,
and its boundaries $\Gamma_{j}\left(j=1,2,3,4\right)$ become
\begin{eqnarray*}
\Gamma_{1}^{*} & = & \left\{ y_{1}=0,\ 0<y_{2}<\ell_{2}^{*}\right\} ,\\
\Gamma_{2}^{*} & = & \left\{ 0<y_{1}<\ell_{1}^{*},\ y_{2}=0\right\} ,\\
\Gamma_{3}^{*} & = & \left\{ y_{1}=\ell_{1}^{*},\ 0<y_{2}<\ell_{2}^{*}\right\} ,\\
\Gamma_{4}^{*} & = & \left\{ 0<y_{1}<\ell_{1}^{*},\ y_{2}=\ell_{2}^{*}\right\} .
\end{eqnarray*}
Therefore, the above boundary value problem is reformulated as
\begin{align}
 & \partial_{y_{2}}v_{1}+\partial_{y_{1}}v_{2}=\sqrt{\frac{a_{2}}{a_{1}}}f_{1}\left(\sqrt{a_{1}a_{2}}y_{1},y_{2}\right):=f_{1}^{*}\left(y\right), & \text{ in } & \Omega^{*}\label{eq:first_order_elliptic_eq_1-T}\\
 & \partial_{y_{2}}v_{2}-\partial_{y_{1}}v_{1}=f_{2}\left(\sqrt{a_{1}a_{2}}y_{1},y_{2}\right):=f_{2}^{*}\left(y\right), & \text{ in } & \Omega^{*}\label{eq:first_order_elliptic_eq_2-T}\\
 & v_{1}=\sqrt{\frac{a_{2}}{a_{1}}}g_{1}\left(y_{2}\right):=g_{1}^{*}\left(y_{2}\right), & \text{ on } & \Gamma_{1}^{*}\label{eq:first_order_elliptic_bc_1-T}\\
 & v_{2}=g_{2}\left(\sqrt{a_{1}a_{2}}y_{1}\right):=g_{2}^{*}\left(y_{1}\right), & \text{ on } & \Gamma_{2}^{*}\label{eq:first_order_elliptic_bc_2-T}\\
 & v_{1}=\sqrt{\frac{a_{2}}{a_{1}}}g_{3}\left(y_{2}\right):=g_{3}^{*}\left(y_{2}\right), & \text{ on } & \Gamma_{3}^{*}\label{eq:first_order_elliptic_bc_3-T}\\
 & v_{2}=g_{4}\left(\sqrt{a_{1}a_{2}}y_{1}\right):=g_{4}^{*}\left(y_{1}\right). & \text{ on } & \Gamma_{4}^{*}\label{eq:first_order_elliptic_bc_4-T}
\end{align}
Thus one can apply Lemma \ref{lem:cauchy_riemann} to obtain the following
consequence.
\begin{lem}
\label{lem:bvp_first_order_elliptic}

Let $q>2$. Assume that $f_{j}\in L^{q}\left(\Omega\right)\left(j=1,2\right)$,
and $g_{j}\in W_{q}^{1-1/q}\left(\Gamma_{j}\right)$, $\left(j=1,2,3,4\right)$.
Then there exists a unique solution $u\defs\left(u_{1},u_{2}\right)^{\top}\in\left(W_{q}^{1}\left(\Omega\right)\right)^{2}$
to the boundary value problem \eqref{eq:first_order_elliptic_eq_1}—\eqref{eq:first_order_elliptic_bc_4},
provided that
\begin{equation}
\int_{\Omega}f_{2}\left(x\right)\dif x=\int_{0}^{\ell_{1}}\left(g_{4}\left(x_{1}\right)-g_{2}\left(x_{1}\right)\right)\dif x_{1}+a_{2}\int_{0}^{\ell_{2}}\left(g_{1}\left(x_{2}\right)-g_{3}\left(x_{2}\right)\right)\dif x_{2}.\label{eq:first_order_elliptic_exist_cond}
\end{equation}
Furthermore, the solution satisfies
\begin{equation}
\sum_{j=1}^{2}\norm{u_{j}}_{W_{q}^{1}\left(\Omega\right)}\leq C\left(\sum_{j=1}^{2}\norm{f_{j}}_{L^{q}\left(\Omega\right)}+\sum_{j=1}^{4}\norm{g_{j}}_{W_{q}^{1-1/q}\left(\Gamma_{j}\right)}\right),\label{eq:first_order_elliptic_estimate_u}
\end{equation}
where the constant $C$ is independent of the $u_{j}$, but depends
on $q$ and $\Omega$.
\end{lem}
\begin{proof}
It suffices to show that the solvability condition \eqref{eq:cauchy_riemann_exist_cond}
for the problem \eqref{eq:first_order_elliptic_eq_1-T}—\eqref{eq:first_order_elliptic_bc_4-T}
yields the condition \eqref{eq:first_order_elliptic_exist_cond} for
the problem \eqref{eq:first_order_elliptic_eq_1}—\eqref{eq:first_order_elliptic_bc_4}.

The condition \eqref{eq:cauchy_riemann_exist_cond} for the problem
\eqref{eq:first_order_elliptic_eq_1-T}—\eqref{eq:first_order_elliptic_bc_4-T}
reads
\[
\int_{\Omega^{*}}f_{2}^{*}\left(y\right)\dif y=\int_{0}^{\ell_{1}^{*}}\left(g_{4}^{*}\left(y_{1}\right)-g_{2}^{*}\left(y_{1}\right)\right)\dif y_{1}+\int_{0}^{\ell_{2}^{*}}\left(g_{1}^{*}\left(y_{2}\right)-g_{3}^{*}\left(y_{2}\right)\right)\dif y_{2},
\]
that is,
\begin{align*}
 & \int_{\Omega^{*}}f_{2}\left(\sqrt{a_{1}a_{2}}y_{1},y_{2}\right)\dif y\\
= & \int_{0}^{\ell_{1}^{*}}\left(g_{4}\left(\sqrt{a_{1}a_{2}}y_{1}\right)-g_{2}\left(\sqrt{a_{1}a_{2}}y_{1}\right)\right)\dif y_{1}+\sqrt{\frac{a_{2}}{a_{1}}}\int_{0}^{\ell_{2}^{*}}\left(g_{1}\left(y_{2}\right)-g_{3}\left(y_{2}\right)\right)\dif y_{2}.
\end{align*}
Then the transformation in \eqref{eq:transform_elliptic_sys} yields
\begin{align*}
 & \frac{1}{\sqrt{a_{1}a_{2}}}\int_{\Omega}f_{2}\left(x\right)\dif x\\
= & \frac{1}{\sqrt{a_{1}a_{2}}}\int_{0}^{\ell_{1}}\left(g_{4}\left(x_{1}\right)-g_{2}\left(x_{1}\right)\right)\dif x_{1}+\sqrt{\frac{a_{2}}{a_{1}}}\int_{0}^{\ell_{2}}\left(g_{1}\left(x_{2}\right)-g_{3}\left(x_{2}\right)\right)\dif x_{2},
\end{align*}
which is exactly the solvability condition \eqref{eq:first_order_elliptic_exist_cond}.
\end{proof}

\section*{Acknowlegements}

The research of Beixiang Fang is partially supported by Natural Science
Foundation of China under Grant Nos. 11631008, and 11371250,
the Shanghai Committee of Science and Technology (Grant No. 15XD1502300),
and Shanghai Jiao Tong University's Chenxing SMC-B Project. The research
of Zhouping Xin is partially supported by Hong Kong RGC Earmarked Research Grants CUHK-14300917, CUHK-14305315, CUHK-14302917, and Zheng-Ge Zu Foundation.

\end{document}